\documentclass[12pt]{amsart}
\usepackage{amssymb,amsmath,amsfonts,latexsym}
\usepackage{bm}

\newcommand{\newsection}[1]{\setcounter{equation}{0} \section{#1}}
\newcommand{\bea}{\begin{eqnarray}}
\newcommand{\eea}{\end{eqnarray}}

\newcommand{\vp}{\varphi}

\newcommand{\clb}{\mathcal{B}}

\newcommand{\cld}{\mathcal{D}}

\newcommand{\clh}{\mathcal{H}}

\newcommand{\clm}{\mathcal{M}}

\newcommand{\clq}{\mathcal{Q}}

\newcommand{\cls}{\mathcal{S}}

\newcommand{\D}{\mathbb{D}}

\newcommand{\raro}{\rightarrow}

\def\textmatrix#1&#2\\#3&#4\\{\bigl({#1 \atop #3}\ {#2 \atop #4}\bigr)}
\def\dispmatrix#1&#2\\#3&#4\\{\left({#1 \atop #3}\ {#2 \atop #4}\right)}
\newcommand{\be}{\begin{equation}}
\newcommand{\ee}{\end{equation}}
\newcommand{\ben}{\begin{eqnarray*}}
\newcommand{\een}{\end{eqnarray*}}

\newcommand{\bi}{\begin{itemize}}
\newcommand{\ei}{\end{itemize}}

\newcommand\la{{\langle }}
\newcommand\ra{{\rangle}}

\theoremstyle{definition}

\theoremstyle{plain}

\newtheorem{thm}{Theorem}[section]
\newtheorem{cor}[thm]{Corollary}
\newtheorem{lem}[thm]{Lemma}
\newtheorem{prop}[thm]{Proposition}
\theoremstyle{definition}
\newtheorem{defn}[thm]{Definition}
\newtheorem{rem}[thm]{Remark}
\newtheorem{ex}[thm]{Example}

\numberwithin{equation}{section}

\let\phi=\varphi

\begin{document}

\title[Submodules of  $H^2(\D^n)$]{Two problems on submodules of  $H^2(\D^n)$}

\author[Debnath]{Ramlal Debnath}
\address{KTH Royal Institute of Technology, Stockholm.}
\email{ramlaldebnath@gmail.com, ramlal@kth.se}

\author[Sarkar] {Srijan Sarkar }
\address{Indian Institute of Science, Bangalore}
\email{srijans@iisc.ac.in, srijansarkar@gmail.com}



\begin{abstract}
Given any shift-invariant closed subspace $\mathcal{S}$ (aka submodule) of the Hardy space over the unit polydisc $H^2(\mathbb{D}^n)$ (where $n \geq 2$), let $R_{z_j}:=M_{z_j}|_{\mathcal{S}}$, and $E_{z_j}:=P_{\mathcal{S}}\circ ev_{z_j}$, for each $j \in \{1,\ldots,n\}$. Here, $ev_{z_j}$ is the operator evaluating at $0$ in the $z_j$-th variable.  In this article, we prove that given any subset $\Lambda \subseteq \{1,\ldots,n\}$, there exists a collection of one-variable inner functions $\{\phi_\lambda (z_\lambda)\}_{\lambda \in \Lambda}$ on $\mathbb{D}^n$,  such that
\[
\mathcal{S}  =  \sum_{\lambda \in \Lambda} \phi_\lambda (z_\lambda)H^2(\mathbb{D}^n),
\]
if and only if the conditions $
(I_{\mathcal{S}}-E_{z_k}E_{z_k}^*)(I_{\mathcal{S}}-R_{z_k}R_{z_k}^*)=0$ for all $k \in \{1,\dots,n\} \setminus \Lambda$, and 
$(I_{\mathcal{S}}-E_{z_{i}}E_{z_{i}}^*)(I_{\mathcal{S}}-R_{z_{i}}R_{z_{i}}^*)(I_{\mathcal{S}}-E_{z_{j}}E_{z_{j}}^*)(I_{\mathcal{S}}-R_{z_{j}}R_{z_{j}}^*)=0
$ for all distinct $i,j \in \Lambda$, are satisfied. Following this, we study R.G. Douglas's question on the commutativity of orthogonal projections onto the corresponding quotient modules.
\end{abstract}

\maketitle

\tableofcontents

\newsection{Introduction}\label{sec: 1}
A landmark result in the theory of Hardy spaces is Beurling's theorem \cite{Beurling} for invariant subspaces of the Hardy space on unit disc $\D$  (denoted by $H^2(\D)$). In particular, it states that any shift-invariant closed subspace of $H^2(\D)$ is of the form $\theta(z) H^2(\D)$, where $\theta(z)$ is an \textit{inner} function on $\D$. Note that a bounded analytic function $\theta(\bm{z})$ on $\D^n$ is said to be an inner function if $|\theta(\bm{z})| = 1$ almost everywhere on $\mathbb{T}^n$.  It is well-known that a clever example by Rudin \cite{Rud} shows that such a characterization is far from being true in the case of $\D^2$. 
This initiated the search for new types of shift-invariant subspaces.  Before moving into further details, let us first recall that $H^2(\D^n)$ is the space of all analytic functions $f(\bm{z})$ on the unit polydisc $\D^n$ such that 
\[
\|f\|: = \sup_{0 \leq r<1} \big( \int_{\mathbb{T}^n} |f(r \bm{z})|^2 d\mu \big)^{\frac{1}{2}} < \infty,
\]
where $\bm{z}:= (z_1,\ldots,z_n) \in \D^n$, and $\mu$ is the normalized Lebesgue measure on the distinguished boundary of $\D^n$, that is the $n$-torus $\mathbb{T}^n$. The algebra of bounded analytic functions on $\D^n$ is denoted by $H^{\infty}(\D^n)$.  The Hardy space $H^2(\D^n)$ has the following module structure over $\mathbb{C}[z_1,\ldots,z_n]$,
\[
p \cdot f  = p(M_{\bm{z}})f \quad (f \in H^2(\D^n)),
\]
where $M_{\bm{z}}:= (M_{z_1}, \ldots, M_{z_n})$ is the tuple of shift operators on $H^2(\D^n)$. A closed subspace $\cls \subseteq H^2(\D^n)$ is called $M_{\bm{z}}$-invariant (also known as shift-invariant) if $M_{z_i} \cls \subseteq \cls$ is satisfied for all $i \in \{1,\ldots,n\}$.  The above module structure gets carried onto shift-invariant subspaces via the restriction operators: $$R_{z_i}: = M_{z_i}|_{\cls} \quad (i \in \{1,\ldots,n\}).$$ Due to this, we will follow the convention of Douglas and Paulsen \cite{DP} by referring to shift-invariant closed subspaces as \textit{submodules} of $H^2(\D^n)$. The classification problem for submodules of $H^2(\D^n)$ for $n \geq 2$ is challenging, and a full description is beyond the scope of the present understanding. We refer the readers to important developments in this direction \cite{ACD, AC, BDDS, Guo, GW, GGK, III, III2, Yang3, Yang2, Yang1}. To study this problem, researchers have devised an important approach over time: studying submodules through the lens of restriction operators. More precisely, characterizing shift-invariant subspaces of $H^2(\D^n)$ by imposing algebraic conditions on the restriction operators. A breakthrough in this direction came via the characterization for \textit{Beurling-type} submodules of $H^2(\D^n)$.

\begin{defn}
A submodule $\cls \subseteq H^2(\D^n)$ is said to be Beurling-type if there exists an inner function $\theta(\bm{z}) \in H^{\infty}(\D^n)$ such that $\cls = \theta(\bm{z}) H^2(\D^n)$.
\end{defn}

The following characterization was obtained by Mandrekar for $\D^2$ \cite{Mandrekar} and later by Sarkar et al. for $n >2$ \cite{SSW}.
\begin{thm}\label{Beurling}
Let $\cls$ be a submodule of $H^2(\D^n)$, then the following are equivalent.
\begin{enumerate}
\item[(i)] $\cls$ is of Beurling-type,
\item[(ii)] $[R_{z_i}^*,R_{z_j}] = 0$ for any distinct $i,j \in \{1,\ldots,n\}$.
\end{enumerate}
\end{thm}
Here, $[\cdot,\cdot]$ denotes the commutator $[A, B]:= AB - BA$ for $A, B \in \clb(\clh)$ (the space of all bounded operators on a Hilbert space $\clh$). We refer the reader to a recent work by Bergqvist \cite{Bergqvist} for a function-theoretic approach to Mandrekar's characterization.  An immediate question that one may ask is when the inner function $\theta(\bm{z})$ depends on a particular variable. Although this question is natural and interesting in its own right,  it is surprising that a single set of explicit conditions has not yet been developed to characterize Beurling-type submodules with this particular feature.  This is the first achievement of this article. To understand the result, let us introduce another collection of operators associated with the submodule $\cls$. More precisely, for each $j \in \{1,\ldots,n\}$ and $\lambda \in \D$, we define evaluation operators corresponding to $\cls$ as the following,
\[
E_{j,\lambda} := P_{\cls} \circ ev_{j, \lambda},
\]
where, for any $j \in \{1,\ldots,n\}$, and $\lambda \in \D$,  $ev_{j, \lambda}: H^2(\D^n)\raro H^2(\D^n)$ is the evaluation operator on the full space defined by
\[
ev_{j, \lambda}f (z_1, \ldots, z_n)=f(z_1,\ldots,z_{j-1},\lambda,z_{j+1},\ldots, z_n) \quad  (f\in H^2(\D^n)).
\]
Here, $P_{\cls}$ denotes the orthogonal projection onto the closed subspace $\cls$. We denote the orthogonal complement of $\cls$ by $\cls^{\perp}$, and the orthogonal projection onto this subspace by $P_{\cls}^{\perp}$. For this article, we are interested in evaluation operators at $\lambda = 0$, which is denoted by
\[
E_j: = E_{j,0}.
\]
It should be noted that evaluation operators corresponding to quotient modules have been studied in \cite{LYY, Yang2}, where several correspondences with the compression operators have been developed. Our interest lies in the study of these operators on submodules. In section \ref{INS2}, we use the collection of evaluation operators and restrictions to completely characterize Beurling-type submodules corresponding to inner functions depending on a particular variable.
\begin{thm}\label{Beurling-type}
Let $\cls$ be a submodule of $H^2(\D^n)$. Then $\cls=\theta(z_i)H^2(\D^n)$ for an inner function $\theta(z_i)$ depending only on the $z_i$-variable if and only if 
\[
(I_{\cls}-E_{z_j}E_{z_j}^*)(I_{\cls}-R_{z_j}R_{z_j}^*)=0.
\]
for all $j \in \{1,\ldots,n\} \setminus \{i\}$.
\end{thm}
The above result is surprising as we need conditions corresponding to all the variables except for the variable on which the inner function depends. It is interesting to observe that if we just assume the condition $(I_{\cls}-R_{z_j}R_{z_j}^*)=0$ for any $i \in \{1,\ldots,n\}$ then the only possibility is $\cls = \{0\}$.  This is because any $R_{z_i}$ is also a pure isometry (in other words, $\|R_{z_i}^{*n}f\| \raro 0$ as $n \raro \infty$ for all $f \in H^2(\D^n)$). Furthermore, in section \ref{Douglas}, we find that there does not exist any non-zero submodule of $H^2(\D^n)$ that satisfies the following condition
\[
(I_{\cls}-R_{z_i}R_{z_i}^*)(I_{\cls}-R_{z_j}R_{z_j}^*) = 0,
\]
for all distinct $i,j \in \{1,\ldots,n\}$ (see Proposition \ref{zerocondn}).   We believe this is an important phenomenon which, in other words, states that: given any non-zero submodule there must exist distinct $i,j \in \{1,\ldots,n\}$ such that $\cls \ominus z_i \cls \not \perp \cls \ominus z_j \cls$. 

Theorem \ref{Beurling-type} motivates us to characterize submodules which are sums of one-variable inner functions. More precisely, given a subset $\Lambda \subseteq \{1,\ldots,n\}$, when a submodule $\cls \subseteq H^2(\D^n)$ can be written as the following:
\[
\cls = \sum_{\lambda \in \Lambda} \vp_{\lambda}(z_{\lambda}) H^2(\D^n),
\]
where $\vp_{\lambda}(z_{\lambda})$ is an inner function depending only on the $z_{\lambda}$-variable for all $\lambda \in \Lambda$? In section \ref{INS3}, we completely answer this question by proving the following result.
\begin{thm}\label{inner based submodule}
Let $\cls$ be a proper submodule of $H^2(\D^n)$ and $\Lambda \subseteq \{1,\ldots,n\}$. There exists a collection of one-variable non-constant inner functions $\{\vp_{\lambda}(z_{\lambda})\}_{\lambda \in \Lambda}$ on $\D^n$,  such that
\[
\cls  =  \sum_{\lambda \in \Lambda} \vp_{\lambda}(z_{\lambda})H^2(\D^n),
\]
if and only if the following conditions are satisfied. 
\begin{equation}\label{condns1}
(I_{\cls}-E_{z_l}E_{z_l}^*)(I_{\cls}-R_{z_l}R_{z_l}^*)=0 \quad (\forall l\in \{1,\dots,n\} \setminus \Lambda),
\end{equation} 
\begin{equation}\label{condns2}
(I_{\cls}-E_{z_{i}}E_{z_{i}}^*)(I_{\cls}-R_{z_{i}}R_{z_{i}}^*)(I_{\cls}-E_{z_{j}}E_{z_{j}}^*)(I_{\cls}-R_{z_{j}}R_{z_{j}}^*)=0 \quad (\text{ for all distinct }i, j \in \Lambda).
\end{equation}
\end{thm}
The proof relies on several results on arbitrary submodules which we gather in section \ref{Douglas}. For instance, we prove and use a fact that for any submodule $\cls$,
\[
[(I_{\cls}-E_{z_{i}}E_{z_{i}}^*),(I_{\cls}-R_{z_{i}}R_{z_{i}}^*)] = 0 \quad (\forall i \in \{1,\ldots,n\}).
\]
Furthermore, the novelty of this section lies in identifying the interplay between $[R_{z_i}^*, R_{z_j}]$ and the cross-commutators  $[C_{z_i}^*, C_{z_j}]$, corresponding to compressions $C_{z_i}: = P_{\clq}M_{z_i}|_{\clq}$ on the quotient modules $\clq:= H^2(\D^n) \ominus \cls$ (that is, the orthogonal complement of $\cls$ inside $H^2(\D^n)$). By definition, it is clear that $\clq$ is $M_{\bm{z}}^*$-invariant that is, $M_{z_i}^* \clq \subseteq \clq$ for all $i \in \{1,\ldots,n\}$. We hope many of these results will prove useful in further studies. 

From the algebraic condition in Theorem \ref{Beurling}, it is evident that Beurling-type submodules of $H^2(\D^n)$ satisfy the following commutativity,
\[
[(I_{\cls}-R_{z_i}R_{z_i}^*),(I_{\cls}-R_{z_j}R_{z_j}^*)]=0 \quad(\text{ for all distinct }i,j \in \{1,\ldots,n\}),
\]
in other words, the defects of restrictions commute with each other. Motivated by this feature, in section \ref{defects}, we have completely characterized when the defects of restrictions corresponding to submodules of type $\cls  =  \sum_{\lambda \in \Lambda} \vp_{\lambda}(z_{\lambda})H^2(\D^n),$ commute with each other.  
\begin{thm}\label{commuting defects}
Let $\cls = \sum_{\lambda \in \Lambda} \vp_{\lambda}(z_{\lambda})H^2(\D^n)$ be a submodule of $H^2(\D^n)$, where $\vp_{\lambda}$ are non-constant inner functions. For any distinct $i,j\in \Lambda$, we have 
 $$[(I_{\cls_{\Phi_{\Lambda}}}-R_{z_{i}}R_{z_{i}}^*), (I_{\cls_{\Phi_{\Lambda}}}-R_{z_{j}}R_{z_{j}}^*)]=0$$  if and only if $\vp_{i}(0)=0$ and $\vp_{j}(0)=0$.
\end{thm}

An intrinsic question in operator theory is when two orthogonal projections say $P_1, P_2$ on a Hilbert space $\clh$, commute with each other. It is well-known that this commutativity holds if and only if $P_1P_2 = P_{\mbox{ran} P_1 \cap \mbox{ran} P_2}$. The analogous question for $H^2(\D^n)$ can be stated as the following: \textit{when does orthogonal projections $P_{\cls_i}$ onto submodules $\cls_i$ for $i=1,2$ commute with each other?} This is a challenging question, and a full description seems out of reach with the existing understanding of general submodules of $H^2(\D^n)$. We refer the reader to recent progress made for Beurling-type submodules \cite{DPS}. In section \ref{commuting}, we completely answer this question for the above type of submodules.
\begin{thm}\label{comm_projn}
Let $\cls_{\Phi_{\Lambda}}$ and $\cls_{\Psi_{\Gamma}}$ be the following submodules of $H^2(\D^n)$
\[
\cls_{\Phi_{\Lambda}} = \sum_{\lambda \in \Lambda} \vp_{\lambda}(z_\lambda) H^2(\D^n); \quad \cls_{\Psi_{\Gamma}} = \sum_{t \in \Gamma} \psi_{t}(z_t) H^2(\D^n),
\]
corresponding to non-constant inner functions $\{\vp_{\lambda}(z_{\lambda})\}_{\Lambda}$ and $\{\psi_{t}(z_{t})\}_{\Gamma}$. Then $[P_{\cls_{\Phi_{\Lambda}}}, P_{\cls_{\Psi_{\Gamma}}}]=0$ if and only if either $\vp_j|\psi_j$ or $\psi_j|\vp_j$ for all $j\in \Lambda \cap \Gamma$.
\end{thm}
The proof of this result is not straightforward and involves several crucial steps. The main difficulty in proving this result lies in the structure of these submodules. Although we want to study when two projections commute, it is worth noting that each of these projections is again a sum of orthogonal projections. In other words, we have to prove when sums of orthogonal projections commute, which significantly raises the difficulty. 

In \cite{BL}, the authors have recorded the following question raised by R.G. Douglas: \textit{when does the product of two orthogonal projections corresponding to Beurling-type quotient modules become a finite-rank projection?} Note that a quotient module $\clq \subseteq H^2(\D^n)$ is said to be a \textit{Beurling-type} quotient module if there exists an inner function $\theta(\bm{z}) \in H^{\infty}(\D^n)$ such that $\clq = \clq_{\theta}:= H^2(\D^n) \ominus \theta(\bm{z}) H^2(\D^n)$. Douglas's question is fundamental in the study of both submodules and quotient modules of $H^2(\D^n)$ (see for instance, see \cite{BL, GW}), and was recently solved by Debnath et al. in \cite[Theorem 1.4]{DPS}.  Their result, in the case of $n>2$ shows that the product $P_{\clq_{\theta}}  P_{\clq_{\psi}}$ corresponding to inner functions $\phi(\bm{z}), \psi(\bm{z}) \in H^{\infty}(\D^n)$ is a finite-rank projection only when it is zero.  However, we have a non-trivial answer to Douglas's question for the submodules which we have considered. Using Theorem \ref{comm_projn}, we establish the analogous result for quotient modules corresponding to quotient modules $\clq_{\Phi_{\Lambda}}: = H^2(\D^n) \ominus\cls_{\Phi_{\Lambda}}$, and $\clq_{\Psi_{\Lambda}}: = H^2(\D^n) \ominus\cls_{\Psi_{\Gamma}}.$
\begin{thm}\label{finite-rank}
Let $\Lambda, \Gamma$ be subsets of $\{1,\ldots,n\}$ such that $\Lambda \cup \Gamma = \{1,\ldots,n\}$ and consider the corresponding quotient modules $\clq_{\Phi_{\Lambda}}, \clq_{\Psi_{\Gamma}}$ inside $H^2(\D^n)$, where $\{\vp_{\lambda}(z_{\lambda})\}_{\Lambda}$ and $\{\psi_{t}(z_{t})\}_{\Gamma}$ are non-constant inner functions. Then $P=P_{\clq_{\Phi_{\Lambda}}}P_{\clq_{\Psi_{\Gamma}}}$ is a finite rank projection if and only if the inner functions $\{\vp_i(z_i)\}_{i \in \Lambda \setminus \Lambda \cap \Gamma}$ and $\{\psi_j(z_j)\}_{j \in \Gamma \setminus \Lambda \cap \Gamma}$ are finite Blaschke products, and for all $j \in \Lambda \cap \Gamma$ any one of the following conditions hold,
\begin{enumerate}
\item[(a)] $\vp_j \text{ divides } \psi_j$ and $\vp_j$ is a finite Blaschke product,
\item[(b)] $\psi_j \text{ divides } \vp_j$ and $\psi_j$ is a finite Blaschke product. 
\end{enumerate}
\end{thm}

\section{Properties of submodule on the polydisc}\label{Douglas}
This section aims to develop several results on general submodules of $H^2(\D^n)$.  Let us begin by looking into the following question: \textit{when does the condition $\cls \ominus z_i \cls \perp \cls \ominus z_j \cls $ holds for all distinct $i,j \in \{1,\ldots,n\}$?} Since the tuple of restriction operators $(R_{z_1},\dots, R_{z_n})$ form a commuting tuple of isometries the condition can be restated into the following $$(I_{\cls} - R_{z_i}R_{z_i}^*)(I_{\cls} - R_{z_j}R_{z_j}^*)=0.$$
Before answering this question, let us note a useful and well-known result on reducing subspaces of $H^2(\D^n)$. Recall that a subspace $\cls \subseteq H^2(\D^n)$ is called $M_{\bm{z}}$-reducing if $M_{z_i} \cls \subseteq \cls$, and $M_{z_i}^* \cls \subseteq \cls$ hold for all $i \in \{1,\ldots,n\}$.
\begin{prop}\label{reducing}
The only $M_{\bm{z}}$-reducing closed subspaces of $H^2(\D^n)$ are the trivial ones.
\end{prop}
\begin{proof}
It can be immediately observed that $\cls$ is $M_{\bm{z}}$-reducing implies that 
\[\underset{i=1}{\overset{n}{\Pi}} (I_{H^2(\D^n)} - M_{z_i} M_{z_i}^*) P_{\cls} = P_{\cls} \underset{i=1}{\overset{n}{\Pi}} (I_{H^2(\D^n)} - M_{z_i} M_{z_i}^*).
\]
Since $\underset{i=1}{\overset{n}{\Pi}} (I_{H^2(\D^n)} - M_{z_i} M_{z_i}^*) = P_{\mathbb{C}}$, we get $P_{\mathbb{C}} P_{\cls} = P_{\cls} P_{\mathbb{C}}$, and therefore, $P_{\mathbb{C}} P_{\cls} 1= P_{\cls} 1$.  Hence, $P_{\cls} 1 \in \mathbb{C}$ and thus, there exists $\lambda \in \mathbb{C}$ such that $P_{\cls} 1 = \lambda$. Since $P_{\mathbb{C}}$ is a projection, $\lambda \in \{0,1\}$. If $\lambda = 0$, then it means $1 \in \cls^{\perp}$, which will imply $\cls^{\perp} = H^2(\D^n)$. In other words, $\cls = \{0\}$. In the other case, when $\lambda = 1$, it will imply that $P_{\cls} 1 = 1$ and thus, $1 \in \cls$. This will again imply that $\cls = H^2(\D^n)$. This completes the proof.
\end{proof}
Now we are ready to answer the question at this section's beginning.
\begin{prop}\label{zerocondn}
Let  $\cls$ submodule of $H^2(\D^n)$. Then \begin{equation}\label{e1}
(I_{\cls} - R_{z_i}R_{z_i}^*)(I_{\cls} -R_{z_j}R_{z_j}^*)=0
\end{equation}
for all distinct $i,j\in \{1,\dots,n\}$ if and only if $\cls=\{0\}$.
\end{prop}

\begin{proof}
Fix any distinct $i,j \in \{1,\dots,n\}$. Then $(I_{\cls} - R_{z_i}R_{z_i}^*)(I_{\cls} - R_{z_j}R_{z_j}^*)=0$ implies 
\[
(I_{\cls} - R_{z_j}R_{z_j}^*)\cls \subseteq \text{ker}(I_{\cls} - R_{z_i}R_{z_i}^*),
\]
that is
\[
\cls\ominus z_j \cls\subseteq z_i\cls.
\]
Since for each $j$ we have $z_j^nz_i\cls=z_iz_j^n\cls \subseteq z_i\cls$, we can further deduce that
\[
z_j^n(\cls\ominus z_j \cls)\subseteq z_i\cls.
\]
From the Wold-von Neumann decomposition we get $\cls=\oplus_{j=0}^n z_j^n(\cls\ominus z_j \cls)$, and therefore, we obtain
$\cls\subseteq z_i\cls$. Acting on the left by $M_{z_i}^*$ on both sides we get $M_{z_i}^*\cls\subseteq \cls$. By considering the adjoint of condition (\ref{e1}), we have $(I_{\cls} - R_{z_j}R_{z_j}^*)(I_{\cls} - R_{z_i}R_{z_i}^*)=0$. So by a similar technique, we conclude   $M_{z_j}^*\cls\subseteq \cls$. Since $i,j$ were fixed but arbitrary elements in $\{1,\ldots,n\}$ the submodule $\cls$ must be $M_{\bm{z}}$-reducing and hence, using Proposition \ref{reducing}, we get
\[
\cls=\{0\} \quad\text{or}\quad H^2(\D^n).
\]
Now if $\cls=H^2(\D^n)$, then
\[
0=(I_{\cls} - R_{z_i}R_{z_i}^*)(I_{\cls} - R_{z_j}R_{z_j}^*)=(I_{H^2(\D^n)} - M_{z_i}M_{z_i}^*)(I_{H^2(\D^n)} - M_{z_j}M_{z_j}^*).
\]
This further implies that either $(I_{H^2(\D^n)}-M_{z_i}M_{z_i}^*)=0$ or $(I_{H^2(\D^n)}-M_{z_j}M_{z_j}^*)=0$. This is a contradiction  because $M_{z_i}$ is never a unitary for any $i \in \{1,\ldots,n\}$. Thus the only possibility is $\cls=\{0\}$. This completes the proof.
\end{proof}

Let us now establish several results useful for later sections. In the following $\clq:= H^2(\D^n) \ominus \cls$ is the corresponding quotient module. It is evident from the definition that $\clq$ is $(M_{z_1}^*, \ldots, M_{z_n}^*)$-joint invariant. We begin by observing the following simple but useful identity. 

\begin{lem}\label{commutator rest}
Let $\cls$ be a submodule of $H^2(\D^n)$. Then 
\[
[R_{z_j}^*, R_{z_i}]=P_{\cls}M_{z_i}P_{\clq}M_{z_j}^*|_{\cls}, \quad (i\neq j).
\]
for all distinct $i,j \in \{1,\dots, n\}$. 
\end{lem}
\begin{proof}
Fix distinct $i, j \in \{1,\dots,n\}$.	We consider
\[
\begin{split}
[R_{z_j}^*, R_{z_i}]
&=R_{z_j}^* R_{z_i}-R_{z_i}R_{z_j}^*
\\
&=P_{\cls} M_{z_j}^*M_{z_i}|_{\cls}-P_{\cls}M_{z_i}P_{\cls}M_{z_j}^*|_{\cls}
\\
&=P_{\cls}M_{z_i}P_{\clq}M_{z_j}^*|_{\cls}.
\end{split}
\]
\end{proof}

The following results show how the evaluation operators on submodules are related to the restriction operators. 

\begin{lem}\label{ABC}
Let $\cls$ be a submodule of $H^2(\D^n)$, then
\[
I_{\cls}-R_{z_j}R_{z_j}^*-E_{z_j}E_{z_j}^*=P_{\cls}M_{z_j}P_{\clq}M^*_{z_j}|_{\cls}, \quad (j \in \{1,\dots,n\}).
\]
\end{lem}

\begin{proof}
First note that by the definition, $ev_{z_j}= I_{H^2(\D^n)} - M_{z_j}M_{z_j}^*$ for all $j \in \{1,\ldots,n\}$. Then 
\[
E_{z_j}E_{z_j}^*=P_{\cls} \big( I_{H^2(\D^n)} - M_{z_j}M_{z_j}^* \big) P_{\cls} = P_{\cls} -  P_{\cls}  M_{z_j}M_{z_j}^* P_{\cls} .
\]
So, we obtain
\[
I_{\cls}-R_{z_j}R_{z_j}^*-E_{z_j}E_{z_j}^*=P_{\cls}M_{z_j}M_{z_j}^*|_{\cls}-P_{\cls}M_{z_j}P_{\cls}M_{z_j}^*|_{\cls}=P_{\cls}M_{z_j}P_{\clq}M^*_{z_j}|_{\cls}.
\]
\end{proof}

As a consequence of this identity, we can further establish that the defects of the restriction and evaluation operators commute with each other.

\begin{lem}\label{1st factor}
Let $\cls$ be a submodule of $H^2(\D^n)$. Then for each $j \in \{1,\dots,n\}$, we get
\[
(I_{\cls} - R_{z_j}R_{z_j}^*) (I_{\cls} - E_{z_j}E_{z_j}^*)
= I_{\cls} - R_{z_j}R_{z_j}^*-E_{z_j}E_{z_j}^*.
\]
\end{lem}

\begin{proof}
For any $j \in \{1, \ldots, n\}$ we have
\[
\begin{split}
(I_{\cls} - R_{z_j}R_{z_j}^*)(I_{\cls} - E_{z_j}E_{z_j}^*)&=(I_{\cls} - P_{\cls}M_{z_j}P_{\cls}M_{z_j}^*|_{\cls})P_{\cls}M_{z_j}M_{z_j}^*|_{\cls}
\\
&=P_{\cls}M_{z_j}M_{z_j}^*|_{\cls} - P_{\cls}M_{z_j}P_{\cls}M_{z_j}^*P_{\cls}M_{z_j}M_{z_j}^*|_{\cls} \\
&=P_{\cls}M_{z_j}M_{z_j}^*|_{\cls} - P_{\cls}M_{z_j}P_{\cls}M_{z_j}^*M_{z_j}M_{z_j}^*|_{\cls} \\
&=P_{\cls}M_{z_j}M_{z_j}^*|_{\cls} - P_{\cls}M_{z_j}P_{\cls}M_{z_j}^*|_{\cls} \\
&=P_{\cls}M_{z_j}M_{z_j}^*|_{\cls} - P_{\cls}M_{z_j}(I_{H^2(\D^n)} - P_{\clq})M_{z_j}^*|_{\cls} \\
&= P_{\cls}M_{z_j} P_{\clq} M_{z_j}^*|_{\cls}.
\end{split}
\]
Using Lemma \ref{ABC} we get
\[
(I_{\cls} - R_{z_j}R_{z_j}^*)(I_{\cls} - E_{z_j}E_{z_j}^*)=I_{\cls}-R_{z_j}R_{z_j}^*-E_{z_j}E_{z_j}^*.
\]
This completes the proof.
\end{proof}
Since the right-hand side of the above result is a self-adjoint operator, we get
\[
[(I_{\cls} - R_{z_j}R_{z_j}^*),(I_{\cls} - E_{z_j}E_{z_j}^*)]=0 \quad (j \in \{1,\ldots,n\}).
\]
Let us now collect several results on cross-commutators of the restriction operators on submodules of $H^2(\D^n)$, which will play crucial roles in the later sections.  The Douglas factorization lemma \cite{Douglas} serves as an important tool in the proofs. 

\begin{lem}\label{submodule factor}
Let $\cls$ be a submodule of $H^2(\D^n)$. For distinct $i,j \in \{1,\dots,n\}$, we have
\begin{itemize}
	\item[(i)] $[R_{z_j}^*, R_{z_i}]=(I_{\cls}-R_{z_i}R_{z_i}^*-E_{z_i}E_{z_i}^*)^{1/2}A_{i,j}$ for some contraction $A_{i,j}$ on $\cls$.\vspace{1mm}
	\item[(ii)] $[R_{z_j}^*, R_{z_i}]=B_{i,j} (I_{\cls}-R_{z_j}R_{z_j}^*-E_{z_j}E_{z_j}^*)^{1/2}$ for some contraction $B_{i,j}$ on $\cls$.\vspace{1mm}
	\item[(iii)] $\overline{\text{ran}} [R_{z_j}^*, R_{z_i}]\subseteq 
	\overline{\text{ran}} (I_{\cls}-R_{z_i}R_{z_i}^*-E_{z_i}E_{z_i}^*)$.  \vspace{1mm}
	\item[(iv)] $\text{ker}(I_{\cls}-R_{z_j}R_{z_j}^*-E_{z_j}E_{z_j}^*)\subseteq \text{ker} [R_{z_j}^*, R_{z_i}]$.
\end{itemize}
\end{lem}
\begin{proof} 
		Fix any distinct $i,j \in \{1,\dots,n\}$, then
		\[
		\begin{split}
&(I_{\cls}-R_{z_i}R_{z_i}^*-E_{z_i}E_{z_i}^*)-[R_{z_j}^*, R_{z_i}][R_{z_j}^*, R_{z_i}]^*\\&=P_{\cls}M_{z_i}P_{\clq}M^*_{z_i}P_{\cls}-P_{\cls}M_{z_i}P_{\clq}M_{z_j}^*P_{\cls}M_{z_j}P_{\clq}M_{z_i}^*P_{\cls}
\\
&=P_{\cls}M_{z_i}P_{\clq}(I_{H^2(\D^n)}-M_{z_j}^*P_{\cls}M_{z_j})P_{\clq}M_{z_i}^*P_{\cls}
\\
&=P_{\cls}M_{z_i}P_{\clq}M_{z_j}^*P_{\clq}M_{z_j}P_{\clq}M_{z_i}^*P_{\cls}
\\
&=P_{\cls}M_{z_i}M_{z_j}^*P_{\clq}M_{z_j}M_{z_i}^*P_{\cls} \geq 0.
		\end{split}
		\]
Furthermore, from Lemma \ref{ABC} we know $(I_{\cls}-R_{z_i}R_{z_i}^*-E_{z_i}E_{z_i}^*)$ is a non-negative operator. Hence, by Douglas's range inclusion theorem there exists a contraction $A_{i,j}$ on $\cls$ such that
	\[
	[R_{z_j}^*, R_{z_i}]=(I_{\cls}-R_{z_i}R_{z_i}^*-E_{z_i}E_{z_i}^*)^{1/2}A_{i,j}.
	\]
	To prove part (ii), we consider
	\[
	\begin{split}
&(I_{\cls}-R_{z_j}R_{z_j}^*-E_{z_j}E_{z_j}^*)-[R_{z_j}^*, R_{z_i}]^*[R_{z_j}^*, R_{z_i}]\\&=P_{\cls}M_{z_j}P_{\clq}M^*_{z_j}P_{\cls}-P_{\cls}M_{z_j}P_{\clq}M_{z_i}^*P_{\cls}M_{z_i}P_{\clq}M_{z_j}^*P_{\cls}
\\
&=P_{\cls}M_{z_j}P_{\clq}(I_{H^2(\D^n)} - M_{z_i}^*P_{\cls}M_{z_i})P_{\clq}M^*_{z_j}P_{\cls}\\
&=P_{\cls}M_{z_j}P_{\clq}M_{z_i}^*(I_{H^2(\D^n)} - P_{\cls})M_{z_i}P_{\clq}M^*_{z_j}P_{\cls}\\
&=P_{\cls}M_{z_j}P_{\clq}M_{z_i}^*P_{\clq}M_{z_i}P_{\clq}M^*_{z_j}P_{\cls} \geq 0.
	\end{split}
	\]
So again by Douglas's range inclusion theorem, there exists a contraction $B_{i,j}$ on $\cls$ such that
	\[
	[R_{z_j}^*, R_{z_i}]=B_{i,j} (I_{\cls}-R_{z_j}R_{z_j}^*-E_{z_j}E_{z_j}^*)^{1/2}.
	\]
Part (iii) is a consequence of part (i) and part (iv) is a consequence of part (ii).
\end{proof}
Let us explore some further connections.
\begin{lem}
Let $\cls$ be a submodule of $H^2(\D^n)$. Then the following are equivalent:
\begin{enumerate}
\item[(i)] $[R_{z_j}^*, R_{z_i}](I_{\cls}-R_{z_i}R_{z_i}^*-E_{z_i}E_{z_i}^*)=0$, for all distinct $i,j \in \{1,\dots,n\}$, 
\item[(ii)] $(I_{\cls}-R_{z_j}R_{z_j}^*-E_{z_j}E_{z_j}^*)[R_{z_j}^*, R_{z_i}]=0$, for all distinct $i,j \in \{1,\dots,n\}$, 
\end{enumerate}

\end{lem}
\begin{proof}
	First note that $(I_{\cls}-R_{z_i}R_{z_i}^*-E_{z_i}E_{z_i}^*)$ is self adjoint and $[R_{z_j}^*, R_{z_i}]^*=[R_{z_i}^*, R_{z_j}]$. Thus,	
   $[R_{z_j}^*, R_{z_i}](I_{\cls}-R_{z_i}R_{z_i}^*-E_{z_i}E_{z_i}^*)=0$ if and only if $(I_{\cls}-R_{z_i}R_{z_i}^*-E_{z_i}E_{z_i}^*)[R_{z_i}^*, R_{z_j}]=0$. In the last condition if we  interchanging the role of $i$ and $j$, we get $(I_{\cls}-R_{z_j}R_{z_j}^*-E_{z_j}E_{z_j}^*)[R_{z_j}^*, R_{z_i}]=0$.  This completes the proof.
\end{proof}
We know that Beurling-type submodules of $H^2(\D^n)$ must satisfy $[R_{z_j}^*, R_{z_i}] = 0$ for all distinct $i,j \in \{1,\ldots,n\}$, and therefore, $[R_{z_j}^*, R_{z_i}]^2 = 0$. Using our methods, we can find sufficient conditions for an arbitrary submodule of $H^2(\D^n)$ to satisfy this condition.
\begin{prop}\label{square}
	Let $\cls$ be a submodule of $H^2(\D^n)$, then for any distinct $i,j \in \{1,\ldots,n\}$, we have
	\[
	[R_{z_j}^*, R_{z_i}](I_{\cls}-R_{z_i}R_{z_i}^*-E_{z_i}E_{z_i}^*)=0\quad \text{implies} \quad[R_{z_j}^*, R_{z_i}]^2=0.
	\]
\end{prop}

\begin{proof}
If we assume $[R_{z_j}^*, R_{z_i}](I_{\cls}-R_{z_i}R_{z_i}^*-E_{z_i}E_{z_i}^*)=0$ for distinct $i,j \in \{1,\ldots,n\}$, then 
\[
\overline{\text{ran}}(I_{\cls}-R_{z_i}R_{z_i}^*-E_{z_i}E_{z_i}^*)\subseteq  \text{ker}[R_{z_j}^*, R_{z_i}].
\]
On the other hand from Lemma \ref{submodule factor} (iii) we get
\[
\overline{\text{ran}}[R_{z_j}^*, R_{z_i}]\subseteq \overline{\text{ran}}(I_{\cls}-R_{z_i}R_{z_i}^*-E_{z_i}E_{z_i}^*).
\]
This implies  $\overline{\text{ran}}[R_{z_j}^*, R_{z_i}]\subseteq \text{ker}[R_{z_j}^*, R_{z_i}]$, which further implies that $[R_{z_j}^*, R_{z_i}]^2=0$. This completes the proof.
\end{proof}

In our study, we sometimes have to consider quotient modules and the corresponding compression operators. Recall that the compression operators (also referred to as the model operators \cite{Ber, Sarason, NF}) on quotient modules $\clq \subseteq H^2(\D^n)$ are denoted by $C_{z_j}$ and defined by
\[
C_{z_j}:=P_{\clq}M_{z_j}|_{\clq}.
\]
For all $j \in \{1,\dots,n\}$, the corresponding defect operators are defined by
\[
D_{C_{z_j}}:= (I_{\clq}-C_{z_j}^*C_{z_j})^{1/2},\quad \text{and}\quad D_{C_{z_j}^*}:=(I_{\clq}-C_{z_j}C_{z_j}^*)^{1/2},.
\]
Expanding the terms we get
\[
D_{C_{z_j}}^2 = P_{\clq} M_{z_j}^*P_{\cls} M_{z_j}P_{\clq},\quad \text{and}\quad D_{C_{z_j}^*}^2 = P_{\clq} (I_{H^2(\D^n)} - M_{z_j}M_{z_j}^*)P_{\clq}.
\]
The corresponding defect spaces are defined by
\[
\cld_{C_{z_j}}:=\overline{\text{ran}}D_{C_{z_j}} \quad \text{and}\quad
\cld_{C_{z_j}^*}:=\overline{\text{ran}}D_{C_{z_j}^*}.
\]
Note that
\begin{equation}\label{range}
\overline{\text{ran}}(I_{\clq} - C_{z_j}^*C_{z_j})^{1/2}=\overline{\text{ran}}(I_{\clq} - C_{z_j}^*C_{z_j})=\overline{\text{ran}}(P_{\clq}M_{z_j}^*|_{\cls}), \quad (j=1,\dots,n).
\end{equation}
Before moving into further details let us note a simple but useful identity.
\begin{lem}\label{comp}
Let $\clq$ be a quotient module of $H^2(\D^n)$, then for any distinct $i,j \in \{1,\dots,n\}$,
 \[
[C_{z_j},C_{z_i}^*]=P_{\clq}M_{z_i}^*P_{\cls}M_{z_j}|_{\clq}.
\]
\end{lem}
\begin{proof}
	Fix any distinct $i,j \in \{1,\dots,n\}$, then
	\[
	\begin{split}
		[C_{z_j},C_{z_i}^*]&=C_{z_j}C_{z_i}^*-C_{z_i}^*C_{z_j}
		\\
		&=P_{\clq}M_{z_j}M_{z_i}^*|_{\clq}-P_{\clq}M_{z_i}^*P_{\clq}M_{z_j}|_{\clq}
		\\
		&=P_{\clq}M_{z_i}^*P_{\cls}M_{z_j}|_{\clq}.
	\end{split}
	\]
\end{proof}
We present a few results on compressions and cross-commutators analogous to Lemma \ref{submodule factor}.
\begin{lem}\label{containment}
Let $\clq$ be a quotient module of $H^2(\D^n)$, then for any distinct $i,j \in \{1,\dots,n\}$, we have
\begin{itemize}
	\item[(i)] $[C_{z_j},C_{z_i}^*]=A_{i,j}D_{C_{z_j}}$ for some contraction $A_{i,j}$ on $\clq$. \vspace{1mm}
	\item[(ii)]$[C_{z_j},C_{z_i}^*]=D_{C_{z_i}}B_{i,j},$ for some contraction $B_{i,j}$ on $\clq$.\vspace{1mm}
	\item[(iii)] $\overline{\text{ran}} [C_{z_j},C_{z_i}^*]\subseteq \cld_{C_{z_i}}$.\vspace{1mm}
	\item[(iv)] $\text{ker}(D_{C_{z_j}})\subseteq \text{ker} [C_{z_j},C_{z_i}^*]$.
\end{itemize}
\end{lem}

\begin{proof}
	Fix any distinct $i,j \in \{1,\dots,n\}$. Then
	\[
	\begin{split}
     D_{C_{z_j}}^2-[C_{z_j},C_{z_i}^*]^*[C_{z_j},C_{z_i}^*]&=P_{\clq}M_{z_j}^*P_{\cls}M_{z_j}P_{\clq}  -   P_{\clq}M_{z_j}^*P_{\cls}M_{z_i}
       P_{\clq}M_{z_i}^*P_{\cls}M_{z_j}P_{\clq}
       \\
       &=P_{\clq}M_{z_j}^*P_{\cls} (I_{H^2(\D^n)} - M_{z_i}
       P_{\clq}M_{z_i}^*) P_{\cls}M_{z_j}P_{\clq} \geq 0.
	\end{split}
	\]
Hence, by Douglas's range inclusion theorem there exists a contraction $A_{i,j}$ on $\clq$ such that
	\[
	[C_{z_j},C_{z_i}^*]=A_{i,j}D_{C_{z_j}}.
	\]
		This completes the proof of (i). To prove (ii), we consider
		\[
\begin{split}
		D_{C_{z_i}}^2-[C_{z_j},C_{z_i}^*][C_{z_j},C_{z_i}^*]^*&=P_{\clq}M_{z_i}^*P_{\cls}M_{z_i}P_{\clq}- P_{\clq}M_{z_i}^*P_{\cls}M_{z_j}P_{\clq}M_{z_j}^*P_{\cls}M_{z_i}P_{\clq}
		\\
		&=P_{\clq}M_{z_i}^*P_{\cls} (I_{H^2(\D^n)}-M_{z_j}P_{\clq}M_{z_j}^*)   P_{\cls}M_{z_i}P_{\clq} \geq 0.
\end{split}
		\]
Thus, again by Douglas's range inclusion theorem, we get
		\[
		[C_{z_j},C_{z_i}^*]=D_{C_{z_i}}B_{i,j}
		\]
		for some contraction $B_{i,j}$ on $\clq$.
Now it is easy to conclude that part (iii)  is a consequence of part (ii) and part (iv) follows from part (i).
\end{proof}
We will now explore some correspondence between the cross-commutators of restrictions and compression operators. 

\begin{lem}
	Let $\cls$ be a submodule of $H^2(\D^n)$. Then for any distinct $i,j \in \{1,\dots,n\}$, 
	$[R_{z_j}^*, R_{z_i}](I_{\cls}-R_{z_i}R_{z_i}^*-E_{z_i}E_{z_i}^*)=0$ if and only if $D_{C_{z_i}}[C_{z_i}, C_{z_j}^*]D_{C_{z_i}}=0$. Moreover,
	\[
	[R_{z_j}^*, R_{z_i}](I_{\cls}-R_{z_i}R_{z_i}^*-E_{z_i}E_{z_i}^*)=0\quad \text{implies} \quad [C_{z_i}, C_{z_j}^*]^2=0.
	\]
\end{lem}

\begin{proof} Fix any distinct but arbitrary $i,j \in \{1,\ldots,n\}$ and let us assume that $[R_{z_j}^*, R_{z_i}](I_{\cls}-R_{z_i}R_{z_i}^*-E_{z_i}E_{z_i}^*)=0$. Using Lemma \ref{commutator rest} and Lemma \ref{ABC}, for any fixed but arbitrary $s_1,s_2\in \cls$, we get
\[
\begin{split}
\la [R_{z_j}^*, R_{z_i}](I_{\cls}-R_{z_i}R_{z_i}^*-E_{z_i}E_{z_i}^*)s_1, s_2\ra
&=\la P_{\cls}M_{z_i}P_{\clq}M_{z_j}^*P_{\cls}M_{z_i}P_{\clq}M_{z_i}^*s_1,s_2 \ra
\\
&=\la P_{\clq}M_{z_j}^*P_{\cls}M_{z_i}P_{\clq}M_{z_i}^*s_1,  P_{\clq}M_{z_i}^*s_2 \ra.
\end{split}
\]
This implies that $P_{\clq}M_{z_j}^*P_{\cls}M_{z_i}P_{\clq}M_{z_i}^*\cls\subseteq\text{ker}(P_{\cls}M_{z_i}|_{\clq})$. So, $[C_{z_i}, C_{z_j}^*]\cld_{C_{z_i}}\subseteq \text{ker}(D_{C_{z_i}})$. Thus, $[R_{z_j}^*, R_{z_i}](I_{\cls}-R_{z_i}R_{z_i}^*-E_{z_i}E_{z_i}^*)=0$ if and only if $D_{C_{z_i}}[C_{z_i}, C_{z_j}^*]D_{C_{z_i}}=0$. Furthermore, by Lemma \ref{containment} (iv) we get $[C_{z_i}, C_{z_j}^*]\cld_{C_{z_i}}\subseteq \text{ker}(D_{C_{z_i}}) \subseteq  \text{ker} [C_{z_i},C_{z_j}^*]$. Hence $$[C_{z_i}, C_{z_j}^*]^2\cld_{C_{z_i}}=\{0\}.$$
This implies $\cld_{C_{z_i}}\subseteq \text{ker}[C_{z_i}, C_{z_j}^*]^2$. Again, using Lemma \ref{containment} (iv) we get $$\text{ker}(D_{C_{z_i}})\subseteq \text{ker} [C_{z_i},C_{z_j}^*]\subseteq \text{ker} [C_{z_i},C_{z_j}^*]^2.$$ Thus
\[
\clq=\text{ker}(D_{C_{z_i}})\oplus \cld_{C_{z_i}}\subseteq \text{ker}[C_{z_i}, C_{z_j}^*]^2,
\]
which further implies that $[C_{z_i}, C_{z_j}^*]^2=0$. This completes the proof.
\end{proof}

\begin{prop}
	Let $\cls$ be a submodule of $H^2(\D^n)$. Then for any distinct $i,j \in \{1,\dots,n\}$,
	$[R_{z_j}^*, R_{z_i}]^2=0$ if and only if $D_{C_{z_i}}[C_{z_i}, C_{z_j}^*]D_{C_{z_j}}=0$. Moreover,
	\[
	[R_{z_j}^*, R_{z_i}]^2=0 \quad \text{implies}\quad  [C_{z_i}, C_{z_j}^*]^3=0.
	\]
\end{prop}

\begin{proof} Fix any distinct $i,j \in \{1,\ldots,n\}$. For any fixed but arbitrary $s_1,s_2\in \cls$, we have 
\begin{align*}
0=\la [R_{z_j}^*, R_{z_i}]^2s_1,s_2 \ra &=\la P_{\cls}M_{z_i}P_{\clq}M_{z_j}^*P_{\cls}M_{z_i}P_{\clq}M_{z_j}^*s_1,s_2  \ra \\
&=\la P_{\clq}M_{z_j}^*P_{\cls}M_{z_i}P_{\clq}M_{z_j}^*s_1, P_{\clq}M_{z_i}^*s_2 \ra.
\end{align*}
This implies that $P_{\clq}M_{z_j}^*P_{\cls}M_{z_i}P_{\clq}M_{z_j}^*\cls\subseteq \text{ker}(P_{\cls}M_{z_i}|_{\clq})$. That is, $[C_{z_i}, C_{z_j}^*]\cld_{C_{z_j}}\subseteq \text{ker}(D_{C_{z_i}})$. Thus, we have $[R_{z_j}^*, R_{z_i}]^2=0$ if and only if $D_{C_{z_i}}[C_{z_i}, C_{z_j}^*]D_{C_{z_j}}=0$.

Now by Lemma \ref{containment} (iv) we have $\text{ker}(D_{C_{z_i}})\subseteq \text{ker} [C_{z_i},C_{z_j}^*]$. Thus $[C_{z_i}, C_{z_j}^*]^2\cld_{C_{z_j}}=\{0\}$. Hence
\begin{equation}\label{eq1}
\cld_{C_{z_j}}\subseteq \text{ker} [C_{z_i},C_{z_j}^*]^2.
\end{equation}

On the other hand by Lemma \ref{containment} (iii) we get
\begin{equation}\label{eq2}
\overline{\text{ran}} [C_{z_i},C_{z_j}^*]\subseteq \cld_{C_{z_j}}.
\end{equation}
Combining \eqref{eq1} and \eqref{eq2} we observe that $\overline{\text{ran}} [C_{z_i},C_{z_j}^*]\subseteq \text{ker} [C_{z_i},C_{z_j}^*]^2.$ Hence $ [C_{z_i},C_{z_j}^*]^3=0.$
\end{proof}
This result motivates the following problem: \textit{characterize submodules $\cls \subseteq H^2(\D^n)$ which satisfy $[R_{z_i}, R_{z_j}^*]^2=0$, but not necessarily $[R_{z_i}, R_{z_j}^*]=0$, for all distinct $i,j \in \{1,\ldots,n\}$?} In the next section, we will study certain submodules that satisfy this criterion. We seek to completely characterize submodules with respect to this condition in future works. We end this section with an observation useful in the later sections.
\begin{lem}\label{projection}
Let $\cls$ be a submodule of $H^2(\D^n)$, then for all $i \in \{1,\ldots,n\}$, the operator $M_{z_i} P_{\cls} M_{z_i}^*$ is a projection onto the closed subspace $z_i \cls$.
\end{lem}
\begin{proof}
It is clear that $M_{z_i} P_{\cls} M_{z_i}^*$ is a self-adjoint and idempotent operator, and hence an orthogonal projection on $H^2(\D^n)$. Also, the range of $M_{z_i} P_{\cls} M_{z_i}^*$ is clearly contained inside $z_i \cls$. We only need to show that $z_i \cls \subseteq \mbox{ran } M_{z_i} P_{\cls} M_{z_i}^*$.  So, let us consider $h \in z_i\cls$. Thus, there must exist $s \in \cls$ such that $h = z_i s$.  This implies that
\[
M_{z_i} P_{\cls} M_{z_i}^*h = z_i s = h.
\]
Hence, $z_i \cls \subseteq \mbox{ran } M_{z_i} P_{\cls} M_{z_i}^*$, and therefore, we get the equality of the subspaces. This completes the proof.
\end{proof}
\section{Inner submodules of $H^2(\D^n)$}\label{INS}

\subsection{One-variable Beurling-type submodules}\label{INS2}
In this section, we characterize when $\cls$ is a Beurling-type submodule of $H^2(\D^n)$ where the corresponding inner function depends on a particular variable.

\begin{proof}[Proof of Theorem \ref{Beurling-type}]
We have $(I_{\cls}-E_{z_j}E_{z_j}^*)(I_{\cls}-R_{z_j}R_{z_j}^*)=P_{\cls}M_{z_j}M_{z_j}^*(I_{\cls}-P_{\cls}M_{z_j}P_{\cls}M_{z_j}^*|_{\cls})$. If this is zero, then
\[
P_{\cls}M_{z_j}P_{\cls}M_{z_j}^*|_{\cls}=P_{\cls}M_{z_j}M_{z_j}^*|_{\cls}.
\] 
Thus we obtain
\[
P_{\cls}M_{z_j}P_{\clq}M_{z_j}^*|_{\cls}=0.
\]
Equivalently,
\[
P_{\cls}M_{z_j}P_{\clq}=0,
\]
which further implies that $ P_{\cls}M_{z_j}=P_{\cls}M_{z_j}P_{\cls} = M_{z_j}P_{\cls}$.  Hence, $\cls$ is $M_{z_j}$ reducing, for any $j \in \{1,\ldots,n\} \setminus \{i\}$. This implies that for any fixed but arbitrary $j \in \{1,\ldots,n\}$,
\[
R_{z_k}R_{z_j}^*=M_{z_k}P_{\cls}M_{z_j}^*|_{\cls}=M_{z_k}M_{z_j}^*|_{\cls}=P_{\cls}M_{z_k}M_{z_j}^*|_{\cls}=P_{\cls}M_{z_j}^*M_{z_k}|_{\cls}=R_{z_j}^*R_{z_k}, 
\]
for any $k \in \{1,\ldots,n\} \setminus \{j\}$. In other words, we get that $[R_{z_i}, R_{z_j}^*]=0$ for any distinct $i,j \in \{1,\ldots,n\}$. Using the characterization for Beurling-type invariant subspace of $H^2(\D^n)$ in Theorem \ref{Beurling}, we get 
\[
\cls=M_{\theta} H^2(\D^n),
\]
for some inner function $\theta(\bm{z}) \in H^{\infty}(\D^n)$.  Here, $M_{\theta}$ denotes the multiplication operator on $H^2(\D^n)$ corresponding to the symbol $\theta(\bm{z})$. Since $\cls$ is $M_{z_j}$-reducing for any $j\in \{1,\ldots,n\} \setminus \{i\}$,  it further implies that
\[
M_{z_j}M_{\theta}M_{\theta}^*=M_{\theta}M_{\theta}^*M_{z_j},
\]
and therefore,
\[
M_{z_j}M_{\theta}^*= M_{\theta}^*M_{z_j}.
\] 
Hence,
\[
(I_{H^2(\D^n)} - M_{z_j}M_{z_j}^*)M_{\theta}^*f=M_{\theta}^*(I_{H^2(\D^n)} - M_{z_j}M_{z_j}^*)f.
\] for all $f\in H^2(\D^n)$. Note that for any $f \in H^2(\D^n)$ we have
\[
(I_{H^2(\D^n)} - M_{z_j}M_{z_j}^*)f = f(z_1,\ldots,z_{j-1},0,z_{j+1}, \ldots,z_n).
\]
In particular, for any $\bm{w} \in \D^n$, if we take the corresponding Szeg\"o kernel $f=s_{\bm{w}}(\bm{z}):= \underset{i=1}{\overset{n}{\Pi}} \frac{1}{1 - z_i \bar{w_i}}$, then we get
\[
(I_{H^2(\D^n)} - M_{z_j}M_{z_j}^*)M_{\theta}^* s_{\bm{w}} = M_{\theta}^*(I_{H^2(\D^n)} -M_{z_j}M_{z_j}^*)s_{\bm{w}}.
\]
By the action of the adjoint of multiplication operators on Szeg\"o kernels, we get
\[ 
(I_{H^2(\D^n)} - M_{z_j}M_{z_j}^*) \overline{\theta(\bm{w})} s_{\bm{w}} = M_{\theta}^* s_{(w_1,\ldots,w_{j-1},0,w_{j+1}, \ldots,w_n)}.
\]
Again by the action of $M_{\theta}^*$ on Szeg\"o kernels, we get
\[
\overline{\theta(\bm{w})}  s_{(w_1,\ldots,w_{j-1},0,w_{j+1}, \ldots,w_n)} = \overline{\theta(w_1,\ldots,w_{j-1},0,w_{j+1}, \ldots,w_n)} s_{(w_1,\ldots,w_{j-1},0,w_{j+1}, \ldots,w_n)}.
\]
This implies that
\[
\theta(\bm{w}) = \theta(w_1,\ldots,w_{j-1},0,w_{j+1}, \ldots,w_n).
\]
Since $\bm{w}$ is an arbitrary point in $\D^n$, the function $\theta(\bm{z})$ must not depend on the variable  $z_j$. Furthermore, $j$ was an arbitrary number in $\{1,\ldots, n\} \setminus \{i\}$. So, this implies that $\theta(\bm{z})$ depends only on the $z_i$-variable.  In other words, 
\[
\cls = \theta(z_i) H^2(\D^n).
\]
Conversely, if $\theta(\bm{z})=\theta(z_i)$ for some $i \in \{1,\ldots,n\}$, then $\cls$ is $M_{z_j}$-reducing for any $j \in \{1,\ldots,n\} \setminus \{i\}$. Consequently, $P_{\cls}M_{z_j}M_{z_j}^*|_{\cls}=P_{\cls}M_{z_j}P_{\cls}M_{z_j}^*|_{\cls}$., and thus
\[
(I_{\cls}-E_{z_j}E_{z_j}^*)(I_{\cls}-R_{z_j}R_{z_j}^*)=P_{\cls}M_{z_j}M_{z_j}^*(I_{\cls}-P_{\cls}M_{z_j}P_{\cls}M_{z_j}^*|_{\cls})=0,
\]
for any $j \in \{1,\ldots,n\} \setminus \{i\}$. This completes the proof.
\end{proof}

\subsection{Sums of inner submodules} \label{INS3}
Let $\Lambda$ be a subset of $\{1,\ldots,n\}$, and $\{\vp_{\lambda}(z_\lambda)\}_{\lambda \in \Lambda}$ be a collection of inner functions on $\D^n$ such that $\vp_{\lambda}$ depends only on the $z_\lambda$-variable for all $\lambda \in \Lambda$. In this section, we completely characterize  submodules of the following type
\[
\cls_{\Phi_{\Lambda}} :=  \sum_{\lambda \in \Lambda} \vp_\lambda(z_\lambda)H^2(\D^n).
\]
In the case of $\Lambda = \{1,\ldots,n\}$ we simply denote the submodule by $\cls_{\Phi}$. We establish criterion based on conditions on the tuple of restrictions $(R_{z_1}, \ldots, R_{z_n})$, and evaluations $(E_{z_1}, \ldots, E_{z_n})$ corresponding to the submodule $\cls_{\Phi_{\Lambda}} \subseteq H^2(\D^n)$.  It is a straightforward observation that the quotient module corresponding to the above submodule has the following form.
\[
\clq_{\Phi_{\Lambda}} := \cls_{\Phi_{\Lambda}}^{\perp} = \underset{\lambda \in \Lambda}{\otimes} \clq_{\vp_\lambda}. 
\]
In the case of $\Lambda = \{1,\ldots,n\}$, there is a characterization for quotient modules of $H^2(\D^n)$ in terms of the compression operators $C_{z_i}: = P_{\clq_{\Phi}} M_{z_i}|_{\clq_{\Phi}}$, that was established by Izuchi et al. in \cite{INS} for the $n=2$ case and was later extended by Sarkar in \cite{Sarkar} for any $n \geq 3$. 
\begin{thm}\label{co-doubly}
Let $\clq$ be a non-zero quotient module of $H^2(\D^n)$, then for each $i \in \{1,\ldots,n\}$, there exists a  quotient module $\clq_i \subseteq H_{z_i}^2(\D)$ such that
\[
\clq = \clq_{1} \otimes \cdots \otimes \clq_{n},
\]
if and only if 
\[
[C_{z_j},C_{z_i}^*]=0,
\]
for all distinct $i,j \in \{1,\dots,n\}$.
\end{thm}
From Beurling's theorem for $H^2(\D)$, either $\clq_i = \clq_{\vp_i}: = H_{z_i}^2(\D) \ominus \vp_i(z_i) H_{z_i}^2(\D)$ corresponding to an inner function $\vp_{i}(z_i)$, or $\clq_i = H_{z_i}^2(\D)$. We refer the readers to the excellent article \cite{DY} for certain results related to the compression operators on these quotient modules even before the above structure was discovered. Regarding some recent results, we refer to the articles \cite{QY, Sarkar2} for Beurling-type representation of the corresponding submodules.  For the sake of computation, let $\Lambda = \{{\lambda}_1,\ldots,{\lambda}_k\}$, where ${\lambda}_1 < \ldots <{\lambda}_k$. The structure of $\cls_{\Phi_{\Lambda}}$ gives the following orthogonal decomposition.
\begin{equation}\label{structure}
\cls_{\Phi_{\Lambda}} = \cls_{1} \oplus\ldots \oplus  \cls_{k},
\end{equation}
where $\cls_{d} := \mbox{ran } M_{\vp_{{\lambda}_d}} M_{\vp_{{\lambda}_d}}^* \big( \underset{{\lambda}_k \in \Lambda; {\lambda}_k > {\lambda}_d}{\Pi} (I_{H^2(\D^n)} - M_{\vp_{{\lambda}_k}} M_{\vp_{{\lambda}_k}}^*) \big)$ for all $d \in \{1,\ldots,k\}$ (see for instance,  \cite[Lemma 2.5]{Sarkar}).  With respect to the above orthogonal decomposition, we get
\begin{equation}\label{proj2}
P_{\cls_{\Phi_{\Lambda}}} = P_{\cls_{1}} \oplus \ldots \oplus P_{\cls_{k}},
\end{equation}
that is, the orthogonal projection onto submodule $\cls_{\Phi_{\Lambda}}$ is the sum of certain orthogonal projections. We are now ready to establish the main result of this section. 

\begin{proof}[Proof for Theorem \ref{inner based submodule}]
Let us begin with the following submodule corresponding to the subset $\Lambda \subseteq \{1,\ldots,n\}$,
	\[
\cls_{\Phi_{\Lambda}} =  \sum_{{\lambda} \in \Lambda} \vp_{\lambda}(z_{\lambda})H^2(\D^n),
	\]
where $\{\vp_{\lambda}(z_{\lambda})\}_{{\lambda} \in \Lambda}$ is the collection of inner functions on $\D^n$.  From the structure in (\ref{structure}), it is clear that the submodule $\cls_{\Phi_{\Lambda}}$ is $M_{z_l}$-reducing for all $l \in \{1,\ldots,n\} \setminus \Lambda$. Therefore, using the definition of evaluation operators we get
\[
(I_{ \cls_{\Phi_{\Lambda}} }-E_{z_l}E_{z_l}^*) = P_{ \cls_{\Phi_{\Lambda}} } M_{z_l} M_{z_l}^*P_{ \cls_{\Phi_{\Lambda}} } = P_{ \cls_{\Phi_{\Lambda}} } M_{z_l} M_{z_l}^*,
\]
and similarly,
\[
(I_{ \cls_{\Phi_{\Lambda}} }-R_{z_l}R_{z_l}^*) = P_{ \cls_{\Phi_{\Lambda}} } - M_{z_l} P_{ \cls_{\Phi_{\Lambda}} } M_{z_l}^* = P_{ \cls_{\Phi_{\Lambda}} } - M_{z_l}  M_{z_l}^* P_{ \cls_{\Phi_{\Lambda}} } = (I_{H^2(\D^n)} - M_{z_l}  M_{z_l}^*) P_{\cls_{\Phi_{\Lambda}} }.
\]
This implies that $(I_{ \cls_{\Phi_{\Lambda}} }-E_{z_l}E_{z_l}^*) (I_{ \cls_{\Phi_{\Lambda}} }-R_{z_l}R_{z_l}^*) = 0$ for all $l \in \{1,\ldots,n\} \setminus \Lambda$. Now let us analyze for the co-ordinates inside $\Lambda$. For the sake of computations, let us denote
 \[
 \Delta_{i,j} = (I_{\cls_{\Phi_{\Lambda}} }-E_{z_{{\lambda}_i}}E_{z_{{\lambda}_i}}^*)(I_{\cls_{\Phi_{\Lambda}} }-R_{z_{{\lambda}_i}}R_{z_{{\lambda}_i}}^*)(I_{\cls_{\Phi_{\Lambda}} }-E_{z_{{\lambda}_j}}E_{z_{{\lambda}_j}}^*)(I_{ \cls_{\Phi_{\Lambda}} }-R_{z_{{\lambda}_j}}R_{z_{{\lambda}_j}}^*).
 \]
for distinct ${\lambda}_i,{\lambda}_j \in \Lambda$. From Lemma \ref{1st factor},  we know that 
\[ 
\Delta_{i,j} = (I_{\cls_{\Phi_{\Lambda}} }-R_{z_{{\lambda}_i}}R_{z_{{\lambda}_i}}^*-E_{z_{{\lambda}_i}}E_{z_{{\lambda}_i}}^*)(I_{\cls_{\Phi_{\Lambda}} }-R_{z_{{\lambda}_j}}R_{z_{{\lambda}_j}}^*-E_{z_{{\lambda}_j}}E_{z_{{\lambda}_j}}^*). 
\]
Using Lemma \ref{ABC}, we get
\[
\Delta_{i,j}=P_{\cls_{\Phi_{\Lambda}}}M_{z_{{\lambda}_i}}P_{\clq_{\Phi_{\Lambda}}}M^*_{z_{{\lambda}_i}}P_{\cls_{\Phi_{\Lambda}}}M_{z_{{\lambda}_j}}P_{\clq_{\Phi_{\Lambda}}}M^*_{z_{{\lambda}_j}}|_{\cls_{\Phi_{\Lambda}}}.
\]
We know that 
\[
P_{\cls_{\Phi_{\Lambda}}}= \sum_{d=1}^k M_{\vp_{{\lambda}_d}} M_{\vp_{{\lambda}_d}}^* \big( \underset{{\lambda} \in \Lambda; {\lambda} > {\lambda}_d}{\Pi} (I_{H^2(\D^n)} - M_{\vp_{{\lambda}}} M_{\vp_{{\lambda}}}^*) \big)
\]
Using the above identification, for any ${\lambda}_i \in \Lambda$, we get
\[
\begin{split}
&P_{\cls_{\Phi_{\Lambda}}} M_{z_{{\lambda}_i}} P_{\clq_{\Phi_{\Lambda}}} \\
&= \Big(\sum_{d=1}^k M_{\vp_{{\lambda}_d}} M_{\vp_{{\lambda}_d}}^* \big( \underset{{\lambda} \in \Lambda; {\lambda} > {\lambda}_d}{\Pi} (I_{H^2(\D^n)} - M_{\vp_{{\lambda}}} M_{\vp_{{\lambda}}}^*) \big) \Big) M_{z_{{\lambda}_i}} \underset{\lambda \in \Lambda}{\Pi} (I_{H^2(\D^n)} - M_{\vp_{{\lambda}}} M_{\vp_{{\lambda}}}^*) \\
&= \Big(\sum_{d=1}^k M_{\vp_{{\lambda}_d}} M_{\vp_{{\lambda}_d}}^* \big( \underset{{\lambda} \in \Lambda; {\lambda} > {\lambda}_d}{\Pi} (I_{H^2(\D^n)} - M_{\vp_{{\lambda}}} M_{\vp_{{\lambda}}}^*) \big) \Big)\\
& \underset{{\lambda} \in \Lambda; \lambda \neq \lambda_i}{\overset{n}{\Pi}} (I_{H^2(\D^n)} - M_{\vp_{{\lambda}}} M_{\vp_{{\lambda}}}^*) M_{z_{{\lambda}_i}}(I_{H^2(\D^n)} - M_{\vp_{{\lambda}_i}} M_{\vp_{{\lambda}_i}}^*)\\
&= \Big(\sum_{d=1}^{i-1} M_{\vp_{{\lambda}_d}} M_{\vp_{{\lambda}_d}}^* \big( \underset{{\lambda} \in \Lambda}{\Pi} (I_{H^2(\D^n)} - M_{\vp_{{\lambda}}} M_{\vp_{{\lambda}}}^*) \big) \Big) M_{z_{{\lambda}_i}} (I_{H^2(\D^n)} - M_{\vp_{{\lambda}_i}} M_{\vp_{{\lambda}_i}}^*)\\
&+\Big(\sum_{d=i}^k M_{\vp_{{\lambda}_d}} M_{\vp_{{\lambda}_d}}^* \big( \underset{{\lambda} \in \Lambda; \lambda \neq \lambda_i}{\Pi} (I_{H^2(\D^n)} - M_{\vp_{{\lambda}}} M_{\vp_{{\lambda}}}^*) \big) \Big) M_{z_{{\lambda}_i}} (I_{H^2(\D^n)} - M_{\vp_{{\lambda}_i}} M_{\vp_{{\lambda}_i}}^*)\\
&= M_{\vp_{{\lambda}_i}} M_{\vp_{{\lambda}_i}}^*  \big( \underset{{\lambda} \in \Lambda; {\lambda} \neq {\lambda}_i}{\Pi} (I_{H^2(\D^n)} - M_{\vp_{{\lambda}}} M_{\vp_{{\lambda}}}^*) \big) M_{z_{{\lambda}_i}} (I_{H^2(\D^n)} - M_{\vp_{{\lambda}_i}} M_{\vp_{{\lambda}_i}}^*)\\
&= M_{\vp_{{\lambda}_i}} M_{\vp_{{\lambda}_i}}^* M_{z_{{\lambda}_i}} \big( \underset{{\lambda} \in \Lambda}{\Pi} (I_{H^2(\D^n)} - M_{\vp_{{\lambda}}} M_{\vp_{{\lambda}}}^*) \big),
\end{split}
\]
and therefore, 
\[
P_{\clq_{\Phi_{\Lambda}}} M_{z_{{\lambda}_i}}^* P_{\cls_{\Phi_{\Lambda}}} = \big( \underset{\lambda \in \Lambda}{\Pi} (I_{H^2(\D^n)} - M_{\vp_{\lambda}} M_{\vp_{\lambda}}^*) \big) M_{z_{\lambda_i}}^* M_{\vp_{\lambda_i}} M_{\vp_{\lambda_i}}^*.
\]
This immediately gives us
\[
\begin{split}
&P_{\clq_{\Phi_{\Lambda}}}M_{z_{\lambda_i}}^*P_{\cls_{\Phi_{\Lambda}}}M_{z_{\lambda_j}}P_{\clq_{\Phi_{\Lambda}}} \\
&=  \big( \underset{\lambda \in \Lambda}{\Pi} (I_{H^2(\D^n)} - M_{\vp_{\lambda}} M_{\vp_{\lambda}}^*) \big) M_{z_{\lambda_i}}^* M_{\vp_{\lambda_i}} M_{\vp_{\lambda_i}}^* M_{\vp_{\lambda_j}} M_{\vp_{\lambda_j}}^* M_{z_{\lambda_j}} \big( \underset{\lambda \in \Lambda}{\Pi} (I_{H^2(\D^n)} - M_{\vp_{\lambda}} M_{\vp_{\lambda}}^*) \big)\\
&=  \big( \underset{\lambda \in \Lambda}{\Pi} (I_{H^2(\D^n)} - M_{\vp_{\lambda}} M_{\vp_{\lambda}}^*) \big) M_{\vp_{\lambda_j}} M_{\vp_{\lambda_j}}^* M_{z_{\lambda_j}}  M_{z_{\lambda_i}}^* M_{\vp_{\lambda_i}} M_{\vp_{\lambda_i}}^* \big( \underset{\lambda \in \Lambda}{\Pi} (I_{H^2(\D^n)} - M_{\vp_{\lambda}} M_{\vp_{\lambda}}^*) \big)\\
&=0,
\end{split}
\]
and hence, 
\[
\Delta_{i,j} = P_{\cls_{\Phi_{\Lambda}}}M_{z_{\lambda_i}}P_{\clq_{\Phi_{\Lambda}}}M^*_{z_{\lambda_i}}P_{\cls_{\Phi_{\Lambda}}}M_{z_{\lambda_j}}P_{\clq_{\Phi_{\Lambda}}}M^*_{z_{\lambda_j}}|_{\cls_{\Phi_{\Lambda}}} = 0,
\]
for any distinct $\lambda_i,\lambda_j \in \Lambda$.
For proving the converse direction, let us assume conditions (\ref{condns1}), and (\ref{condns2}). The second condition is equivalent to $\Delta_{i,j}=0$ for all distinct $\lambda_i,\lambda_j \in \Lambda$. Now, for any $h,k \in \cls$
\[
\begin{split}
0=\la \Delta_{i,j} h, k \ra =\la P_{\cls}M_{z_{\lambda_i}}P_{\clq}M^*_{z_{\lambda_i}}P_{\cls}M_{z_{\lambda_j}}P_{\clq}M^*_{z_{\lambda_j}}h, k \ra
=\la P_{\clq}M^*_{z_{\lambda_i}}P_{\cls}M_{z_{\lambda_j}}P_{\clq}M^*_{z_{\lambda_j}}h , P_{\clq}M^*_{z_{\lambda_i}}k \ra,
\end{split}
\]
This implies that $P_{\clq}M^*_{z_{\lambda_i}}P_{\cls}M_{z_{\lambda_j}}P_{\clq}M^*_{z_{\lambda_j}}\cls\subseteq \text{ker}(P_{\cls}M_{z_{\lambda_i}}|_{\clq})$. From condition (\ref{range}) and Lemma \ref{comp}, we get
\[
[C_{z_{\lambda_j}}, C_{z_{\lambda_i}}^*]\cld_{C_{z_{\lambda_j}}}\subseteq \text{ker}(D_{z_{\lambda_i}}).
\]
 By Lemma \ref{containment} (iii), we get $[C_{z_{\lambda_j}}, C_{z_{\lambda_i}}^*]\cld_{C_{z_{\lambda_j}}}\subseteq \cld_{z_{\lambda_i}}$. Consequently, $[C_{z_{\lambda_j}}, C_{z_{\lambda_i}}^*]\cld_{C_{z_{\lambda_j}}}={0}$. Thus 
\[
\cld_{C_{z_{\lambda_j}}}\subseteq \text{ker}([C_{z_{\lambda_j}}, C_{z_{\lambda_i}}^*]).
\]
 
On the other hand using Lemma \ref{containment} (iv), we obtain, $\text{ker}(D_{C_{z_{\lambda_j}}})\subseteq \text{ker}([C_{z_{\lambda_j}}, C_{z_{\lambda_i}}^*])$. Hence
 \[
 \clq=\cld_{C_{z_{\lambda_j}}}\oplus \text{ker}(D_{C_{z_{\lambda_j}}})\subseteq \text{ker}([C_{z_{\lambda_j}}, C_{z_{\lambda_i}}^*]).
 \]
This gives, $[C_{z_{\lambda_j}}, C_{z_{\lambda_i}}^*]=0$ for all distinct $\lambda_i,\lambda_j \in \Lambda$. Following the proof of Theorem \ref{Beurling-type}, we see that condition (\ref{condns1}) imply that $\cls$ is $M_{z_i}$-reducing for all $i \in \{1,\ldots,n\} \setminus \Lambda$.  This further implies that for distinct $i,j \in \{1,\ldots,n\} \setminus \Lambda$, we get
\[
C_{z_i} C_{z_j}^* = P_{\clq} M_{z_i}M_{z_j}^* P_{\clq} = P_{\clq} M_{z_j}^* M_{z_i} P_{\clq} = P_{\clq} M_{z_j}^* P_{\clq} M_{z_i} = M_{z_j}^* P_{\clq} M_{z_i}  = C_{z_j}^* C_{z_i}.
\]
Combining the above, we get $[C_{z_i},C_{z_j}^*]=0$, for all distinct $i,j \in \{1,\ldots,n\}$. Using Theorem \ref{co-doubly}, we observe that there must exist quotient modules $\clq_i \subseteq H_{z_i}^2(\D)$ for all $i \in \{1,\ldots,n\}$ such that 
\[
\cls^{\perp} = \clq_1 \otimes \ldots \otimes \clq_n.
\]
Furthermore, since $\cls$ is $M_{z_j}$-reducing for all $j \in \{1,\ldots,n\} \setminus \Lambda$, we get that 
\[
M_{z_j} P_{\clq_j} = P_{\clq_j}  M_{z_j},
\]
which implies $\clq_j = H_{z_j}^2(\D)$ for all $j \in \{1,\ldots,n\} \setminus \Lambda$. This follows from our assumption that $\cls$ is a proper subspace of $H^2(\D^n)$.  Thus,  we must get
 \[
\cls = \sum_{\lambda \in \Lambda} \vp_{\lambda}(z_{\lambda}) H^2(\D^n).
\]
This completes the proof.
\end{proof}
\begin{cor}
Given a submodule $\cls \subseteq H^2(\D^n)$, there exists a collection of one-variable inner functions $\{\vp_i(z_i)\}_{i \in \{1,\ldots,n\}}$ on $H^2(\D^n)$, for which $\cls = \sum_{i =1}^n \vp_i(z_i) H^2(\D^n)$, if and only if 
\[
(I_{\cls}-E_{z_i}E_{z_i}^*) (I_{\cls}-R_{z_i}R_{z_i}^*)(I_{\cls}-R_{z_j}R_{z_j}^*)(I_{\cls}-E_{z_j}E_{z_j}^*)=0,
\]
for all distinct $i,j \in \{1,\dots,n\}$.
\end{cor}
We can identify a new necessary condition for submodules of type $\cls_{\Phi_{\Lambda}}$ using our methods.
\begin{prop}
Given any subset $\Lambda \subseteq \{1,\ldots,n\}$, let $\cls_{\Phi_{\Lambda}} = \sum_{\lambda \in \Lambda} \vp_{\lambda}(z_{\lambda}) H^2(\D^n)$ be a submodule of $H^2(\D^n)$.  Then $$[R_{z_i}, R_{z_j}^*]^2 = 0,$$ for all distinct $i,j \in\{1,\ldots,n\}$.
\end{prop}
\begin{proof}
From Proposition \ref{square}, it is enough to show $[R_{z_j}^*, R_{z_i}](I_{\cls}-R_{z_i}R_{z_i}^*-E_{z_i}E_{z_i}^*)=0$ for distinct $i,j \in \{1,\ldots,n\}$. By Theorem \ref{inner based submodule}, we know that $\cls = \cls_{\Phi_{\Lambda}}$ if and only if the following conditions are satisfied
\begin{equation}\label{c1}
(I_{\cls_{\Phi_{\Lambda}}}-E_{z_l}E_{z_l}^*)(I_{\cls_{\Phi_{\Lambda}}}-R_{z_l}R_{z_l}^*)=0 \quad (\forall l\in \{1,\dots,n\} \setminus \Lambda),
\end{equation}
\begin{equation}\label{c2}
(I_{\cls_{\Phi_{\Lambda}} }-E_{z_{i}}E_{z_{i}}^*)(I_{\cls_{\Phi_{\Lambda}} }-R_{z_{i}}R_{z_{i}}^*)(I_{\cls}-E_{z_{j}}E_{z_{j}}^*)(I_{\cls_{\Phi_{\Lambda}} }-R_{z_{j}}R_{z_{j}}^*)=0 \quad (\forall \text{ distinct }i, j \in \Lambda).
\end{equation}
Note that the first condition implies that $I_{\cls_{\Phi_{\Lambda}} }-R_{z_l}R_{z_l}^*-E_{z_l}E_{z_l}^* = 0$ for all $l \in \{1,\ldots,n\} \setminus \Lambda$, which further implies that
\[
[R_{z_j}^*, R_{z_l}](I_{\cls_{\Phi_{\Lambda}} }-R_{z_l}R_{z_l}^*-E_{z_l}E_{z_l}^*) = 0 \quad (\forall j \in \{1,\ldots,n\}).
\]
Hence, from Proposition \ref{square},we get for all $ l \in \{1,\ldots,n\} \setminus \Lambda$,
\[
[R_{z_j}^*, R_{z_l}]^2=0 \quad (\forall j \in \{1,\ldots,n\} \setminus l).
\]
So, we only need to show that $[R_{z_{{\lambda}_j}}^*, R_{z_{{\lambda}_i}}]^2=0$ for all distinct ${\lambda}_i,{\lambda}_j \in \Lambda$. Using Lemma \ref{1st factor}, the second condition (\ref{c2}), is again equivalent to the following,
\[
(I_{\cls_{\Phi_{\Lambda}} }-R_{z_{{\lambda}_i}}R_{z_{{\lambda}_i}}^*-E_{z_{{\lambda}_i}}E_{z_{{\lambda}_i}}^*)(I_{\cls_{\Phi_{\Lambda}} }-R_{z_{{\lambda}_j}}R_{z_{{\lambda}_j}}^*-E_{z_{{\lambda}_j}}E_{z_{{\lambda}_j}}^*)=0.
\]
From Lemma \ref{submodule factor} (ii),  we obtain that
\[
\begin{split}
&[R_{z_{\lambda_j}}^*, R_{z_{\lambda_i}}](I_{\cls_{\Phi_{\Lambda}} }-R_{z_{\lambda_i}}R_{z_{\lambda_i}}^*-E_{z_{\lambda_i}}E_{z_{\lambda_i}}^*)^{1/2}\\
&=B_{i,j} (I_{\cls_{\Phi_{\Lambda}} }-R_{z_{\lambda_j}}R_{z_{\lambda_j}}^*-E_{z_{\lambda_j}}E_{z_{\lambda_j}}^*)^{1/2}(I_{\cls_{\Phi_{\Lambda}} }-R_{z_{\lambda_i}}R_{z_{\lambda_i}}^*-E_{z_{\lambda_i}}E_{z_{\lambda_i}}^*)^{1/2}.
\end{split}
\]
Since $(I_{\cls_{\Phi_{\Lambda}} }-R_{z_{\lambda_i}}R_{z_{\lambda_i}}^*-E_{z_{\lambda_i}}E_{z_{\lambda_i}}^*)$ is a non-negative operator and $(I_{\cls_{\Phi_{\Lambda}} }-R_{z_{\lambda_i}} R_{z_{\lambda_i}}^*-E_{z_{\lambda_i}}E_{z_{\lambda_i}}^*)$ commutes with $(I_{\cls_{\Phi_{\Lambda}} }-R_{z_{\lambda_j}}R_{z_{\lambda_j}}^*-E_{z_{\lambda_j}}E_{z_{\lambda_j}}^*)$ it follows that
\[
(I_{\cls_{\Phi_{\Lambda}} }-R_{z_{\lambda_i}}R_{z_{\lambda_i}}^*-E_{z_{\lambda_i}}E_{z_{\lambda_i}}^*)(I_{\cls_{\Phi_{\Lambda}} }-R_{z_{\lambda_j}}R_{z_{\lambda_j}}^*-E_{z_{\lambda_j}}E_{z_{\lambda_j}}^*)=0,
\]
if and only if 
\[
(I_{\cls_{\Phi_{\Lambda}} }-R_{z_{\lambda_i}}R_{z_{\lambda_i}}^*-E_{z_{\lambda_i}}E_{z_{\lambda_i}}^*)^{1/2}(I_{\cls_{\Phi_{\Lambda}} }-R_{z_{\lambda_j}}R_{z_{\lambda_j}}^*-E_{z_{\lambda_j}}E_{z_{\lambda_j}}^*)^{1/2}=0.
\]
This implies that $[R_{z_{\lambda_j}}^*, R_{z_{\lambda_i}}](I_{\cls_{\Phi_{\Lambda}} }-R_{z_{\lambda_i}}R_{z_{\lambda_i}}^*-E_{z_{\lambda_i}}E_{z_{\lambda_i}}^*)^{1/2}=0$ for all distinct $\lambda_i, \lambda_j \in \Lambda$. Combining this result, along with the previous observation, we get $[R_{z_{j}}^*, R_{z_{i}}](I_{\cls_{\Phi_{\Lambda}} }-R_{z_{i}}R_{z_{i}}^*-E_{z_{i}}E_{z_{i}}^*)^{1/2}=0$ for all distinct $i, j \in \{1,\ldots,n\}$. This completes the proof.
\end{proof}

\begin{rem}
The condition $[R_{z_i}, R_{z_j}^*]^2 = 0$ is not sufficient. For instance, let $\theta(\bm{z}) = \frac{z_1 - \alpha_1}{1 - \bar{\alpha_1} z_1} \frac{z_2 - \alpha_2}{1 - \bar{\alpha_2} z_2}$, be an inner function corresponding to $\alpha_1, \alpha_2 \in \D$. Then if we consider the Beurling-type submodule $\cls = \theta(\bm{z}) H^2(\D^n)$, then we know that $[R_{z_i}, R_{z_j}^*] = 0$, and hence, $[R_{z_i}, R_{z_j}^*]^2=0$. However, $\theta(\bm{z})$ does not depend only on a single variable.
\end{rem}
It is natural question to ask whether condition (\ref{c2}) can be realized as a condition on certain compressions, say
\[
P_{\clm_i} (I_{\cls}-R_{z_i}R_{z_i}^*)(I_{\cls}-R_{z_j}R_{z_j}^*) P_{\clm_j} = 0,
\]
for all distinct $i,j \in \{1,\ldots,n\}$. Here $P_{\clm_i}$ is an orthogonal projection onto a closed subspace $\clm_i$ for $i=1,2$.  One way this is possible is when $E_i$ is a partial isometry for all $i \in \{ 1,\ldots,n \}$. Recall that an operator $V$ on a Hilbert space $\clh$ is said to be a partial isometry if $V$ is an isometry on $\clh \ominus \ker V$. It is well known that a partial isometry is characterized by the following equivalent conditions:\\
$(i)$ $VV^*$ is a projection.\\
$(ii)$ $VV^*V = V$.\\
We determine this behaviour of the evaluation operators in the following result.
\begin{thm}\label{partial}
Let $\cls_{\Phi_{\Lambda}} =  \underset{{\lambda \in \Lambda}}{\sum} \vp_{\lambda}(z_{\lambda}) H^2(\D^n)$ be a submodule of $H^2(\D^n)$ corresponding to non-constant inner functions $\{\vp_{\lambda}(\lambda)\}_{\lambda \in \Lambda}$. Then $E_{z_{\lambda}}$ is a partial isometry for any $\lambda \in \Lambda$ if and only if $\vp_{\lambda}(0) = 0$. Moreover, in this case, there exists a submodule $\cls_j$ of $H^2(\D^n)$ such that $I_{\cls_{\Phi_{\Lambda}}}-E_{z_{\lambda_j}}E_{z_{\lambda_j}}^*=P_{z_{\lambda_j}\cls_j}$, that is, the orthogonal projection onto the submodule $z_{\lambda_j}\cls_j$.
\end{thm}
\begin{proof}
Fix any $\lambda_j\in \Lambda = \{\lambda_1,\ldots, \lambda_k\}$, ordered increasingly. Define
\[
\clm_j=\sum_{\lambda \in \Lambda \setminus \{\lambda_j\}}\vp_{\lambda}(z_{\lambda})H^2(\D^n). 
\]
If we re-arrange $\cls_{\Phi_{\Lambda}}$ as the following
	\[
	\cls_{\Phi_{\Lambda}} = \clm_{j} + \vp_{\lambda_j}(\lambda_j) H^2(\D^n),
	\]
	then by the structure in condition (\ref{structure}), we get
	\[
\cls_{\Phi_{\Lambda}}=\vp_{\lambda_j}(z_{\lambda_j})H^2(\D^n)\oplus (I_{H^2(\D^n)} -M_{\vp_{\lambda_j}}M_{\vp_{\lambda_j}}^*)\clm_j. 
\]
So, $P_{\cls_{\Phi_{\Lambda}}}=M_{\vp_{\lambda_j}}M_{\vp_{\lambda_j}}^*+(I_{H^2(\D^n)} - M_{\vp_{\lambda_j}}M_{\vp_{\lambda_j}}^*)P_{\clm_j}$. From the structure of $\clm_j$, it is clear that 
\[
(I_{H^2(\D^n)} - M_{\vp_{\lambda_j}}M_{\vp_{\lambda_j}}^*) P_{\clm_j}  = P_{\clm_j} (I_{H^2(\D^n)} -M_{\vp_{\lambda_j}}M_{\vp_{\lambda_j}}^*).
\]
Suppose $\vp_{\lambda_j}(0)=0$, then there exists an inner function $\tilde{\vp}_{\lambda_j}$ depending only on $z_{\lambda_j}$-variable such that $\vp_{\lambda_j}=z_{\lambda_j}\tilde{\vp}_{\lambda_j}.$ Now, we have
\[
\begin{split}
	M_{\vp_{\lambda_j}}M_{\vp_{\lambda_j}}^*M_{z_{\lambda_j}}M_{z_{\lambda_j}}^*&=M_{z_{\lambda_j}}M_{\tilde{\vp}_{\lambda_j}}M_{z_{\lambda_j}}^*M_{\tilde{\vp}_{\lambda_j}}^*M_{z_{\lambda_j}}M_{z_{\lambda_j}}^*
	\\
	&=M_{z_{\lambda_j}}M_{\tilde{\vp}_{\lambda_j}}M_{z_{\lambda_j}}^*M_{\tilde{\vp}_{\lambda_j}}^*
	\\
	&=M_{\vp_{\lambda_j}}M_{\vp_{\lambda_j}}^*.
\end{split}
\]
So, $M_{\vp_{\lambda_j}}M_{\vp_{\lambda_j}}^*M_{z_{\lambda_j}}M_{z_{\lambda_j}}^* = M_{z_{\lambda_j}}M_{z_{\lambda_j}}^* M_{\vp_{\lambda_j}}M_{\vp_{\lambda_j}}^*= M_{\vp_{\lambda_j}}M_{\vp_{\lambda_j}}^*$. Also,
\[
\begin{split}
	(I_{H^2(\D^n)}  - M_{\vp_{\lambda_j}}M_{\vp_{\lambda_j}}^*) M_{z_{\lambda_j}}M_{z_{\lambda_j}}^* (I_{H^2(\D^n)}  - M_{\vp_{\lambda_j}}M_{\vp_{\lambda_j}}^*)&=M_{z_{\lambda_j}}M_{z_{\lambda_j}}^* - M_{\vp_{\lambda_j}}M_{\vp_{\lambda_j}}^*
	\\
	&=M_{z_{\lambda_j}}M_{z_{\lambda_j}}^* - M_{z_{\lambda_j}}M_{\tilde{\vp}_{\lambda_j}}M_{\tilde{\vp}_{\lambda_j}}^*M_{z_{\lambda_j}}^*
	\\
	&=M_{z_{\lambda_j}}(I_{H^2(\D^n)} -M_{\tilde{\vp}_{\lambda_j}}M_{\tilde{\vp}_{\lambda_j}}^*)M_{z_{\lambda_j}}^*.
\end{split}
\]
Define $\cls_{\lambda_j}: =\sum_{\lambda \in \Lambda \setminus \{\lambda_j\}}\vp_{\lambda}(z_\lambda)H^2(\D^n) + \tilde{\vp}_{\lambda_j}H^2(\D^n)$. Thus, we can write 
\[
\cls_{\lambda_j}= \tilde{\vp}_{\lambda_j}(z_{\lambda_j})H^2(\D^n)\oplus (I_{H^2(\D^n)} -M_{\tilde{\vp}_{\lambda_j}}M_{\tilde{\vp}_{\lambda_j}}^*)\clm_j, 
\]
Hence, $\cls_{\lambda_j}$ is a submodule of $H^2(\D^n)$, and moreover, $z_{\lambda_j}\cls_{\lambda_j}\subseteq \cls_{\Phi_{\Lambda}}$. Now
\[
\begin{split}
	&P_{\cls_{\Phi_{\Lambda}}}M_{z_{\lambda_j}}M_{z_{\lambda_j}}^*P_{\cls_{\Phi_{\Lambda}}}\\
	&=\big(M_{\vp_{\lambda_j}}M_{\vp_{\lambda_j}}^*+(I_{H^2(\D^n)} -M_{\vp_{\lambda_j}}M_{\vp_{\lambda_j}}^* )P_{\clm_j} \big) M_{z_{\lambda_j}}M_{z_{\lambda_j}}^* \big( M_{\vp_{\lambda_j}}M_{\vp_{\lambda_j}}^*+(I_{H^2(\D^n)} -M_{\vp_{\lambda_j}}M_{\vp_{\lambda_j}}^*)P_{\clm_j} \big)
	\\
	&=M_{\vp_{\lambda_j}}M_{\vp_{\lambda_j}}^*+(I_{H^2(\D^n)} -M_{\vp_{\lambda_j}}M_{\vp_{\lambda_j}}^*) P_{\clm_j} M_{z_{\lambda_j}}M_{z_{\lambda_j}}^*(I_{H^2(\D^n)} -M_{\vp_{\lambda_j}}M_{\vp_{\lambda_j}}^*)P_{\clm_j}
	\\
	&=M_{\vp_{\lambda_j}}M_{\vp_{\lambda_j}}^*+ (I_{H^2(\D^n)} -M_{\vp_{\lambda_j}}M_{\vp_{\lambda_j}}^*)  M_{z_{\lambda_j}}M_{z_{\lambda_j}}^*(I_{H^2(\D^n)} -M_{\vp_{\lambda_j}}M_{\vp_{\lambda_j}}^*)P_{\clm_j}
	\\
	&=M_{z_{\lambda_j}} M_{\tilde{\vp}_{\lambda_j}}M_{\tilde{\vp}_{\lambda_j}}^* M_{z_{\lambda_j}}^* + M_{z_{\lambda_j}}(I_{H^2(\D^n)} -M_{\tilde{\vp}_{\lambda_j}}M_{\tilde{\vp}_{\lambda_j}}^*)M_{z_{\lambda_j}}^*P_{\clm_j}
	\\&
	=M_{z_{\lambda_j}} (M_{\tilde{\vp}_{\lambda_j}}M_{\tilde{\vp}_{\lambda_j}}^*+(I_{H^2(\D^n)} -M_{\tilde{\vp}_{\lambda_j}}M_{\tilde{\vp}_{\lambda_j}}^*)P_{\clm_j})M_{z_{\lambda_j}}^*
	\\
	&=M_{z_{\lambda_j}}  P_{\cls_{\lambda_j}} M_{z_{\lambda_j}}^*
	\\
	&=P_{z_{\lambda_j}\cls_{\lambda_j}}.
\end{split}
\]
The last equality follows from Lemma \ref{projection}. Thus, we have
\[
I_{\cls} -E_{z_{\lambda_j}}E_{z_{\lambda_j}}^*=P_{\cls_{\Phi_{\Lambda}}}M_{z_{\lambda_j}}M_{z_{\lambda_j}}^*|_{\cls_{\Phi_{\Lambda}}}=P_{z_{\lambda_j}\cls_{\lambda_j}} 
\]
is an orthogonal projection. Conversely, suppose $I_{\cls} -E_{z_{\lambda_j}}E_{z_{\lambda_j}}^*$ is a projection, then
\[
\begin{split}
P_{\cls_{\Phi_{\Lambda}}}M_{z_{\lambda_j}}M_{z_{\lambda_j}}^*P_{\cls_{\Phi_{\Lambda}}}&=P_{\cls_{\Phi_{\Lambda}}}M_{z_{\lambda_j}}M_{z_{\lambda_j}}^*P_{\cls_{\Phi_{\Lambda}}}M_{z_{\lambda_j}}M_{z_{\lambda_j}}^*P_{\cls_{\Phi_{\Lambda}}}
\\
&=P_{\cls_{\Phi_{\Lambda}}}M_{z_{\lambda_j}}M_{z_{\lambda_j}}^* (I_{H^2(\D^n)} -P_{\cls_{\Phi_{\Lambda}}}^{\perp})M_{z_{\lambda_j}}M_{z_{\lambda_j}}^*P_{\cls_{\Phi_{\Lambda}}}
\\
&=P_{\cls_{\Phi_{\Lambda}}}M_{z_{\lambda_j}}M_{z_{\lambda_j}}^*P_{\cls_{\Phi_{\Lambda}}}-P_{\cls_{\Phi_{\Lambda}}}M_{z_{\lambda_j}}M_{z_{\lambda_j}}^*P_{\cls_{\Phi_{\Lambda}}}^{\perp} M_{z_{\lambda_j}}M_{z_{\lambda_j}}^*P_{\cls_{\Phi_{\Lambda}}}.
\end{split}
\]
Thus, we have $P_{\cls_{\Phi_{\Lambda}}}M_{z_{\lambda_j}}M_{z_{\lambda_j}}^*P_{\cls_{\Phi_{\Lambda}}}^{\perp} M_{z_{\lambda_j}}M_{z_{\lambda_j}}^*P_{\cls_{\Phi_{\Lambda}}}.=0$. Equivalently, $P_{\cls_{\Phi_{\Lambda}}}M_{z_{\lambda_j}}M_{z_{\lambda_j}}^*P_{\cls_{\Phi_{\Lambda}}}^{\perp}=0$. So,
\[
P_{\cls_{\Phi_{\Lambda}}}M_{z_{\lambda_j}}M_{z_{\lambda_j}}^*=P_{\cls_{\Phi_{\Lambda}}}M_{z_{\lambda_j}}M_{z_{\lambda_j}}^*P_{\cls_{\Phi_{\Lambda}}}=M_{z_{\lambda_j}}M_{z_{\lambda_j}}^* P_{\cls_{\Phi_{\Lambda}}}.
\]
Then
\begin{multline}
(M_{{\vp}_{\lambda_j}}M_{{\vp}_{\lambda_j}}^* +(I_{H^2(\D^n)} -M_{{\vp}_{\lambda_j}}M_{{\vp}_{\lambda_j}}^* )P_{\clm_j})M_{z_{\lambda_j}}M_{z_{\lambda_j}}^* \\
=M_{z_{\lambda_j}}M_{z_{\lambda_j}}^*(M_{{\vp}_{\lambda_j}}M_{{\vp}_{\lambda_j}}^* +(I_{H^2(\D^n)} -M_{{\vp}_{\lambda_j}}M_{{\vp}_{\lambda_j}}^* )P_{\clm_j}).
\end{multline}
This implies,
\[
\begin{split}
0&=[M_{{\vp}_{\lambda_j}}M_{{\vp}_{\lambda_j}}^* , M_{z_{\lambda_j}}M_{z_{\lambda_j}}^*]+[(I_{H^2(\D^n)} -M_{{\vp}_{\lambda_j}}M_{{\vp}_{\lambda_j}}^* ), M_{z_{\lambda_j}}M_{z_{\lambda_j}}^*]P_{\clm_j}
\\
&=[M_{{\vp}_{\lambda_j}}M_{{\vp}_{\lambda_j}}^*  M_{z_{\lambda_j}}M_{z_{\lambda_j}}^*]-[M_{{\vp}_{\lambda_j}}M_{{\vp}_{\lambda_j}}^*, M_{z_{\lambda_j}}M_{z_{\lambda_j}}^*]P_{\clm_j}
\\
&=[M_{{\vp}_{\lambda_j}}M_{{\vp}_{\lambda_j}}^* M_{z_{\lambda_j}}M_{z_{\lambda_j}}^*]P_{\clm_j^{\perp}}.
\end{split}
\]
Thus $[M_{{\vp}_{\lambda_j}}M_{{\vp}_{\lambda_j}}^* , M_{z_{\lambda_j}}M_{z_{\lambda_j}}^*]=0$ or $P_{\clm_j^{\perp}}=0$.
Note that $\clm_j^{\perp} = \underset{\lambda \in \Lambda \setminus \{\lambda_j\}}{\otimes} (H^2(\D)\ominus \vp_{\lambda} H^2(\D))$. Since all $\vp_{\lambda_i}$ are non-constant inner functions, we have $P_{\clm_j^{\perp}}\neq 0$. 

Consequently, $[M_{{\vp}_{\lambda_j}}M_{{\vp}_{\lambda_j}}^*, M_{z_{\lambda_j}}M_{z_{\lambda_j}}^*]=0$. Using \cite[Theorem 2.2]{DPS}, we get that there must exist an inner function $ \tilde{\vp}_j $
depending only on $z_{\lambda_j}$-variable such that $\vp_{\lambda_j}=z_{\lambda_j} \tilde{\vp}_{\lambda_j}$. In this case, we have $I_{\cls_{\Phi_{\Lambda}}}-E_{z_{\lambda_j}}E_{z_{\lambda_j}}^*=P_{z_{\lambda_j}\cls_{\lambda_j}}$. Where
\[
\cls_{\lambda_j}= \sum_{\lambda \in \Lambda; \lambda \neq \lambda_j} \vp_{\lambda}(z_{\lambda})H^2(\D^n) + \tilde{\vp}_{\lambda_j} H^2(\D^n).
\]
Clearly, $\cls_j$ is a submodule of $ H^2(\D^n)$, amd moreover,  $z_{\lambda_j} \cls_{\lambda_j} \subseteq \cls_{\vp_\Lambda}$.  This completes the proof.
\end{proof}

\begin{cor}
Let $\cls_{\Phi_{\Lambda}}= \sum_{\lambda \in \Lambda}\vp_\lambda(z_\lambda)H^2(\D^n)$ be a submodule of $H^2(\D^n)$, where $\vp_\lambda(z_\lambda)$ are non-constant inner functions. Then $E_{z_{\lambda}}$ is a partial isometry for all $\lambda \in \Lambda$, if and only if $1\in \clq_{\Phi_{\Lambda}}$.
\end{cor}
\begin{proof}
	Suppose that $1\in \cls_{\Phi_{\Lambda}}^{\perp}$. Since
	\[
	\cls_{\Phi_{\Lambda}}^{\perp}=\otimes_{\lambda \in \Lambda} (H^2(\D)\ominus \vp_\lambda(z_\lambda)H^2(\D)).
	\]
	So, $1\in H^2(\D)\ominus \vp_\lambda(z_\lambda)H^2(\D)$ for all $\lambda \in \Lambda$. Thus for each $\lambda \in \Lambda$ there exists an inner function $\tilde{\vp}_\lambda(z_\lambda)$ depending only on the $z_\lambda$-variable such that
	\[
	\vp_\lambda=z_\lambda \tilde{\vp}_\lambda.
	\] 
	Now by the previous Theorem \ref{partial}, we conclude that $E_{z_\lambda}$ is a partial isometry for all $\lambda \in \Lambda$. On the other hand, if $E_{z_\lambda}$ is a partial isometry for all $\lambda \in \Lambda$, then by previous Theorem $\vp_\lambda(0)=0$. Therefore $1\in H^2(\D)\ominus \vp_\lambda(z_\lambda)H^2(\D)$ for all $\lambda \in \Lambda$. Consequently, $1\in \cls_{\Phi_{\Lambda}}^{\perp}$. This completes the proof.
\end{proof}

\newsection{Commuting defects of restriction operators}\label{defects}

Let $\cls$ be a submodule of $H^2(\D^n)$, then the collection of the corresponding restrictions $(R_{z_1},\ldots,R_{z_n})$ is a tuple of isometries.  Thus, we have a natural collection of orthogonal projections associated with $\cls$, namely, $$(I_{\cls} - R_{z_1}R_{z_1}^*, \ldots, I_{\cls} - R_{z_n}R_{z_n}^*).$$
It is a natural question to ask when these projections commute with each other. Whenever we have this property, we can write 
\[
\cls \ominus \sum_{i=1}^n z_i \cls = \mbox{ran }\underset{i=1} {\overset{n}{\Pi}}(I_{\cls} - R_{z_i}R_{z_i}^*).
\]
Since the \textit{wandering} subspace (the subspace on the left-hand side) contains a significant amount of information, this identity can serve as an important tool in the analysis of submodules. Due to the important advantage that it carries, this problem is substantially difficult for general submodules. In this section, motivated by our earlier results, we study this question for the following class of submodules of $H^2(\D^n)$ 
\[
\cls_{\Phi_{\Lambda}} = \underset{\lambda \in \Lambda}{\sum} \vp_\lambda(z_\lambda) H^2(\D^n).
\]

Let $\Lambda = \{\lambda_1, \ldots, \lambda_k\} \subseteq \{1,\ldots,n\}$. Before going into the proof in detail, let us highlight a re-structuring of the submodule $\cls_{\Phi_{\Lambda}}$ that will be useful in the sequel.  Let $n\geq 2$, and let us fix distinct $\lambda_i,\lambda_j\in \Lambda$. Define
	\[
	\clm_{i,j}:=\sum_{\lambda \in \Lambda; \lambda \neq \lambda_i,\lambda_j}\vp_\lambda(z_\lambda)H^2(\D^n). 
	\] 
	It is straightforward to observe that 
	\[
	\clm_{i,j}^{\perp} = \underset{\lambda \in \Lambda; \lambda \neq \lambda_i,\lambda_j}{\otimes} \clq_{\vp_\lambda}.
	\]
	If we re-arrange $\cls_{\Phi_{\Lambda}}$ as the following
	\[
	\cls_{\Phi_{\Lambda}} = \clm_{i,j} + \vp_{\lambda_j}(z_{\lambda_j}) H^2(\D^n) + \vp_{\lambda_i}(z_{\lambda_i}) H^2(\D^n),
	\]
	then by the structure in condition (\ref{structure}), we get
	\[
	\begin{split}
	\cls_{\Phi_{\Lambda}}&=\vp_{\lambda_i}(z_{\lambda_i})H^2(\D^n)\oplus (I_{\cls_{\Phi_{\Lambda}}}-M_{\vp_{\lambda_i}}M_{\vp_{\lambda_i}}^*)\vp_{\lambda_j}(z_{\lambda_j})H^2(\D^n)\\
	&\oplus (I_{\cls_{\Phi_{\Lambda}}}-M_{\vp_{\lambda_i}}M_{\vp_{\lambda_i}}^*)(I_{\cls_{\Phi_{\Lambda}}}-M_{\vp_{\lambda_j}}M_{\vp_{\lambda_j}}^*)\clm_{i,j}.
	\end{split}
	\]
	and in terms of projections we get,
	\begin{equation}\label{projn2}
	\begin{split}
	P_{\cls_{\Phi_{\Lambda}}}&=M_{\vp_{\lambda_i}}M_{\vp_{\lambda_i}}^*\oplus (I_{\cls_{\Phi_{\Lambda}}}-M_{\vp_{\lambda_i}}M_{\vp_{\lambda_i}}^*)M_{\vp_{\lambda_j}}M_{\vp_{\lambda_j}}^*\\ 
	&\oplus (I_{\cls_{\Phi_{\Lambda}}}-M_{\vp_{\lambda_i}}M_{\vp_{\lambda_i}}^*)(I_{\cls{\vp_{\Lambda}}}-M_{\vp_{\lambda_j}}M_{\vp_{\lambda_j}}^*)P_{\clm_{i,j}}.
	\end{split}
	\end{equation}
Note that, for any $\lambda \in \Lambda$, we can again write
\[
P_{\cls_{\Phi_{\Lambda}}} = M_{\vp_\lambda} M_{\vp_\lambda}^* \oplus (I_{H^2(\D^n)} - M_{\vp_\lambda} M_{\vp_\lambda}^*) \sum_{t \in \Lambda \setminus \{\lambda\}} M_{\vp_{t}}M_{\vp_{t}}^*.
\]
Using this identity we can conclude that for any $\lambda \in \Lambda$,
	\begin{equation}\label{adj_eqn}
	M_{\vp_{\lambda}}^*P_{\cls_{\Phi_{\Lambda}}}=M_{\vp_{\lambda}}^*\quad \text{and hence, }\quad 	P_{\cls_{\Phi_{\Lambda}}}M_{\vp_{\lambda}}=M_{\vp_{\lambda}}.
  \end{equation}

\begin{proof}[Proof of Theorem \ref{commuting defects}]
Let us now begin the proof by assuming that 
\[
(I_{\cls_{\Phi_{\Lambda}}} - R_{z_{\lambda_1}}R_{z_{\lambda_1}}^*, \ldots, I_{\cls_{\Phi_{\Lambda}}} - R_{z_{\lambda_n}}R_{z_{\lambda_k}}^*),
\]
forms a commuting tuple. By expanding the  individual defect operators, we see that this condition turns into
\begin{equation}\label{comm}
\big(M_{z_{\lambda_i}} P_{\cls_{\Phi_{\Lambda}}} M_{z_{\lambda_i}}^* \big) \big( M_{z_{\lambda_j}} P_{\cls_{\Phi_{\Lambda}}} M_{z_{\lambda_j}}^* \big) = \big( M_{z_{\lambda_j}} P_{\cls_{\Phi_{\Lambda}}} M_{z_{\lambda_j}}^* \big) \big(M_{z_{\lambda_i}} P_{\cls_{\Phi_{\Lambda}}} M_{z_{\lambda_i}}^* \big)
\end{equation}
for all distinct $\lambda_i, \lambda_j \in \Lambda$. 
	This implies that
	\[
\begin{split}
M_{\vp_{\lambda_i}}^*(M_{z_{\lambda_i}}P_{\cls_{\Phi_{\Lambda}}}M_{z_{\lambda_i}}^*)(M_{z_{\lambda_j}}P_{\cls_{\Phi_{\Lambda}}}M_{z_{{\lambda_j}}}^*)	M_{\vp_{\lambda_i}}
&=M_{\vp_{\lambda_i}}^*(M_{z_{\lambda_i}}P_{\cls_{\Phi_{\Lambda}}}M_{z_{\lambda_i}}^*)(M_{z_{\lambda_j}}P_{\cls_{\Phi_{\Lambda}}} M_{\vp_{\lambda_i}})M_{z_{{\lambda_j}}}^*
\\
&=M_{\vp_{\lambda_i}}^*(M_{z_{\lambda_i}}P_{\cls_{\Phi_{\Lambda}}}M_{z_{\lambda_i}}^*)M_{\vp_{\lambda_i}}(M_{z_{{\lambda_j}}} M_{z_{{\lambda_j}}}^*).
\end{split}
	\]
	Also, by using the commutativity in condition (\ref{comm}), we get
	\begin{align*}
	M_{\vp_{\lambda_i}}^*(M_{z_{\lambda_i}}P_{\cls_{\Phi_{\Lambda}}}M_{z_{\lambda_i}}^*)(M_{z_{\lambda_j}}P_{\cls_{\Phi_{\Lambda}}}M_{z_{{\lambda_j}}}^*)	M_{\vp_{\lambda_i}} 
	&= M_{\vp_{\lambda_i}}^* (M_{z_{\lambda_j}}P_{\cls_{\Phi_{\Lambda}}}M_{z_{{\lambda_j}}}^*)	 (M_{z_{\lambda_i}}P_{\cls_{\Phi_{\Lambda}}}M_{z_{\lambda_i}}^*)M_{\vp_{\lambda_i}} \\
	&= (M_{z_{{\lambda_j}}} M_{z_{{\lambda_j}}}^*) M_{\vp_{\lambda_i}}^*(M_{z_{\lambda_i}}P_{\cls_{\Phi_{\Lambda}}}M_{z_{\lambda_i}}^*)M_{\vp_{\lambda_i}}.
	\end{align*}
	Now
	\[
	\begin{split}
&M_{\vp_{\lambda_i}}^*(M_{z_{\lambda_i}}P_{\cls_{\Phi_{\Lambda}}}M_{z_{\lambda_i}}^*)M_{\vp_{\lambda_i}}\\
&=M_{\vp_{\lambda_i}}^*M_{z_{\lambda_i}}[M_{\vp_{\lambda_i}}M_{\vp_{\lambda_i}}^*+ (I_{H^2(\D^n)} -M_{\vp_{\lambda_i}}M_{\vp_{\lambda_i}}^*)M_{\vp_{\lambda_j}}M_{\vp_{\lambda_j}}^*
\\
& +(I_{H^2(\D^n)}-M_{\vp_{\lambda_i}}M_{\vp_{\lambda_i}}^*)(I_{H^2(\D^n)} -M_{\vp_{\lambda_j}}M_{\vp_{\lambda_j}}^*)P_{\clm_{i,j}}]M_{z_{\lambda_i}}^*M_{\vp_{\lambda_i}}
\\
&
=M_{z_{\lambda_i}}M_{z_{\lambda_i}}^*+M_{\vp_{\lambda_i}}^*M_{z_{\lambda_i}}(I_{H^2(\D^n)}-M_{\vp_{\lambda_i}}M_{\vp_{\lambda_i}}^*)M_{z_{\lambda_i}}^*M_{\vp_{\lambda_i}}M_{\vp_{\lambda_j}}M_{\vp_{\lambda_j}}^*
\\
&+M_{\vp_{\lambda_i}}^*M_{z_{\lambda_i}}(I_{H^2(\D^n)}-M_{\vp_{\lambda_i}}M_{\vp_{\lambda_i}}^*)M_{z_{\lambda_i}}^*M_{\vp_{\lambda_i}}
(I_{H^2(\D^n)} - M_{\vp_{\lambda_j}}M_{\vp_{\lambda_j}}^*)P_{\clm_{i,j}}.
	\end{split}
	\]
	So, we have
	\[
	\begin{split}
&M_{\vp_{\lambda_i}}^*(M_{z_{\lambda_i}}P_{\cls_{\Phi_{\Lambda}}}M_{z_{\lambda_i}}^*)M_{\vp_{\lambda_i}}(M_{z_{\lambda_j}}M_{z_{\lambda_j}}^*)\\
&=(M_{z_{\lambda_i}}M_{z_{\lambda_i}}^*)(M_{z_{\lambda_j}}M_{z_{\lambda_j}}^*)
\\
&+M_{\vp_{\lambda_i}}^*M_{z_{\lambda_i}}(I_{H^2(\D^n)}-M_{\vp_{\lambda_i}}M_{\vp_{\lambda_i}}^*)M_{z_{\lambda_i}}^*M_{\vp_{\lambda_i}}M_{\vp_{\lambda_j}}M_{\vp_{\lambda_j}}^*(M_{z_{\lambda_j}}M_{z_{\lambda_j}}^*)
\\
&+M_{\vp_{\lambda_i}}^*M_{z_{\lambda_i}}(I_{H^2(\D^n)}-M_{\vp_{\lambda_i}}M_{\vp_{\lambda_i}}^*)M_{z_{\lambda_i}}^*M_{\vp_{\lambda_i}}
(I_{H^2(\D^n)}-M_{\vp_{\lambda_j}}M_{\vp_{\lambda_j}}^*)(M_{z_{\lambda_j}}M_{z_{\lambda_j}}^*)P_{\clm_{i,j}}
	\end{split}
	\]
	Now, let us consider the following difference.
\[	\begin{split}
&M_{\vp_{\lambda_i}}^*(M_{z_{\lambda_i}}P_{\cls_{\Phi_{\Lambda}}}M_{z_{\lambda_i}}^*)(M_{z_{\lambda_j}}P_{\cls_{\Phi_{\Lambda}}}M_{z_{\lambda_j}}^*)	M_{\vp_{\lambda_i}} - M_{\vp_{\lambda_i}}^*(M_{z_{\lambda_j}}P_{\cls_{\Phi_{\Lambda}}}M_{z_{\lambda_j}}^*) (M_{z_{\lambda_i}}P_{\cls_{\Phi_{\Lambda}}}M_{z_{\lambda_i}}^*)M_{\vp_{\lambda_i}}
\\
&=M_{\vp_{\lambda_i}}^*(M_{z_{\lambda_i}}P_{\cls_{\Phi_{\Lambda}}}M_{z_{\lambda_i}}^*)M_{\vp_{\lambda_i}}(M_{z_{v_j}}M_{z_{\lambda_j}}^*)-
(M_{z_{\lambda_j}}M_{z_{\lambda_j}}^*)M_{\vp_{\lambda_i}}^*(M_{z_{\lambda_i}}P_{\cls_{\Phi_{\Lambda}}}M_{z_{\lambda_i}}^*)M_{\vp_{\lambda_i}}
\\
&=M_{\vp_{\lambda_i}}^*M_{z_{\lambda_i}}(I_{H^2(\D^n)}-M_{\vp_{\lambda_i}}M_{\vp_{\lambda_i}}^*)M_{z_{\lambda_i}}^*M_{\vp_{\lambda_i}}[M_{\vp_{\lambda_j}}M_{\vp_{\lambda_j}}^*,M_{z_{\lambda_j}}M_{z_{\lambda_j}}^*]
\\
&+M_{\vp_{\lambda_i}}^*M_{z_{\lambda_i}}(I-M_{\vp_{\lambda_i}}M_{\vp_{\lambda_i}}^*)M_{z_{\lambda_i}}^*M_{\vp_{\lambda_i}}
[(I-M_{\vp_{\lambda_j}}M_{\vp_{\lambda_j}}^*), M_{z_{\lambda_j}}M_{z_{\lambda_j}}^*]P_{\clm_{i,j}}
\\
&=M_{\vp_{\lambda_i}}^*M_{z_{\lambda_i}}(I_{H^2(\D^n)}-M_{\vp_{\lambda_i}}M_{\vp_{\lambda_i}}^*)M_{z_{\lambda_i}}^*M_{\vp_{\lambda_i}}[M_{\vp_{\lambda_j}}M_{\vp_{\lambda_j}}^*,M_{z_{\lambda_j}}M_{z_{\lambda_j}}^*]
\\
&-M_{\vp_{\lambda_i}}^*M_{z_{\lambda_i}}(I_{H^2(\D^n)}-M_{\vp_{\lambda_i}}M_{\vp_{\lambda_i}}^*)M_{z_{\lambda_i}}^*M_{\vp_{\lambda_i}}
[M_{\vp_{\lambda_j}}M_{\vp_{\lambda_j}}^*, M_{z_{\lambda_i}}M_{z_{\lambda_i}}^*]P_{\clm_{i,j}}
\\
&=M_{\vp_{\lambda_i}}^*M_{z_{\lambda_i}}(I_{H^2(\D^n)}-M_{\vp_{\lambda_i}}M_{\vp_{\lambda_i}}^*)M_{z_{\lambda_i}}^*M_{\vp_{\lambda_i}}
[M_{\vp_{\lambda_j}}M_{\vp_{\lambda_j}}^*, M_{z_{\lambda_i}}M_{z_{\lambda_i}}^*](I_{H^2(\D^n)}-P_{\clm_{i,j}})
\\
&=M_{\vp_{\lambda_i}}^*M_{z_{\lambda_i}}(I_{H^2(\D^n)}-M_{\vp_{\lambda_i}}M_{\vp_{\lambda_i}}^*)M_{z_{\lambda_i}}^*M_{\vp_{\lambda_i}}
[M_{\vp_{\lambda_j}}M_{\vp_{\lambda_j}}^*, M_{z_{\lambda_j}}M_{z_{\lambda_j}}^*]P_{\clm_{i,j}^{\perp}}.
	\end{split}
\]
From our assumption the above difference is zero, and hence, 
\[
\big( M_{\vp_{\lambda_i}}^*M_{z_{\lambda_i}}(I_{H^2(\D^n)}-M_{\vp_{\lambda_i}}M_{\vp_{\lambda_i}}^*)M_{z_{\lambda_i}}^*M_{\vp_{\lambda_i}} \big)
\big([M_{\vp_{\lambda_j}}M_{\vp_{\lambda_j}}^*, M_{z_{\lambda_j}}M_{z_{\lambda_j}}^*] \big) \big( P_{\clm_{i,j}^{\perp}} \big) = 0.
\]
Since the above terms inside different parentheses depend on disjoint set of variables, we can deduce that either $$M_{\vp_{\lambda_i}}^*M_{z_{\lambda_i}}(I_{H^2(\D^n)}-M_{\vp_{\lambda_i}}M_{\vp_{\lambda_i}}^*)M_{z_{\lambda_i}}^*M_{\vp_{\lambda_i}}=0,$$ or $[M_{\vp_{\lambda_j}}M_{\vp_{\lambda_j}}^*, M_{z_{\lambda_j}}M_{z_{\lambda_j}}^*]=0$, or $P_{\clm_{i,j}^{\perp}}=0$.
	Since $\vp_{\lambda}$'s are non-constant for all $\lambda \in \Lambda$, we get $P_{\clm_{i,j}^{\perp}}\neq0$.  Now if $M_{\vp_{\lambda_i}}^*M_{z_{\lambda_i}}(I_{H^2(\D^n)}-M_{\vp_{\lambda_i}}M_{\vp_{\lambda_i}}^*)M_{z_{\lambda_i}}^*M_{\vp_{\lambda_i}}=0$, then $M_{\vp_{\lambda_i}}^*M_{z_{\lambda_i}}M_{z_{\lambda_i}}^*M_{\vp_{\lambda_i}} = M_{z_{\lambda_i}}M_{z_{\lambda_i}}^*$ is an orthogonal projection and hence, $M_{\vp_{\lambda_i}}^*M_{z_{\lambda_i}}$ must be a partial isometry. This implies that
	\[
	M_{\vp_{\lambda_i}}^*M_{z_{\lambda_i}} M_{z_{\lambda_i}}^* M_{\vp_{\lambda_i}} M_{\vp_{\lambda_i}}^*M_{z_{\lambda_i}} = M_{\vp_{\lambda_i}}^*M_{z_{\lambda_i}}.
	\]
Using $I_{H^2(\D^n)} - M_{z_{\lambda_i}}M_{z_{\lambda_i}}^* = P_{\ker M_{z_{\lambda_i}}^*}$, we get
\[
M_{\vp_{\lambda_i}}^* M_{\vp_{\lambda_i}} M_{\vp_{\lambda_i}}^*M_{z_{\lambda_i}} - M_{\vp_{\lambda_i}}^* P_{\ker M_{z_{\lambda_i}}^*} M_{\vp_{\lambda_i}} M_{\vp_{\lambda_i}}^*M_{z_{\lambda_i}} = M_{\vp_{\lambda_i}}^*M_{z_{\lambda_i}},
\]
and hence,
\[
M_{\vp_{\lambda_i}}^* P_{\ker M_{z_{\lambda_i}}^*} M_{\vp_{\lambda_i}} M_{\vp_{\lambda_i}}^*M_{z_{\lambda_i}} = 0.
\]
This will again imply that $ P_{\ker M_{z_{\lambda_i}}^*} M_{\vp_{\lambda_i}} M_{\vp_{\lambda_i}}^*M_{z_{\lambda_i}} = 0$. Thus, we get
\[
M_{\vp_{\lambda_i}} M_{\vp_{\lambda_i}}^*M_{z_{\lambda_i}} = M_{z_{\lambda_i}} M_{z_{\lambda_i}}^*M_{\vp_{\lambda_i}} M_{\vp_{\lambda_i}}^*M_{z_{\lambda_i}}.
\]
By acting on the right of both sides by $M_{z_{\lambda_i}}^*$, we get that
\[
[M_{\vp_{\lambda_i}} M_{\vp_{\lambda_i}}^*, M_{z_{\lambda_i}}M_{z_{\lambda_i}}^*]=0.
\]
Using \cite[Theorem 2.2]{DPS} we deduce that there must exist an inner function $\tilde{\vp}_{\lambda_i}(z_{\lambda_i})$ such that $\vp_{\lambda_i}(z_{\lambda_i})=z_{\lambda_i}\tilde{\vp}_{\lambda_i}(z_{\lambda_i})$. Then
	\[
	M_{z_{\lambda_i}}M_{z_{\lambda_i}}^*=M_{\vp_{\lambda_i}}^*M_{z_{\lambda_i}}M_{z_{\lambda_i}}^*M_{\vp_{\lambda_i}}  =M_{\tilde{\vp}_{\lambda_i}}^* M_{z_{\lambda_i}}^* M_{z_{\lambda_i}}M_{z_{\lambda_i}}^* M_{z_{\lambda_i}} M_{\tilde{\vp}_{\lambda_i}} =I_{H^2(\D^n)},
	\]
which is a contradiction. Thus, the only possibility that remains is
	 \[
	 [M_{\vp_{\lambda_j}}M_{\vp_{\lambda_j}}^*, M_{z_{\lambda_j}}M_{z_{\lambda_j}}^*]=0.
	 \]
	 So, there must exists an an inner function $\tilde{\vp}_{\lambda_j}(z_{\lambda_j})$ such that $\vp_{\lambda_j}(z_{\lambda_j})=z_j\tilde{\vp}_{\lambda_j}(z_{\lambda_j})$, and hence, $\vp_{\lambda_j}(0)=0$. Next, we consider the difference
	 \[
	 0=M_{\vp_{\lambda_j}}^*(M_{z_{\lambda_i}}P_{\cls_{\Phi_{\Lambda}}}M_{z_{\lambda_i}}^*)(M_{z_{\lambda_j}}P_{\cls_{\Phi_{\Lambda}}}M_{z_{\lambda_j}}^*)	M_{\vp_{\lambda_j}}-M_{\vp_{\lambda_j}}^* (M_{z_{\lambda_j}}P_{\cls_{\Phi_{\Lambda}}}M_{z_{\lambda_j}}^*) (M_{z_{\lambda_i}}P_{\cls_{\Phi_{\Lambda}}}M_{z_{\lambda_i}}^*)M_{\vp_{\lambda_j}},
	 \]
and in a similar manner, we can conclude that $\vp_{\lambda_i}(0)=0$. Thus, $\vp_{\lambda_i}(0)=0$ and $\vp_{\lambda_j}(0)=0$. 
	 
Let us now look at the converse direction. Assume $\vp_{\lambda_i}(z_{\lambda_i})=z_{\lambda_i}\tilde{\vp}_{\lambda_i}(z_{\lambda_i})$ and $\vp_{\lambda_j}(z_{\lambda_j})=z_{\lambda_j}\tilde{\vp}_{\lambda_j}(z_{\lambda_j})$, then using the structure of projections in identity (\ref{projn2}), we get
	 \begin{align*}
	 P_{\cls_{\Phi_{\Lambda}}}&=M_{z_{\lambda_i} \tilde{\vp}_{\lambda_i}}M_{z_{\lambda_i}\tilde{\vp}_{\lambda_i}}^*\oplus (I_{\cls_{\Phi_{\Lambda}}}-M_{z_{\lambda_i}\tilde{\vp}_{\lambda_i}}M_{z_{\lambda_i} \tilde{\vp}_{\lambda_i}}^*)M_{z_{\lambda_j} \tilde{\vp}_{\lambda_j}}M_{z_{\lambda_j}\tilde{\vp}_{\lambda_j}}^*\\
&\oplus (I_{\cls_{\Phi_{\Lambda}}}-M_{z_{\lambda_i}\tilde{\vp}_{\lambda_i}}M_{z_{\lambda_i}\tilde{\vp}_{\lambda_i}}^*)(I_{\cls_{\Phi_{\Lambda}}}-M_{z_{\lambda_j}\tilde{\vp}_{\lambda_j}}M_{z_{\lambda_j}\tilde{\vp}_{\lambda_j}}^*)P_{\clm_{i,j}}.
   \end{align*}
	 So,
	 \[
\begin{split}
M_{z_{\lambda_i}}P_{\cls_{\Phi_{\Lambda}}}M_{z_{\lambda_i}}^* &= M_{z_{\lambda_i}^2 \tilde{\vp}_{\lambda_i}}M_{z_{\lambda_i}^2\tilde{\vp}_{\lambda_i}}^* \\
&\oplus
M_{z_{\lambda_i}}(I_{\cls_{\Phi_{\Lambda}}}-M_{z_{\lambda_i}\tilde{\vp}_{\lambda_i}}M_{z_{\lambda_i}\tilde{\vp}_{\lambda_i}}^*)M_{z_{\lambda_i}}^* M_{z_{\lambda_j}\tilde{\vp}_{\lambda_j}}M_{z_{\lambda_j}\tilde{\vp}_{\lambda_j}}^*
\\
&\oplus 
M_{z_{\lambda_i}}(I_{\cls_{\Phi_{\Lambda}}}-M_{z_{\lambda_i}\tilde{\vp}_{\lambda_i}}M_{z_{\lambda_i} \tilde{\vp}_{\lambda_i}}^*)M_{z_{\lambda_i}}^*
(I_{\cls_{\Phi_{\Lambda}}}-M_{z_{\lambda_j}\tilde{\vp}_{\lambda_j}}M_{z_{\lambda_j}\tilde{\vp}_{\lambda_j}}^*)P_{\clm_{i,j}},
\end{split}
	 \]
and
	 \[
\begin{split}
	 M_{z_{\lambda_j}}P_{\cls_{\Phi_{\Lambda}}}M_{z_{\lambda_j}}^* &=M_{z_{\lambda_i}\tilde{\vp}_{\lambda_i}}M_{z_{\lambda_i}\tilde{\vp}_{\lambda_i}}^*M_{z_{\lambda_j}}M_{z_{\lambda_j}}^*
	\oplus (I_{\cls_{\Phi_{\Lambda}}}-M_{z_{\lambda_i}\tilde{\vp}_{\lambda_i}}M_{z_{\lambda_i}\tilde{\vp}_{\lambda_i}}^*)M_{z_{\lambda_j}^2 \tilde{\vp}_{\lambda_j}}M_{z_{\lambda_j}^2 \tilde{\vp}_{\lambda_j}}^*
	 \\
	 &\oplus (I_{\cls_{\Phi_{\Lambda}}}-M_{z_{\lambda_i}\tilde{\vp}_{\lambda_i}}M_{z_{\lambda_i}\tilde{\vp}_{\lambda_i}}^*)M_{z_{\lambda_j}}(I_{\cls_{\Phi_{\Lambda}}}-M_{z_{\lambda_j}\tilde{\vp}_{\lambda_j}}M_{z_{\lambda_j}\tilde{\vp}_{\lambda_j}}^*)M_{z_{\lambda_j}}^*P_{\clm_{i,j}}.
\end{split}
	 \]
 Now let us look at the product $(M_{z_{\lambda_i}}P_{\cls_{\Phi_{\Lambda}}}M_{z_{\lambda_i}}^*)(M_{z_{\lambda_j}}P_{\cls_{\Phi_{\Lambda}}}M_{z_{\lambda_j}}^*)$ term wise. First, observe that
 \[
 \begin{split}
 M_{z_{\lambda_i}^2 \tilde{\vp}_{\lambda_i}}M_{z_{\lambda_i}^2\tilde{\vp}_{\lambda_i}}^* (M_{z_{\lambda_j}}P_{\cls_{\Phi_{\Lambda}}}M_{z_{\lambda_j}}^*)& = M_{z_{\lambda_i}^2 \tilde{\vp}_{\lambda_i}}M_{z_{\lambda_i}^2\tilde{\vp}_{\lambda_i}}^* M_{z_{\lambda_i}\tilde{\vp}_{\lambda_i}}M_{z_{\lambda_i}\tilde{\vp}_{\lambda_i}}^*M_{z_{\lambda_j}}M_{z_{\lambda_j}}^*\\
 & = M_{z_{\lambda_i}^2\tilde{\vp}_{\lambda_i}}M_{z_{\lambda_i}^2\tilde{\vp}_{\lambda_i}}^*M_{z_{\lambda_j}}M_{z_{\lambda_j}}^*.
 \end{split}
 \]
 as the other terms will vanish. Next, let us look at the term
\[
 \begin{split}
&M_{z_{\lambda_i}}(I_{\cls_{\Phi_{\Lambda}}}-M_{z_{\lambda_i}\tilde{\vp}_{\lambda_i}}M_{z_{\lambda_i}\tilde{\vp}_{\lambda_i}}^*)M_{z_{\lambda_i}}^* M_{z_{\lambda_j}\tilde{\vp}_{\lambda_j}}M_{z_{\lambda_j}\tilde{\vp}_{\lambda_j}}^* (M_{z_{\lambda_j}}P_{\cls_{\Phi_{\Lambda}}}M_{z_{\lambda_j}}^*)\\
&= M_{z_{\lambda_i}}(I_{\cls_{\Phi_{\Lambda}}}-M_{z_{\lambda_i}\tilde{\vp}_{\lambda_i}}M_{z_{\lambda_i}\tilde{\vp}_{\lambda_i}}^*)M_{z_{\lambda_i}}^* M_{z_{\lambda_j}\tilde{\vp}_{\lambda_j}}M_{z_{\lambda_j}\tilde{\vp}_{\lambda_j}}^* M_{z_{\lambda_i}\tilde{\vp}_{\lambda_i}}M_{z_{\lambda_i}\tilde{\vp}_{\lambda_i}}^*M_{z_{\lambda_j}}M_{z_{\lambda_j}}^*\\
&\oplus M_{z_{\lambda_i}}(I_{\cls_{\Phi_{\Lambda}}}-M_{z_{\lambda_i}\tilde{\vp}_{\lambda_i}}M_{z_{\lambda_i}\tilde{\vp}_{\lambda_i}}^*)M_{z_{\lambda_i}}^* M_{z_{\lambda_j}\tilde{\vp}_{\lambda_j}}M_{z_{\lambda_j}\tilde{\vp}_{\lambda_j}}^* (I_{\cls_{\Phi_{\Lambda}}}-M_{z_{\lambda_i}\tilde{\vp}_{\lambda_i}}M_{z_{\lambda_i}\tilde{\vp}_{\lambda_i}}^*)M_{z_{\lambda_j}^2 \tilde{\vp}_{\lambda_j}}M_{z_{\lambda_j}^2 \tilde{\vp}_{\lambda_j}}^*\\
&\oplus M_{z_{\lambda_i}}(I_{\cls_{\Phi_{\Lambda}}}-M_{z_{\lambda_i}\tilde{\vp}_{\lambda_i}}M_{z_{\lambda_i}\tilde{\vp}_{\lambda_i}}^*)M_{z_{\lambda_i}}^* M_{z_{\lambda_j}\tilde{\vp}_{\lambda_j}}M_{z_{\lambda_j}\tilde{\vp}_{\lambda_j}}^*  (I_{\cls_{\Phi_{\Lambda}}}-M_{z_{\lambda_i}\tilde{\vp}_{\lambda_i}}M_{z_{\lambda_i}\tilde{\vp}_{\lambda_i}}^*)M_{z_{\lambda_j}}\\
&(I_{\cls_{\Phi_{\Lambda}}}-M_{z_{\lambda_j}\tilde{\vp}_{\lambda_j}}M_{z_{\lambda_j}\tilde{\vp}_{\lambda_j}}^*)M_{z_{\lambda_j}}^*P_{\clm_{i,j}}.
\end{split}
\]
By re-arranging the terms we get
\[
 \begin{split}
&M_{z_{\lambda_i}}(I_{\cls_{\Phi_{\Lambda}}}-M_{z_{\lambda_i}\tilde{\vp}_{\lambda_i}}M_{z_{\lambda_i}\tilde{\vp}_{\lambda_i}}^*)M_{z_{\lambda_i}}^* M_{z_{\lambda_j}\tilde{\vp}_{\lambda_j}}M_{z_{\lambda_j}\tilde{\vp}_{\lambda_j}}^* (M_{z_{\lambda_j}}P_{\cls_{\Phi_{\Lambda}}}M_{z_{\lambda_j}}^*)\\
&= M_{z_{\lambda_i}}(I_{\cls_{\vp_{\Lambda}}}-M_{z_{\lambda_i}\tilde{\vp}_{\lambda_i}}M_{z_{\lambda_i}\tilde{\vp}_{\lambda_i}}^*) M_{\tilde{\vp}_{\lambda_i}}M_{z_{\lambda_i}\tilde{\vp}_{\lambda_i}}^*  M_{z_{\lambda_j}\tilde{\vp}_{\lambda_j}}M_{z_{\lambda_j}\tilde{\vp}_{\lambda_j}}^* \\
&\oplus M_{z_{\lambda_i}}(I_{\cls_{\vp_{\Lambda}}}-M_{z_{\lambda_i}\tilde{\vp}_{\lambda_i}}M_{z_{\lambda_i}\tilde{\vp}_{\lambda_i}}^*) (I_{\cls_{\Phi_{\Lambda}}}-M_{\tilde{\vp}_{\lambda_i}}M_{\tilde{\vp}_{\lambda_i}}^*) M_{z_{\lambda_i}}^*M_{z_{\lambda_j}^2 \tilde{\vp}_{\lambda_j}}M_{z_{\lambda_j}^2 \tilde{\vp}_{\lambda_j}}^*\\
&\oplus M_{z_{\lambda_i}}(I_{\cls_{\Phi_{\Lambda}}}-M_{z_{\lambda_i}\tilde{\vp}_{\lambda_i}}M_{z_{\lambda_i}\tilde{\vp}_{\lambda_i}}^*) (I_{\cls_{\Phi_{\Lambda}}}-M_{\tilde{\vp}_{\lambda_i}}M_{\tilde{\vp}_{\lambda_i}}^*) M_{z_{\lambda_i}}^* M_{z_{\lambda_j}\tilde{\vp}_{\lambda_j}} (I_{\cls_{\Phi_{\Lambda}}}- M_{z_{\lambda_j}}M_{z_{\lambda_j}}^*) M_{z_{\lambda_j} \tilde{\vp}_{\lambda_j}}^*P_{\clm_{i,j}}.
\end{split}
\]
Hence,
\[
\begin{split}
&M_{z_{\lambda_i}}(I_{\cls_{\Phi_{\Lambda}}}-M_{z_{\lambda_i}\tilde{\vp}_{\lambda_i}}M_{z_{\lambda_i}\tilde{\vp}_{\lambda_i}}^*)M_{z_{\lambda_i}}^* M_{z_{\lambda_j}\tilde{\vp}_{\lambda_j}}M_{z_{\lambda_j}\tilde{\vp}_{\lambda_j}}^* (M_{z_{\lambda_j}}P_{\cls_{\Phi_{\Lambda}}}M_{z_{\lambda_j}}^*)\\
&= M_{z_{\lambda_i} \tilde{\vp}_{\lambda_i}} (I_{\cls_{\Phi_{\Lambda}}}-M_{z_{\lambda_i}}M_{z_{\lambda_i}}^*) M_{z_{\lambda_i}\tilde{\vp}_{\lambda_i}}^*  M_{z_{\lambda_j}\tilde{\vp}_{\lambda_j}}M_{z_{\lambda_j}\tilde{\vp}_{\lambda_j}}^* \\
&\oplus M_{z_{\lambda_i}} (I_{\cls_{\Phi_{\Lambda}}}-M_{\tilde{\vp}_{\lambda_i}}M_{\tilde{\vp}_{\lambda_i}}^*) M_{z_{\lambda_i}}^*M_{z_{\lambda_j}^2 \tilde{\vp}_{\lambda_j}}M_{z_{\lambda_j}^2 \tilde{\vp}_{\lambda_j}}^* \\
&\oplus M_{z_{\lambda_i}} (I_{\cls_{\Phi_{\Lambda}}}-M_{\tilde{\vp}_{\lambda_i}}M_{\tilde{\vp}_{\lambda_i}}^*) M_{z_{\lambda_i}}^* M_{z_{\lambda_j}\tilde{\vp}_{\lambda_j}} (I_{\cls_{\Phi_{\Lambda}}}- M_{z_{\lambda_j}}M_{z_{\lambda_j}}^*) M_{z_{\lambda_j} \tilde{\vp}_{\lambda_j}}^*P_{\clm_{i,j}}.
\end{split}
\]
Now let us look at the final term, that is,
\[
\begin{split}
&M_{z_{\lambda_i}}(I_{\cls_{\Phi_{\Lambda}}}-M_{z_{\lambda_i}\tilde{\vp}_{\lambda_i}}M_{z_{\lambda_i} \tilde{\vp}_{\lambda_i}}^*)M_{z_{\lambda_i}}^*
(I_{\cls_{\Phi_{\Lambda}}}-M_{z_{\lambda_j}\tilde{\vp}_{\lambda_j}}M_{z_{\lambda_j}\tilde{\vp}_{\lambda_j}}^*)P_{\clm_{i,j}} (M_{z_{\lambda_j}}P_{\cls_{\Phi_{\Lambda}}}M_{z_{\lambda_j}}^*) \\
&= M_{z_{\lambda_i}}(I_{\cls_{\Phi_{\Lambda}}}-M_{z_{\lambda_i}\tilde{\vp}_{\lambda_i}}M_{z_{\lambda_i} \tilde{\vp}_{t_i}}^*)M_{z_{\lambda_i}}^*
(I_{\cls_{\Phi_{\Lambda}}}-M_{z_{\lambda_j}\tilde{\vp}_{\lambda_j}}M_{z_{\lambda_j}\tilde{\vp}_{\lambda_j}}^*)P_{\clm_{i,j}}  M_{z_{\lambda_i}\tilde{\vp}_{\lambda_i}}M_{z_{\lambda_i}\tilde{\vp}_{\lambda_i}}^*M_{z_{\lambda_j}}M_{z_{\lambda_j}}^*\\
&\oplus M_{z_{\lambda_i}}(I_{\cls_{\Phi_{\Lambda}}}-M_{z_{\lambda_i}\tilde{\vp}_{\lambda_i}}M_{z_{\lambda_i} \tilde{\vp}_{\lambda_i}}^*)M_{z_{\lambda_i}}^*
(I_{\cls_{\Phi_{\Lambda}}}-M_{z_{\lambda_j}\tilde{\vp}_{\lambda_j}}M_{z_{\lambda_j}\tilde{\vp}_{\lambda_j}}^*)P_{\clm_{i,j}}   (I_{\cls_{\Phi_{\Lambda}}}-M_{z_{\lambda_i}\tilde{\vp}_{\lambda_i}}M_{z_{\lambda_i}\tilde{\vp}_{\lambda_i}}^*)\\
&M_{z_{\lambda_j}^2 \tilde{\vp}_{\lambda_j}}M_{z_{\lambda_j}^2 \tilde{\vp}_{\lambda_j}}^*\\
&\oplus M_{z_{\lambda_i}}(I_{\cls_{\Phi_{\Lambda}}}-M_{z_{\lambda_i}\tilde{\vp}_{\lambda_i}}M_{z_{\lambda_i} \tilde{\vp}_{\lambda_i}}^*)M_{z_{\lambda_i}}^*
(I_{\cls_{\Phi_{\Lambda}}}-M_{z_{\lambda_j}\tilde{\vp}_{\lambda_j}}M_{z_{\lambda_j}\tilde{\vp}_{\lambda_j}}^*)P_{\clm_{i,j}}   (I_{\cls_{\Phi_{\Lambda}}}-M_{z_{\lambda_i}\tilde{\vp}_{\lambda_i}}M_{z_{\lambda_i}\tilde{\vp}_{\lambda_i}}^*)\\
&M_{z_{\lambda_j}}(I_{\cls_{\Phi_{\Lambda}}}-M_{z_{\lambda_j}\tilde{\vp}_{\lambda_j}}M_{z_{\lambda_j}\tilde{\vp}_{\lambda_j}}^*)M_{z_{\lambda_j}}^*P_{\clm_{i,j}}.
\end{split}
\]
Again, by re-arranging the terms we get
\[
\begin{split}
&M_{z_{\lambda_i}}(I_{\cls_{\Phi_{\Lambda}}}-M_{z_{\lambda_i}\tilde{\vp}_{\lambda_i}}M_{z_{\lambda_i} \tilde{\vp}_{\lambda_i}}^*)M_{z_{\lambda_i}}^*
(I_{\cls_{\Phi_{\Lambda}}}-M_{z_{\lambda_j}\tilde{\vp}_{\lambda_j}}M_{z_{\lambda_j}\tilde{\vp}_{\lambda_j}}^*)P_{\clm_{i,j}} (M_{z_{\lambda_j}}P_{\cls_{\Phi_{\Lambda}}}M_{z_{\lambda_j}}^*) \\
&= M_{z_{\lambda_i} \tilde{\vp}_{\lambda_i}} (I_{\cls_{\Phi_{\Lambda}}}-M_{z_{\lambda_i}}M_{z_{\lambda_i}}^*)M_{z_{\lambda_i}\tilde{\vp}_{\lambda_i}}^*
M_{z_{\lambda_j}} (I_{\cls_{\Phi_{\Lambda}}}-M_{\tilde{\vp}_{\lambda_j}}M_{\tilde{\vp}_{\lambda_j}}^*)  M_{z_{\lambda_j}}^* P_{\clm_{i,j}}\\
&\oplus 0 \\
&\oplus M_{z_{\lambda_i}} (I_{\cls_{\Phi_{\Lambda}}}-M_{\tilde{\vp}_{\lambda_i}}M_{\tilde{\vp}_{\lambda_i}}^*)M_{z_{\lambda_i}}^* M_{z_{\lambda_j}}
(I_{\cls_{\Phi_{\Lambda}}}-M_{\tilde{\vp}_{\lambda_j}}M_{\tilde{\vp}_{\lambda_j}}^*) M_{z_{\lambda_j}}^*P_{\clm_{i,j}}.
\end{split}
\]
Therefore, we get
\[
\begin{split}
&(M_{z_{\lambda_i}}P_{\cls_{\Phi_{\Lambda}}}M_{z_{\lambda_i}}^*)(M_{z_{\lambda_j}}P_{\cls_{\Phi_{\Lambda}}}M_{z_{\lambda_j}}^*)\\
&= M_{z_{\lambda_i}^2\tilde{\vp}_{\lambda_i}}M_{z_{\lambda_i}^2\tilde{\vp}_{\lambda_i}}^*M_{z_{\lambda_j}}M_{z_{\lambda_j}}^*\\
&\oplus M_{z_{\lambda_i} \tilde{\vp}_{\lambda_i}} (I_{\cls_{\Phi_{\Lambda}}}-M_{z_{\lambda_i}}M_{z_{\lambda_i}}^*) M_{z_{\lambda_i}\tilde{\vp}_{\lambda_i}}^*  M_{z_{\lambda_j}\tilde{\vp}_{\lambda_j}}M_{z_{\lambda_j}\tilde{\vp}_{\lambda_j}}^* \\
&\oplus M_{z_{\lambda_i}} (I_{\cls_{\Phi_{\Lambda}}}-M_{\tilde{\vp}_{\lambda_i}}M_{\tilde{\vp}_{\lambda_i}}^*) M_{z_{\lambda_i}}^*M_{z_{\lambda_j}^2 \tilde{\vp}_{\lambda_j}}M_{z_{\lambda_j}^2 \tilde{\vp}_{\lambda_j}}^*\\
&\oplus M_{z_{\lambda_i}} (I_{\cls_{\Phi_{\Lambda}}}-M_{\tilde{\vp}_{\lambda_i}}M_{\tilde{\vp}_{\lambda_i}}^*) M_{z_{\lambda_i}}^* M_{z_{\lambda_j}\tilde{\vp}_{\lambda_j}} (I_{\cls_{\Phi_{\Lambda}}}- M_{z_{\lambda_j}}M_{z_{\lambda_j}}^*) M_{z_{\lambda_j} \tilde{\vp}_{\lambda_j}}^*P_{\clm_{i,j}}\\
&\oplus M_{z_{\lambda_i} \tilde{\vp}_{\lambda_i}} (I_{\cls_{\Phi_{\Lambda}}}-M_{z_{\lambda_i}}M_{z_{\lambda_i}}^*)M_{z_{\lambda_i}\tilde{\vp}_{\lambda_i}}^*
M_{z_{\lambda_j}} (I_{\cls_{\Phi_{\Lambda}}}-M_{\tilde{\vp}_{\lambda_j}}M_{\tilde{\vp}_{\lambda_j}}^*)  M_{z_{\lambda_j}}^* P_{\clm_{i,j}}\\
&\oplus M_{z_{\lambda_i}} (I_{\cls_{\Phi_{\Lambda}}}-M_{\tilde{\vp}_{\lambda_i}}M_{\tilde{\vp}_{\lambda_i}}^*)M_{z_{\lambda_i}}^* M_{z_{\lambda_j}}
(I_{\cls_{\Phi_{\Lambda}}}-M_{\tilde{\vp}_{\lambda_j}}M_{\tilde{\vp}_{\lambda_j}}^*) M_{z_{\lambda_j}}^*P_{\clm_{i,j}}.
\end{split}
\]
We observe that the right-hand side of the above expression is self-adjoint and hence, $M_{z_{\lambda_i}}P_{\cls_{\Phi_{\Lambda}}}M_{z_{\lambda_i}}^*$ commutes with $M_{z_{\lambda_j}}P_{\cls_{\Phi_{\Lambda}}}M_{z_{\lambda_j}}^*$.  In other words,  $I_{\cls_{\Phi_{\Lambda}}}-R_{z_{\lambda_i}}R_{z_{\lambda_i}}^* $ commutes with $ I_{\cls_{\Phi_{\Lambda}}}-R_{z_{\lambda_j}}R_{z_{\lambda_j}}^*$. This completes the proof.
\end{proof}

\begin{cor}
Let $\cls_{\Phi_{\Lambda}} = \sum_{\lambda \in \Lambda} \vp_\lambda(z_\lambda) H^2(\D^n)$ be a submodule of $H^2(\D^n)$ where $\vp_{t}$ are non-constant inner functions, and $\Lambda = \{\lambda_1,\ldots, \lambda_k\}$.  Then the following tuple of orthogonal projections
\[
\big( I_{\cls_{\Phi_{\Lambda}}}-R_{z_{\lambda_1}}R_{z_{\lambda_i}}^*,\ldots, I_{\cls_{\Phi_{\Lambda}}}-R_{z_{\lambda_k}}R_{z_{\lambda_k}}^* \big)
\]
 is commuting  if and only if $\vp_{\lambda}(0)=0$ for all $\lambda \in \Lambda$.
\end{cor}

Before ending this section, let us highlight an example with explicit computations.

\begin{ex} Let $0<|\alpha|<1$ and let $b_{\alpha}(z)=\frac{z-\alpha}{1-\overline{\alpha}z}$, $z\in \D$ be a single Blaschke function.
Let us consider the submodule $\cls=b_{\alpha}(z_1)H^2(\D^2)+b_{\alpha}(z_2)H^2(\D^2) \subseteq H^2(\D^2)$ and choose an element  $f(\bm{z}) :=[b_{\alpha}(z_1)+b_{\alpha}(z_2)]b_{\alpha}(z_2)$ in $\cls$. First, observe that 
\[
M_{z}^*b_{\alpha}(z)=(1-|\alpha|^2) s_{\alpha}(z)  \quad \text{and }\quad M_z^*b_{\alpha}^2(z)=(1-|\alpha|^2)(b_{\alpha}(z)-\alpha) s_{\alpha}(z),
\]
where $s_{\alpha}(z)=(1-\overline{\alpha}z)^{-1} \in H^2(\D)$. It is well-known that the corresponding Szeg\"o kernel for $H^2(\D^2)$ is  $s_{(\alpha_1, \alpha_2)}(z_1, z_2) :=\underset{i=1}{\overset{2}{\Pi}} \frac{1}{1 - z_i \bar{\alpha}_i}$. Now, 
\[
\begin{split}
R_{z_1}R_{z_1}^*f=M_{z_1}P_{\cls}M_{z_1}^*f&=M_{z_1}P_{\cls}M_{z_1}^*b_{\alpha}(z_1)b_{\alpha}(z_2)
\\
&=M_{z_1}P_{\cls}b_{\alpha}(z_2)M_{z_1}^*b_{\alpha}(z_1)
\\
&=(1-|\alpha|^2)z_1s_{(\alpha,0)}(z_1, z_2)b_{\alpha}(z_2),
\end{split}
\]
Note that $P_{\cls}^{\perp} = (I_{H^2(\D^2)} - M_{b_{\alpha}(z_1)} M_{b_{\alpha}(z_1)}^*) (I_{H^2(\D^2)} - M_{b_{\alpha}(z_2)} M_{b_{\alpha}(z_2)}^*)$, and therefore
\[
\begin{split}
&P_{\cls}^{\perp} z_1 s_{(\alpha,0)}(z_1,z_2) s_{(0,\alpha)}(z_1,z_2) \\
&= (I_{H^2(\D^2)} - M_{b_{\alpha}(z_1)} M_{b_{\alpha}(z_1)}^*) (I_{H^2(\D^2)} - M_{b_{\alpha}(z_2)} M_{b_{\alpha}(z_2)}^*) z_1 s_{(\alpha,0)}(z_1,z_2) s_{(0,\alpha)}(z_1,z_2)\\
&= s_{(0,\alpha)}(z_1,z_2) (I_{H^2(\D^2)} - M_{b_{\alpha}(z_1)} M_{b_{\alpha}(z_1)}^*) z_1 s_{(\alpha,0)}(z_1,z_2)
\end{split}
\]
Now, for any $\lambda=(\lambda_1, \lambda_2) \in \D^2$ we get
\[
\begin{split}
 \big \langle  M_{b_{\alpha}(z_1)}^* z_1 s_{(\alpha,0)}(z_1,z_2), s_{\lambda}(z_1,z_2) \rangle &=  \langle   s_{(\alpha,0)}(z_1,z_2), M_{z_1}^* b_{\alpha}(z_1) s_{\lambda}(z_1,z_2) \rangle\\
&=   \big  \langle  s_{(\alpha,0)}(z_1, z_2),  \frac{b_{\alpha}(z_1) s_{\lambda}(z_1,z_2) + \alpha s_{\lambda}(0,z_2) }{z_1}  \big  \rangle\\
&= 1
\end{split}
\]
which gives $ M_{b_{\alpha}(z_1)}^* z_1 s_{(\alpha,0)}(z_1,z_2) = 1$, and therefore,
\[
(I_{H^2(\D^2)} - M_{b_{\alpha}(z_1)} M_{b_{\alpha}(z_1)}^*) z_1 s_{(\alpha,0)}(z_1,z_2) = z_1 s_{(\alpha,0)}(z_1,z_2) - b_{\alpha}(z_1) = \alpha s_{(\alpha,0)}(z_1,z_2)
\]
 Using this we get
\[
\begin{split}
&P_{\cls}  z_1 s_{(\alpha,0)}(z_1,z_2) s_{(0,\alpha)}(z_1,z_2) \\
&= z_1 s_{(\alpha,0)}(z_1,z_2) s_{(0,\alpha)}(z_1,z_2) - P_{\cls}^{\perp}  z_1 s_{(\alpha,0)}(z_1,z_2) s_{(0,\alpha)}(z_1,z_2) \\
&=  z_1 s_{(\alpha,0)}(z_1,z_2) s_{(0,\alpha)}(z_1,z_2) - \alpha s_{(\alpha,0)}(z_1,z_2) s_{(0,\alpha)}(z_1,z_2)\\
&=  (z_1 - \alpha) s_{(\alpha,0)}(z_1,z_2) s_{(0,\alpha)}(z_1,z_2)\\
&=  b_{\alpha}(z_1) s_{(0,\alpha)}(z_1,z_2).
\end{split}
\]
and thus,
\[
\begin{split}
R_{z_2}R_{z_2}^*R_{z_1}R_{z_1}^*f &=(1-|\alpha|^2)M_{z_2}P_{\cls}M_{z_2}^*z_1s_{(\alpha,0)}(z_1,z_2)b_{\alpha}(z_2) \\
&=(1-|\alpha|^2)^2 z_2 P_{\cls} z_1 s_{(\alpha,0)}(z_1,z_2) s_{(0,\alpha)}(z_1,z_2)\\
&=(1-|\alpha|^2)^2 z_2 b_{\alpha}(z_1) s_{(0,\alpha)}(z_1,z_2).
\end{split}
\]
Now, let us again note that
\[
P_{\cls}^{\perp}s_{(0,\alpha)}(z_1,z_2) = (I_{H^2(\D^2)} - M_{b_{\alpha}(z_1)} M_{b_{\alpha}(z_1)}^*) s_{(0,\alpha)}(z_1,z_2) = (1 + \bar{\alpha} b_{\alpha}(z_1)) s_{(0,\alpha)}(z_1,z_2),
\]
which gives $P_{\cls} s_{(0,\alpha)}(z_1,z_2) = -\bar{\alpha} b_{\alpha}(z_1)) s_{(0,\alpha)}(z_1,z_2)$. Using this, we get
\[
\begin{split}
R_{z_2}R_{z_2}^*f&=M_{z_2}P_{\cls}M_{z_2}^*f
\\&=M_{z_2}P_{\cls}M_{z_2}^*b_{\alpha}(z_1)b_{\alpha}(z_2)+M_{z_2}P_{\cls}M_{z_2}^*b_{\alpha}^2(z_2)
\\&=(1-|\alpha|^2)b_{\alpha}(z_1)z_2s_{(0,\alpha)}(z_1,z_2)+M_{z_2}P_{\cls}M_{z_2}^*b_{\alpha}^2(z_2)
\\& =(1-|\alpha|^2) z_2[b_{\alpha}(z_1) s_{(0,\alpha)}(z_1,z_2) + P_{\cls} (b_{\alpha}(z_2)-\alpha) s_{(0,\alpha)}(z_1,z_2)]
\\& =(1-|\alpha|^2) z_2 s_{(0,\alpha)}(z_1,z_2) [b_{\alpha}(z_1)+ b_{\alpha}(z_2) + |\alpha|^2 b_{\alpha}(z_1)]
\\& =(1-|\alpha|^2) z_2 s_{(0,\alpha)}(z_1,z_2) \big((1 + |\alpha|^2) b_{\alpha}(z_1)+ b_{\alpha}(z_2) \big),
\end{split}
\]
and hence,
\[
\begin{split}
R_{z_1}R_{z_1}^*R_{z_2}R_{z_2}^*f &= (1-|\alpha|^2) R_{z_1}R_{z_1}^* z_2 s_{(0,\alpha)}(z_1,z_2) \big((1 + |\alpha|^2) b_{\alpha}(z_1)+ b_{\alpha}(z_2) \big)\\
&= (1-|\alpha|^2) [ (1 + |\alpha|^2) M_{z_1}P_{\cls}M_{z_1}^*b_{\alpha}(z_1) z_2 s_{(0,\alpha)}(z_1,z_2)\\
&+ M_{z_1}P_{\cls}M_{z_1}^* z_2 b_{\alpha}(z_2)  s_{(0,\alpha)}(z_1,z_2)]
\\
&=(1-|\alpha|^2)^2 (1 + |\alpha|^2) z_1 P_{\cls} z_2 s_{(\alpha,0)}(z_1,z_2) s_{(0,\alpha)}(z_1,z_2).
\end{split}
\]
Similarly, as above $P_{\cls} z_2 s_{(\alpha,0)}(z_1,z_2) s_{(0,\alpha)}(z_1,z_2) = s_{(\alpha,0)}(z_1,z_2)b_{\alpha}(z_2)$, and hence,
\[
R_{z_1}R_{z_1}^*R_{z_2}R_{z_2}^*f  = (1-|\alpha|^2)^2 (1 + |\alpha|^2) z_1 s_{(\alpha,0)}(z_1,z_2)b_{\alpha}(z_2).
\]
Thus
\[
\begin{split}
&R_{z_1}R_{z_1}^*R_{z_2}R_{z_2}^*f-R_{z_2}R_{z_2}^*R_{z_1}R_{z_1}^*f \\
&= (1-|\alpha|^2)^2 [(1 + |\alpha|^2) z_1 s_{(\alpha,0)}(z_1,z_2)b_{\alpha}(z_2) - z_2 b_{\alpha}(z_1) s_{(0,\alpha)}(z_1,z_2)].
\end{split}
\]
Hence, $R_{z_1}R_{z_1}^*R_{z_2}R_{z_2}^*f = R_{z_2}R_{z_2}^*R_{z_1}R_{z_1}^*f$, will imply that $$(1 + |\alpha|^2) z_1 s_{(\alpha,0)}(z_1,z_2)b_{\alpha}(z_2) = z_2 b_{\alpha}(z_1) s_{(0,\alpha)}(z_1,z_2),$$ which is clearly not true.  Thus, we get that $[R_{z_1}R_{z_1}^*, R_{z_2}R_{z_2}^*] \neq 0$ on $\cls$.
\end{ex}

\newsection{Commuting projections}\label{commuting}

This section aims to give a complete characterization of commuting pairs of projections onto submodules $\cls_{\Phi_{\Lambda}}$, and $\cls_{\Psi_{\Gamma}}$ of $H^2(\D^n)$, where $\Lambda, \Gamma \subseteq \{1,\ldots,n\}$.  We will see that we can deduce this result by solving the question for the following submodules.
\[
\cls_{\Phi} = \sum_{i=1}^n \vp_i(z_i) H^2(\D^n); \quad  \cls_{\Psi} = \sum_{j=1}^n \psi_j(z_j) H^2(\D^n).
\]
It is worth reminding the readers that 
\[
P_{\cls_{\Phi}} =  \sum_{i=1}^n M_{\vp_{i}} M_{\vp_{i}}^* \big( \underset{k>i}{\Pi} (I_{H^2(\D^n)} - M_{\vp_{k}} M_{\vp_{k}}^*) \big)
\]
Let us now highlight a certain symmetry about these submodules. Note that for any $j \in \{1,\ldots,n\}$, the submodule $\cls_{\Phi}$ can again be realized as
\[
\cls_{\Phi,j} = \vp_1 H^2(\D^n) + \ldots +\vp_{j-1} H^2(\D^n) +\vp_n H^2(\D^n) + \vp_{j+1} H^2(\D^n) + \ldots + \vp_j H^2(\D^n).
\]
This is to say that the submodule remains the same even if we switch the $j$-th component with the $n$-th one. So, for any fixed but arbitrary $j \in \{1,\ldots,n\}$, we consider the tuple $(M_{\vp_1,j}, \ldots, M_{\vp_n,j})$, where
\[
M_{\vp_i,j} = M_{\vp_i} \quad (i \neq j, n)
\]
and
\[
M_{\vp_j,j} = M_{\vp_n}; \quad M_{\vp_n,j} = M_{\vp_j}. 
\]
For instance, if $j=5$, then
\[
(M_{\vp_1,5}, \ldots, M_{\vp_n,5}) = (M_{\vp_1},  \ldots, M_{\vp_4}, M_{\vp_n}, M_{\vp_6}, \ldots,, M_{\vp_{n-1}}, M_{\vp_5}).
\]
Using this re-arrangement, we can always re-write the projection $P_{\cls_{\Phi}}$ as follows. 
\[
P_{\cls_{\Phi}} = P_{\cls_{\Phi,j}} = \sum_{i=1}^n M_{\vp_i,j}M_{\vp_i,j}^* \underset{k>i}{\overset{n}{\Pi}} (I_{H^2(\D^n)} - M_{\vp_k,j}M_{\vp_k,j}^*).
\]

We can collect the above observations into the following result.

\begin{lem}\label{st}
Let $\cls=\cls_{\Phi}$ be a submodule of $H^2(\D^n)$. Then for any $j \in \{1,\ldots,n\}$
\[
P_{\cls_{\Phi}} = P_{\cls_{\Phi,j}} = \sum_{i=1}^n M_{\vp_i,j}M_{\vp_i,j}^* \underset{k>i}{\overset{n}{\Pi}} (I_{H^2(\D^n)} - M_{\vp_k,j}M_{\vp_k,j}^*) = \sum_{i=1}^n P_{\vp_i,j},
\]
where $P_{\vp_i,j} = M_{\vp_i,j}M_{\vp_i,j}^* \underset{k>i}{\overset{n}{\Pi}} (I_{H^2(\D^n)} - M_{\vp_k,j}M_{\vp_k,j}^*) $.
\end{lem}
The above result implies that the initial configuration is $P_{\cls_{\Phi,n}}$. Let us now collect some useful identities. For any $j \in \{1,\ldots,n\}$, let $P_{n,j} :=M_{\vp_{n,j}}^*M_{\psi_{n,j}} = M_{\vp_j}^* M_{\psi_j}$. 
\begin{equation}\label{condn1}
M_{\vp_{n,j}}^*(I_{H^2(\D^n)}-M_{\psi_{n,j}}M_{\psi_{n,j}}^*)M_{\vp_{n,j}}M_{\vp_{n,j}}^*M_{\psi_{n,j}}=(P_{n,j} -P_{n,j}P_{n,j}^*P_{n,j}),
\end{equation}
\begin{equation}\label{condn2}
M_{\vp_{n,j}}^*(I_{H^2(\D^n)}-M_{\psi_{n,j}}M_{\psi_{n,j}}^*)(I_{H^2(\D^n)}-M_{\vp_{n,j}}M_{\vp_{n,j}}^*)M_{\psi_{n,j}}=-(P_{n,j} -P_{n,j}P_{n,j}^*P_{n,j}).
\end{equation}
Let us now note a result that we will use to prove the main result of this section.
\begin{lem}\label{non-zero}
Let $\vp, \psi \in H^{\infty}(\D)$ be non-constant inner functions, then $$(I_{H^2(\D)} - M_{\vp} M_{\vp}^*) (I_{H^2(\D)} - M_{\psi} M_{\psi}^*) \neq 0.$$
\end{lem}
\begin{proof}
Note that the above product is
\[
(I_{H^2(\D)} - M_{\vp} M_{\vp}^*) (I_{H^2(\D)} - M_{\psi} M_{\psi}^*) = I_{H^2(\D)} - M_{\vp} M_{\vp}^* - M_{\psi} M_{\psi}^* + M_{\vp} M_{\vp}^*  M_{\psi} M_{\psi}^*.
\]
Therefore, the above product is zero if and only if $[M_{\vp} M_{\vp}^*, M_{\psi} M_{\psi}^*]=0$. From \cite[Theorem 2.2]{DPS}, this happens only when $\vp$ divides $\psi$, or $\psi$ divides $\vp$. Thus, we get,
\[
M_{\psi} M_{\psi}^* M_{\vp} M_{\vp}^*  = M_{\psi} M_{\psi}^* \quad \text{or} \quad    M_{\vp} M_{\vp}^* M_{\psi} M_{\psi}^* = M_{\vp} M_{\vp}^*, 
\]
respectively.  Therefore, the above product becomes either,
\[
I_{H^2(\D)} - M_{\vp} M_{\vp}^* \quad \text{or} \quad I_{H^2(\D)} - M_{\psi} M_{\psi}^*.
\]
Thus, $(I_{H^2(\D)} - M_{\vp} M_{\vp}^*) (I_{H^2(\D)} - M_{\psi} M_{\psi}^*) = 0$, will imply that either $I_{H^2(\D)} - M_{\vp} M_{\vp}^* =0$  or $ I_{H^2(\D)} - M_{\psi} M_{\psi}^* = 0$. Since $\vp$ and $\psi$ are non-constant inner functions, this cannot happen. This completes the proof.
\end{proof}
We are now ready to prove the main result of this section. 
\begin{thm}\label{main2}
Let $\cls_{\Phi}$ and $\cls_{\Psi}$ be the following submodules of $H^2(\D^n)$
\[
\cls_{\Phi} = \sum_{i=1}^n \vp_i(z_i) H^2(\D^n); \quad \cls_{\Psi} = \sum_{i=1}^n \psi_i(z_i) H^2(\D^n).
\]
corresponding to non-constant inner functions $\{\vp_i(z_i)\}_{i=1}^n$ and $\{\psi_i(z_i)\}_{i=1}^n$. Then $[P_{\cls_{\Phi}}, P_{\cls_{\Psi}}]=0$ if and only if either $\vp_i|\psi_i$ or $\psi_i|\vp_i$ for all $i\in \{1,\dots, n\}$.
\end{thm}

\begin{proof}
Let us begin by assuming the condition $[P_{\cls_{\Phi}}, P_{\cls_{\Psi}}]=0$. Note that this condition is equivalent to saying that $[P_{\cls_{\Phi,j}}, P_{\cls_{\Psi,j}}]=0$ for all $j \in \{1,\ldots,n\}$. It is because $\cls_{\Phi}, \cls_{\Psi}$ and $\cls_{\Phi,j}, \cls_{\Psi,j}$ are exactly the same submodules, respectively.  Since $\vp_j$ and $\psi_j$ are inner functions for all $j \in\{1,\ldots,n\}$ they are always non-zero.  For any fixed but arbitrary $j \in \{1,\ldots,n\}$, let us look at the configurations $\cls_{\Phi,j}$ and $\cls_{\Psi,j}$. First, observe that we can consider a decomposition of $\cls_{\Phi}, \cls_{\Psi}$ similar to condition (\ref{adj_eqn}) to get the following relations.
\[
M_{\vp_n,j}^*P_{\cls_{\Phi,j}}=M_{\vp_n,j}^* = M_{\vp_j}^* \quad \text{and}\quad P_{\cls_{\Psi,j}}M_{\psi_n,j}=M_{\psi_n,j} = M_{\psi_j}.
\]
So,
\begin{equation}\label{eqns0}
M_{\vp_n,j}^*P_{\cls_{\Phi,j}} P_{\cls_{\Psi,j}}M_{\psi_n,j} = M_{\vp_n,j}^* M_{\psi_n,j}  = M_{\vp_j}^*M_{\psi_j}.
\end{equation}
We now want to expand $M_{\vp_n,j}^*P_{\cls_{\Phi,j}} P_{\cls_{\Psi,j}}M_{\psi_n,j} $ as sum of operators. Note that
\begin{equation}\label{eqns1}
M_{\vp_n,j}^*P_{\cls_{\Phi,j}} P_{\cls_{\Psi,j}}M_{\psi_n,j} = M_{\vp_n,j}^* P_{\cls_{\Psi,j}} P_{\cls_{\Phi,j}}M_{\psi_n,j}  
= M_{\vp_n,j}^* \big( \sum_{i=1}^n P_{\psi_i,j} \big)P_{\cls_{\Phi},j} M_{\psi_n,j}.
\end{equation}
For simplicity, let us begin by expanding the last term of the right-most expression.  \vspace{-1mm}
\[
\begin{split}
&M_{\vp_{n,j}}^* M_{\psi_n,j} M_{\psi_{n,j}}^* P_{\cls_{\Phi,j}} M_{\psi_{n,j}}\\
&= M_{\vp_{n,j}}^* M_{\psi_{n,j}} M_{\psi_{n,j}}^* \big( \sum_{i=1}^n P_{\vp_i,j} \big) M_{\psi_{n,j}}\\
&= M_{\vp_{n,j}}^* M_{\psi_{n,j}} M_{\psi_{n,j}}^* \big( \sum_{i=1}^{n-1} P_{\vp_i,j} \big) M_{\psi_{n,j}} + M_{\vp_{n,j}}^* M_{\psi_{n,j}} M_{\psi_{n,j}}^* M_{\vp_{n,j}} M_{\vp_{n,j}}^*  M_{\psi_{n,j}}\\
&= M_{\vp_{n,j}}^* M_{\psi_{n,j}} M_{\psi_{n,j}}^* \big( \sum_{i=1}^{n-1} P_{\vp_{i,j}} \big) M_{\psi_{n,j}} + P_{n,j} P_{n,j}^* P_{n,j} \\
&= M_{\vp_{n,j}}^* M_{\psi_{n,j}} M_{\psi_{n,j}}^* \big( \sum_{i=1}^{n-1} M_{\vp_{i,j}}M_{\vp_{i,j}}^* \underset{t>i}{\overset{n}{\Pi}} (I_{H^2(\D^n)} - M_{\vp_{t,j}}M_{\vp_{t,j}}^*) \big) M_{\psi_{n,j}} + P_{n,j} P_{n,j}^* P_{n,j}\\
&= M_{\vp_{n,j}}^* M_{\psi_{n,j}} M_{\psi_{n,j}}^* (I_{H^2(\D^n)} - M_{\vp_{n,j}}M_{\vp_{n,j}}^*)  M_{\psi_{n,j}} \big( \sum_{i=1}^{n-1} M_{\vp_{i,j}}M_{\vp_{i,j}}^* \underset{t>i}{\overset{n-1}{\Pi}} (I_{H^2(\D^n)} - M_{\vp_{t,j}}M_{\vp_{t,j}}^*) \big) \\
&+ P_{n,j} P_{n,j}^* P_{n,j}\\
&= (P_{n,j}  - P_{n,j} P_{n,j}^* P_{n,j}) \big( \sum_{i=1}^{n-1} M_{\vp_{i,j}}M_{\vp_{i,j}}^* \underset{t>i}{\overset{n-1}{\Pi}} (I_{H^2(\D^n)} - M_{\vp_{t,j}}M_{\vp_{t,j}}^*) \big)  + (P_{n,j} P_{n,j}^* P_{n,j}) .
\end{split}
\]
Using the above identity, we get
\[
\begin{split}
&P_{n,j} - (P_{n,j} - P_{n,j} P_{n,j}^* P_{n,j}) \big( \sum_{i=1}^{n-1} M_{\vp_{i,j}}M_{\vp_{i,j}}^* \underset{t>i}{\overset{n-1}{\Pi}} (I_{H^2(\D^n)} - M_{\vp_{t,j}}M_{\vp_{t,j}}^*) \big)  - (P_{n,j} P_{n,j}^* P_{n,j})\\
&= (P_{n,j} -  P_{n,j} P_{n,j}^* P_{n,j}) [I_{H^2(\D^n)} - \big( \sum_{i=1}^{n-1} M_{\vp_{i,j}}M_{\vp_{i,j}}^* \underset{t>i}{\overset{n-1}{\Pi}} (I_{H^2(\D^n)} - M_{\vp_{t,j}}M_{\vp_{t,j}}^*) \big)]\\
&=(P_{n,j} -  P_{n,j} P_{n,j}^* P_{n,j}) \big(  \underset{i=1}{\overset{n-1}{\Pi}} (I_{H^2(\D^n)} - M_{\vp_{i,j}}M_{\vp_{i,j}}^*) \big).
\end{split}
\]
Now, using (\ref{eqns0}) we can deduce that
\begin{multline}\label{eqn1}
M_{\vp_{n,j}}^*P_{\cls_{\Phi,j}}P_{\cls_{\Psi,j}}M_{\psi_{n,j}} - M_{\vp_{n,j}}^*M_{\psi_{n,j}}M_{\psi_{n,j}}^*P_{\cls_{\Phi,j}}M_{\psi_{n,j}}\\
=(P_{n,j} -  P_{n,j} P_{n,j}^* P_{n,j}) \big(  \underset{i=1}{\overset{n-1}{\Pi}} (I_{H^2(\D^n)} - M_{\vp_{i,j}}M_{\vp_{i,j}}^*) \big).
\end{multline}
We will now consider the general term of the right-most expression in (\ref{eqns1}),. In particular, for any $ k \in \{1,\ldots,n-1\}$ we consider the following.
\begin{align*}
&M_{\vp_{n,j}}^* M_{\psi_{k,j}} M_{\psi_{k,j}}^* \big( \underset{l>k}{\overset{n}{\Pi}}  (I_{H^2(\D^n)} - M_{\psi_{l,j}} M_{\psi_{l,j}}^*) \big) P_{\cls_{\Phi,j}} M_{\psi_{n,j}}\\
&=M_{\vp_{n,j}}^* M_{\psi_{k,j}} M_{\psi_{k,j}}^* \big( \underset{l>k}{\overset{n}{\Pi}}  (I_{H^2(\D^n)} - M_{\psi_{l,j}} M_{\psi_{l,j}}^*) \big)   \big( \sum_{i=1}^{n-1} P_{\vp_{i,j}} \big) M_{\psi_{n,j}} \\
&+ M_{\vp_{n,j}}^* M_{\psi_{k,j}} M_{\psi_{k,j}}^* \big( \underset{l>k}{\overset{n}{\Pi}}  (I_{H^2(\D^n)} - M_{\psi_{l,j}} M_{\psi_{l,j}}^*) \big)  M_{\vp_{n,j}} M_{\vp_{n,j}}^* M_{\psi_{n,j}}\\
&=M_{\vp_{n,j}}^* M_{\psi_{k,j}} M_{\psi_{k,j}}^* \big( \underset{l>k}{\overset{n}{\Pi}}  (I_{H^2(\D^n)} - M_{\psi_{l,j}} M_{\psi_{l,j}}^*) \big)   \big( \sum_{i=1}^{n-1} P_{\vp_{i,j}} \big) M_{\psi_{n,j}} \\
&+M_{\psi_{k,j}} M_{\psi_{k,j}}^* \big( \underset{l>k}{\overset{n-1}{\Pi}}  (I_{H^2(\D^n)} - M_{\psi_{l,j}} M_{\psi_{l,j}}^*) \big)  M_{\vp_{n,j}}^* (I_{H^2(\D^n)} - M_{\psi_{n,j}} M_{\psi_{n,j}}^*) M_{\vp_{n,j}} M_{\vp_{n,j}}^* M_{\psi_{n,j}}\\
&= M_{\vp_{n,j}}^* M_{\psi_{k,j}} M_{\psi_{k,j}}^* \big( \underset{l>k}{\overset{n}{\Pi}}  (I_{H^2(\D^n)} - M_{\psi_{l,j}} M_{\psi_{l,j}}^*) \big)   \big( \sum_{i=1}^{n-1} P_{\vp_{i,j}} \big) M_{\psi_{n,j}} \\
&+M_{\psi_{k,j}} M_{\psi_{k,j}}^* \big( \underset{l>k}{\overset{n-1}{\Pi}}  (I_{H^2(\D^n)} - M_{\psi_{l,j}} M_{\psi_{l,j}}^*) \big)  (P_{n,j} - P_{n,j} P_{n,j}^* P_{n,j}).
\end{align*}
As before, let us compute the first term in the above sum separately.
\begin{align*}
&M_{\vp_{n,j}}^* M_{\psi_{k,j}} M_{\psi_{k,j}}^* \big( \underset{l>k}{\overset{n}{\Pi}}  (I_{H^2(\D^n)} - M_{\psi_{l,j}} M_{\psi_{l,j}}^*) \big)   \big( \sum_{i=1}^{n-1} P_{\vp_{i,j}} \big) M_{\psi_{n,j}} \\
&=M_{\psi_{k,j}} M_{\psi_{k,j}}^* \big( \underset{l>k}{\overset{n-1}{\Pi}}  (I_{H^2(\D^n)} - M_{\psi_{l,j}} M_{\psi_{l,j}}^*) \big) M_{\vp_{n,j}}^* (I_{H^2(\D^n)} - M_{\psi_{n,j}} M_{\psi_{n,j}}^*)  \big( \sum_{i=1}^{n-1} P_{\vp_{i,j}} \big) M_{\psi_{n,j}} \\
&= M_{\psi_{k,j}} M_{\psi_{k,j}}^* \big( \underset{l>k}{\overset{n-1}{\Pi}}  (I_{H^2(\D^n)} - M_{\psi_{l,j}} M_{\psi_{l,j}}^*) \big) M_{\vp_{n,j}}^* (I_{H^2(\D^n)} - M_{\psi_{n,j}} M_{\psi_{n,j}}^*)  \\
&\big( \sum_{i=1}^{n-1} M_{\vp_{i,j}}M_{\vp_{i,j}}^* \underset{t>i}{\overset{n}{\Pi}} (I_{H^2(\D^n)} - M_{\vp_{t,j}}M_{\vp_{t,j}}^*) \big)  M_{\psi_{n,j}} \\
&= M_{\psi_{k,j}} M_{\psi_{k,j}}^* \big( \underset{l>k}{\overset{n-1}{\Pi}}  (I_{H^2(\D^n)} - M_{\psi_{l,j}} M_{\psi_{l,j}}^*) \big) \big( \sum_{i=1}^{n-1} M_{\vp_{i,j}}M_{\vp_{i,j}}^* \underset{t>i}{\overset{n-1}{\Pi}} (I_{H^2(\D^n)} - M_{\vp_{t,j}}M_{\vp_{t,j}}^*) \big) M_{\vp_{n,j}}^*  \\
&(I_{H^2(\D^n)} - M_{\psi_{n,j}} M_{\psi_{n,j}}^*) 
(I_{H^2(\D^n)} - M_{\vp_{n,j}}M_{\vp_{n,j}}^*)M_{\psi_{n,j}} \\
&=  - M_{\psi_{k,j}} M_{\psi_{k,j}}^* \big( \underset{l>k}{\overset{n-1}{\Pi}}  (I_{H^2(\D^n)} - M_{\psi_{l,j}} M_{\psi_{l,j}}^*) \big) \big( \sum_{i=1}^{n-1} M_{\vp_{i,j}}M_{\vp_{i,j}}^* \underset{t>i}{\overset{n-1}{\Pi}} (I_{H^2(\D^n)} - M_{\vp_{t,j}}M_{\vp_{t,j}}^*) \big)\\
&(P_{n,j} - P_{n,j} P_{n,j}^* P_{n,j}).
\end{align*}
The important fact in the above computation is that we can extract the factor $(P_{n,j} - P_{n,j} P_{n,j}^* P_{n,j})$. This will be useful for drawing the main conclusion. Using the above expression we can now write the sum as the following.
\[
\begin{split}
&M_{\vp_{n,j}}^* M_{\psi_{k,j}} M_{\psi_{k,j}}^* \big( \underset{l>k}{\overset{n}{\Pi}}  (I_{H^2(\D^n)} - M_{\psi_{l,j}} M_{\psi_{l,j}}^*) \big)   \big( \sum_{i=1}^{n-1} P_{\vp_{i,j}} \big) M_{\psi_{n,j}} \\
&+M_{\psi_{k,j}} M_{\psi_{k,j}}^* \big( \underset{l>k}{\overset{n-1}{\Pi}}  (I_{H^2(\D^n)} - M_{\psi_{l,j}} M_{\psi_{l,j}}^*) \big)  (P_{n,j} - P_{n,j} P_{n,j}^* P_{n,j})\\
&= - M_{\psi_{k,j}} M_{\psi_{k,j}}^* \big( \underset{l>k}{\overset{n-1}{\Pi}}  (I_{H^2(\D^n)} - M_{\psi_{l,j}} M_{\psi_{l,j}}^*) \big) \big( \sum_{i=1}^{n-1} M_{\vp_{i,j}}M_{\vp_{i,j}}^* \underset{t>i}{\overset{n-1}{\Pi}} (I_{H^2(\D^n)} - M_{\vp_{t,j}}M_{\vp_{t,j}}^*) \big)\\
&(P_{n,j} - P_{n,j} P_{n,j}^* P_{n,j})\\ &+M_{\psi_{k,j}} M_{\psi_{k,j}}^* \big( \underset{l>k}{\overset{n-1}{\Pi}}  (I_{H^2(\D^n)} - M_{\psi_{l,j}} M_{\psi_{l,j}}^*) \big)  (P_{n,j} - P_{n,j} P_{n,j}^* P_{n,j})\\
&= (P_{n,j} - P_{n,j} P_{n,j}^* P_{n,j}) M_{\psi_{k,j}} M_{\psi_{k,j}}^* \big( \underset{l>k}{\overset{n-1}{\Pi}}  (I_{H^2(\D^n)} - M_{\psi_{l,j}} M_{\psi_{l,j}}^*) \big) \\
&[I_{H^2(\D^n)} - \big( \sum_{i=1}^{n-1} M_{\vp_{i,j}}M_{\vp_{i,j}}^* \underset{t>i}{\overset{n-1}{\Pi}} (I_{H^2(\D^n)} - M_{\vp_{t,j}}M_{\vp_{t,j}}^*) \big)]\\
&= (P_{n,j} - P_{n,j} P_{n,j}^* P_{n,j})  M_{\psi_{k,j}} M_{\psi_{k,j}}^* \big( \underset{l>k}{\overset{n-1}{\Pi}}  (I_{H^2(\D^n)} - M_{\psi_{l,j}} M_{\psi_{l,j}}^*) \big) \big(\underset{i=1}{\overset{n-1}{\Pi}} (I_{H^2(\D^n)} - M_{\vp_{i,j}}M_{\vp_{i,j}}^*) \big).
\end{split}
\]
This gives a simple expression of the original expression.
\begin{multline}\label{eqn2}
M_{\vp_{n,j}}^* M_{\psi_{k,j}} M_{\psi_{k,j}}^* \big( \underset{l>k}{\overset{n}{\Pi}}  (I_{H^2(\D^n)} - M_{\psi_{l,j}} M_{\psi_{l,j}}^*) \big) P_{\cls_{\Phi,j}} M_{\psi_{n,j}}\\
= (P_{n,j} - P_{n,j} P_{n,j}^* P_{n,j})  M_{\psi_{k,j}} M_{\psi_{k,j}}^* \big( \underset{l>k}{\overset{n-1}{\Pi}}  (I_{H^2(\D^n)} - M_{\psi_{l,j}} M_{\psi_{l,j}}^*) \big) \big(\underset{i=1}{\overset{n-1}{\Pi}} (I_{H^2(\D^n)} - M_{\vp_{i,j}}M_{\vp_{i,j}}^*) \big).
\end{multline}
Our main aim is to now consider the following difference.
\begin{align*}
&M_{\vp_{n,j}}^*P_{\cls_{\Phi,j}}P_{\cls_{\Psi,j}}M_{\psi_{n,j}} -  M_{\vp_{n,j}}^* P_{\cls_{\Psi,j}}P_{\cls_{\Phi,j}} M_{\psi_{n,j}}\\
&=  M_{\vp_{n,j}}^*P_{\cls_{\Phi,j}}P_{\cls_{\Psi,j}}M_{\psi_{n,j}}  - M_{\vp_{n,j}}^* \big( \sum_{k=1}^{n} M_{\psi_{k,j}}M_{\psi_{k,j}}^* \underset{l>k}{\overset{n}{\Pi}} (I_{H^2(\D^n)} - M_{\psi_{l,j}}M_{\psi_{l,j}}^*) P_{\cls_{\Phi,j}} M_{\psi_{n,j}}\\
&= M_{\vp_{n,j}}^*P_{\cls_{\Phi,j}}P_{\cls_{\Psi,j}}M_{\psi_{n,j}} - M_{\vp_{n,j}}^*M_{\psi_{n,j}}M_{\psi_{n,j}}^*P_{\cls_{\Phi,j}}M_{\psi_{n,j}}  \\
&- M_{\vp_{n,j}}^* \big( \sum_{k=1}^{n-1} M_{\psi_{k,j}}M_{\psi_{k,j}}^* \underset{l>k}{\overset{n}{\Pi}} (I_{H^2(\D^n)} - M_{\psi_{l,j}}M_{\psi_{l,j}}^*) P_{\cls_{\Phi,j}} M_{\psi_{n,j}}\\
&=  (P_{n,j} - P_{n,j} P_{n,j}^* P_{n,j}) \big(  \underset{i=1}{\overset{n-1}{\Pi}} (I_{H^2(\D^n)} - M_{\vp_{i,j}}M_{\vp_{i,j}}^*) \big) \\
&- M_{\vp_{n,j}}^* \big( \sum_{k=1}^{n-1} M_{\psi_{k,j}}M_{\psi_{k,j}}^* \underset{l>k}{\overset{n}{\Pi}} (I_{H^2(\D^n)} - M_{\psi_{l,j}}M_{\psi_{l,j}}^*) P_{\cls_{\Phi,j}} M_{\psi_{n,j}}.
\end{align*}
The last equality follows by using the condition (\ref{eqn1}).  Thus, the above difference now turns into the following.
\begin{align*}
&(P_{n,j} - P_{n,j} P_{n,j}^* P_{n,j}) \big(  \underset{i=1}{\overset{n-1}{\Pi}} (I_{H^2(\D^n)} - M_{\vp_{i,j}}M_{\vp_{i,j}}^*) \big) - M_{\vp_{n,j}}^* M_{\psi_{n-1,j}} M_{\psi_{n-1,j}}^* (I_{H^2(\D^n)} - M_{\psi_{n,j}} M_{\psi_{n,j}}^*) M_{\psi_{n,j}} \\
&- M_{\vp_{n,j}}^* \big( \sum_{k=1}^{n-2} M_{\psi_{k,j}}M_{\psi_{k,j}}^* \underset{l>k}{\overset{n}{\Pi}} (I_{H^2(\D^n)} - M_{\psi_{l,j}}M_{\psi_{l,j}}^*) P_{\cls_{\Phi,j}} M_{\psi_{n,j}}\\
&=   (P_{n,j} - P_{n,j} P_{n,j}^* P_{n,j}) \big(  \underset{i=1}{\overset{n-1}{\Pi}} (I_{H^2(\D^n)} - M_{\vp_{i,j}}M_{\vp_{i,j}}^*) \big) \\
&-  (P_{n,j} - P_{n,j} P_{n,j}^* P_{n,j}) M_{\psi_{n-1,j}} M_{\psi_{n-1,j}}^* \big(  \underset{i=1}{\overset{n-1}{\Pi}} (I_{H^2(\D^n)} - M_{\vp_{i,j}}M_{\vp_{i,j}}^*) \big)\\
 &- M_{\vp_{n,j}}^* \big( \sum_{k=1}^{n-2} M_{\psi_{k,j}}M_{\psi_{k,j}}^* \underset{l>k}{\overset{n}{\Pi}} (I_{H^2(\D^n)} - M_{\psi_{l,j}}M_{\psi_{l,j}}^*) P_{\cls_{\Phi,j}} M_{\psi_{n,j}}\\
&=  (P_{n,j} - P_{n,j} P_{n,j}^* P_{n,j}) (I_{H^2(\D^n)} - M_{\psi_{n-1,j}} M_{\psi_{n-1,j}}^*) \big(  \underset{i=1}{\overset{n-1}{\Pi}} (I_{H^2(\D^n)} - M_{\vp_{i,j}}M_{\vp_{i,j}}^*) \big)\\
&- M_{\vp_{n,j}}^* \big( \sum_{k=1}^{n-2} M_{\psi_{k,j}}M_{\psi_{k,j}}^* \underset{l>k}{\overset{n}{\Pi}} (I_{H^2(\D^n)} - M_{\psi_{l,j}}M_{\psi_{l,j}}^*) P_{\cls_{\Phi,j}} M_{\psi_{n,j}}.
\end{align*}
The last but one equality follows by using the condition (\ref{eqn2}). Continuing in this manner  for another $n-2$-times while using (\ref{eqn2}) gives
\begin{align*}
&M_{\vp_{n,j}}^*P_{\cls_{\Phi,j}}P_{\cls_{\Psi,j}}M_{\psi_{n,j}} -  M_{\vp_{n,j}}^* P_{\cls_{\Psi,j}}P_{\cls_{\Phi,j}} M_{\psi_{n,j}}\\
&=(P_{n,j} - P_{n,j} P_{n,j}^* P_{n,j})  \big(  \underset{i=1}{\overset{n-1}{\Pi}}(I_{H^2(\D^n)} - M_{\psi_{i,j}} M_{\psi_{i,j}}^*) \big)  \big(  \underset{i=1}{\overset{n-1}{\Pi}} (I_{H^2(\D^n)} - M_{\vp_{i,j}}M_{\vp_{i,j}}^*) \big)
\end{align*}
By our assumption, we have $[P_{\cls_{\Phi}},P_{\cls_{\Psi}}]=0$, hence, $$M_{\vp_{n,j}}^*P_{\cls_{\Phi,j}}P_{\cls_{\Psi,j}}M_{\psi_{n,j}} = M_{\vp_{n,j}}^* P_{\cls_{\Psi,j}}P_{\cls_{\Phi,j}} M_{\psi_{n,j}}.$$ Therefore, from the above identity we get
\[
(P_{n,j} - P_{n,j} P_{n,j}^* P_{n,j})  \big(  \underset{i=1}{\overset{n-1}{\Pi}}(I_{H^2(\D^n)} - M_{\psi_{i,j}} M_{\psi_{i,j}}^*) \big)  \big(  \underset{i=1}{\overset{n-1}{\Pi}} (I_{H^2(\D^n)} - M_{\vp_{i,j}}M_{\vp_{i,j}}^*) \big) = 0.
\]
Since $ (P_{n,j} - P_{n,j} P_{n,j}^* P_{n,j})$ and $  \big(  \underset{i=1}{\overset{n-1}{\Pi}}(I_{H^2(\D^n)} - M_{\psi_{i,j}} M_{\psi_{i,j}}^*) \big)  \big(  \underset{i=1}{\overset{n-1}{\Pi}} (I_{H^2(\D^n)} - M_{\vp_{i,j}}M_{\vp_{i,j}}^*) \big)$ depend on disjoint set of variables, 
\[
 (P_{n,j} - P_{n,j} P_{n,j}^* P_{n,j})  \big(  \underset{i=1}{\overset{n-1}{\Pi}}(I_{H^2(\D^n)} - M_{\psi_{i,j}} M_{\psi_{i,j}}^*) \big)  \big(  \underset{i=1}{\overset{n-1}{\Pi}} (I_{H^2(\D^n)} - M_{\vp_{i,j}}M_{\vp_{i,j}}^*) \big) = 0
\]
if and only if $ (P_{n,j} - P_{n,j} P_{n,j}^* P_{n,j}) =0$ or $\underset{i=1}{\overset{n-1}{\Pi}}(I_{H^2(\D^n)} - M_{\psi_{i,j}} M_{\psi_{i,j}}^*)   \underset{i=1}{\overset{n-1}{\Pi}} (I_{H^2(\D^n)} - M_{\vp_{i,j}}M_{\vp_{i,j}}^*) =0$, or both. Since $\vp_j$ and $\psi_j$ are non-constant functions, from Lemma \ref{non-zero},  we can conclude that
\[
 \underset{i=1}{\overset{n-1}{\Pi}}(I_{H^2(\D^n)} - M_{\psi_{i,j}} M_{\psi_{i,j}}^*)   (I_{H^2(\D^n)} - M_{\vp_{i,j}}M_{\vp_{i,j}}^*)  = \underset{i=1; i \neq j}{\overset{n}{\Pi}}(I_{H^2(\D^n)} - M_{\psi_i} M_{\psi_i}^*) (I_{H^2(\D^n)} - M_{\vp_i}M_{\vp_i}^*) \neq 0.
\]
Therefore, the only possibility is $P_{n,j} = P_{n,j} P_{n,j}^* P_{n,j}$. In other words, $P_{n,j}  = M_{\vp_{n,j}}^* M_{\psi_{n,j}} = M_{\vp_{j}}^* M_{\psi_{j}}$ is a partial isometry. From \cite[Theorem 2.2]{DPS}, it must follow that either
\[
\vp_j \text{ divides } \psi_j \text{ or } \psi_j \text{ divides } \vp_j.
\]
Since $j$ was an arbitrary element in $\{1,\ldots,n\}$, we get
\[
\vp_j \text{ divides } \psi_j \text{ or } \psi_j \text{ divides } \vp_j,
\]
for all $j \in \{1,\ldots,n\}$. Conversely, suppose the inner functions $\psi_j$ and $\vp_j$ for all $j\in \{1,\dots, n\}$ satisfy either $\vp_j|\psi_j$ or $\psi_j|\vp_j$. This implies that
\[
M_{\vp_j} M_{\vp_j}^* M_{\psi_j} M_{\psi_j}^* \text{ is a projection},
\]
and therefore,
\[
[M_{\vp_j} M_{\vp_j}^*, M_{\psi_i} M_{\psi_i}^*] = 0,
\]
for all $i,j \in \{1,\ldots,n\}$. Using this, we can immediately conclude that,
\begin{align*}
P_{\cls_{\Phi}} P_{\cls_{\Psi}} &= \sum_{l=1}^n P_{\vp_l} \sum_{k=1}^n P_{\vp_k}\\
&=  \big( \sum_{l=1}^{n} M_{\vp_{l}}M_{\vp_{l}}^* \underset{j>l}{\overset{n}{\Pi}} (I_{H^2(\D^n)} - M_{\vp_j}M_{\vp_j}^*) \big) \big( \sum_{k=1}^{n} M_{\psi_{k}}M_{\psi_{k}}^* \underset{i>k}{\overset{n}{\Pi}} (I_{H^2(\D^n)} - M_{\psi_i}M_{\psi_i}^*) \big)\\
&=   \big( \sum_{k=1}^{n} M_{\psi_{k}}M_{\psi_{k}}^* \underset{i>k}{\overset{n}{\Pi}} (I_{H^2(\D^n)} - M_{\psi_i}M_{\psi_i}^*) \big) \big( \sum_{l=1}^{n} M_{\vp_{l}}M_{\vp_{l}}^* \underset{j>l}{\overset{n}{\Pi}} (I_{H^2(\D^n)} - M_{\vp_j}M_{\vp_j}^*) \big)\\
&=  P_{\cls_{\Psi}} P_{\cls_{\Phi}}
\end{align*}
This completes the proof.
\end{proof}
We can use the above characterization to prove the result for submodules $\cls_{\Phi_{\Lambda}}$ and $S_{\cls_{\Psi_{\Gamma}}}$.
\begin{proof}[Proof for Theorem \ref{comm_projn}]
For this proof, let us extend the submodules into the setting of the previous result. Consider
\begin{multline}
\cls_{1} := \sum_{j \in \Lambda \cap \Gamma} \vp_j(z_j) H^2(\D^n) +  \sum_{\lambda \in \Lambda \setminus \Lambda \cap \Gamma} \vp_\lambda(z_\lambda) H^2(\D^n) \\+ \sum_{t \in \Gamma \setminus \Lambda \cap \Gamma} \psi_t(z_t) H^2(\D^n) +  \sum_{k \in \{1,\ldots,n\} \setminus \Lambda \cup \Gamma } z_k H^2(\D^n),
\end{multline}
and similarly,
\begin{multline}
\cls_{2} := \sum_{j \in \Lambda \cap \Gamma} \psi_j(z_j) H^2(\D^n) +  \sum_{\lambda \in \Lambda \setminus \Lambda \cap \Gamma} \vp_\lambda(z_\lambda) H^2(\D^n) \\+ \sum_{t \in \Gamma \setminus \Lambda \cap \Gamma} \psi_t(z_t) H^2(\D^n) +  \sum_{k \in \{1,\ldots,n\} \setminus \Lambda \cup \Gamma } z_k H^2(\D^n).
\end{multline}
From the structures of the above submodules, it is evident that 
\begin{multline}
P_{\cls_1}^{\perp} = \Big( \underset{j \in \Lambda \cap \Gamma}{\Pi} (I_{H^2(\D^n)} - M_{\vp_j} M_{\vp_j}^*)\Big)  \Big( \underset{\lambda \in \Lambda \setminus \Lambda \cap \Gamma}{\Pi} (I_{H^2(\D^n)} - M_{\vp_\lambda} M_{\vp_\lambda}^*)\Big) \\
 \Big( \underset{t \in \Gamma \setminus \Lambda \cap \Gamma}{\Pi} (I_{H^2(\D^n)} - M_{\psi_t} M_{\psi_t}^*)\Big)  \Big( \underset{k \in \{1,\ldots,n\} \setminus \Lambda \cap \Gamma}{\Pi} (I_{H^2(\D^n)} - M_{z_k} M_{z_k}^*)\Big),
\end{multline}
and
\begin{multline}
P_{\cls_2}^{\perp} = \Big( \underset{j \in \Lambda \cap \Gamma}{\Pi} (I_{H^2(\D^n)} - M_{\psi_j} M_{\psi_j}^*)\Big) \Big( \underset{\lambda \in \Lambda \setminus \Lambda \cap \Gamma}{\Pi} (I_{H^2(\D^n)} - M_{\vp_\lambda} M_{\vp_\lambda}^*)\Big) \\
 \Big( \underset{t \in \Gamma \setminus \Lambda \cap \Gamma}{\Pi} (I_{H^2(\D^n)} - M_{\psi_t} M_{\psi_t}^*)\Big) \Big( \underset{k \in \{1,\ldots,n\} \setminus \Lambda \cap \Gamma}{\Pi} (I_{H^2(\D^n)} - M_{z_k} M_{z_k}^*)\Big),
\end{multline}
It is clear that $[P_{\cls_1}, P_{\cls_2}]=0$ if and only if $[P_{\cls_1}^{\perp}, P_{\cls_2}^{\perp}]=0$. Now,
\begin{multline}
[P_{\cls_1}^{\perp}, P_{\cls_2}^{\perp}] = \big( [ \underset{j \in \Lambda \cap \Gamma}{\Pi} (I_{H^2(\D^n)} - M_{\vp_j} M_{\vp_j}^*), \underset{j \in \Lambda \cap \Gamma}{\Pi} (I_{H^2(\D^n)} - M_{\psi_j} M_{\psi_j}^*)] \Big)\\  \Big( \underset{\lambda \in \Lambda \setminus \Lambda \cap \Gamma}{\Pi} (I_{H^2(\D^n)} - M_{\vp_\lambda} M_{\vp_\lambda}^*)\Big) 
 \Big( \underset{t \in \Gamma \setminus \Lambda \cap \Gamma}{\Pi} (I_{H^2(\D^n)} - M_{\psi_t} M_{\psi_t}^*)\Big) \Big( \underset{k \in \{1,\ldots,n\} \setminus \Lambda \cap \Gamma}{\Pi} (I_{H^2(\D^n)} - M_{z_k} M_{z_k}^*)\Big).
\end{multline}
The terms inside the different parentheses depend on a disjoint set of variables also, the terms inside all but the first bracket cannot be zero. Hence, $[P_{\cls_1}^{\perp}, P_{\cls_2}^{\perp}]=0$ if and only if $[ \underset{j \in \Lambda \cap \Gamma}{\Pi} (I_{H^2(\D^n)} - M_{\vp_j} M_{\vp_j}^*), \underset{j \in \Lambda \cap \Gamma}{\Pi} (I_{H^2(\D^n)} - M_{\psi_j} M_{\psi_j}^*)]=0$.  Now, $\cls_{\Phi_{\Lambda}}, \cls_{\Psi_{\Gamma}}$ are the submodule corresponding to the quotient modules $\underset{j \in \Lambda \cap \Gamma}{\Pi} (I_{H^2(\D^n)} - M_{\vp_j} M_{\vp_j}^*)$ and $\underset{j \in \Lambda \cap \Gamma}{\Pi} (I_{H^2(\D^n)} - M_{\psi_j} M_{\psi_j}^*)$, respectively. Thus, $[P_{\cls_{\Phi_{\Lambda}}},  P_{\cls_{\Psi_{\Gamma}}}]=0$ if and only if $[P_{\cls_1}^{\perp}, P_{\cls_2}^{\perp}]=0$, which is again equivalent to $[P_{\cls_1}, P_{\cls_2}]=0$. From the above Theorem \ref{main2}, it is clear that $P_{\cls_1}$ commutes with $P_{\cls_2}$ if and only if either $\vp_j|\psi_j$ or $\psi_j|\vp_j$ for all $j\in \Lambda \cap \Gamma$. This completes the proof.
\end{proof}

We are now ready to apply the above result to answer Douglas's question on the product of orthogonal projections onto quotient modules (corresponding to the above submodules). Let us first highlight that we need an additional assumption,
\[
\Lambda \cup \Gamma = \{1,\ldots,n\}.
\]
Suppose we do not have this property, and for example consider $\Lambda = \{1\}, \Gamma = \{2\} \subseteq \{1,2,3\}$. Then the quotient modules will be of the following form
\[
\clq_{\Phi_{\Lambda}} = \clq_{\vp_1} \otimes H^2(\D) \otimes H^2(\D); \quad \clq_{\Psi_{\Lambda}} =  H^2(\D) \otimes H^2(\D) \otimes \clq_{\psi_3},
\]
for certain inner functions $\vp_1(z_1) \in H_{z_1}^{\infty}(\D)$ and $\psi_3(z_3) \in H_{z_3}^{\infty}(\D)$. Now even if we consider $\vp_1, \psi_3$ to be finite Blaschke products, the product of projections $P_{\clq_{\Phi_{\Lambda}}} P_{\Psi_{\Gamma}}$ will never be a finite-rank projection.  This justifies the additional assumption.
\begin{proof}[Proof of Theorem \ref{finite-rank}]
Let us begin by noting that
\[
P_{\clq_{\Phi_{\Lambda}}}P_{\clq_{\Psi_{\Gamma}}} = \big(\underset{\lambda \in \Lambda}{\Pi} (I_{H^2(\D^n)} - M_{\vp_\lambda}M_{\vp_\lambda}^*)\big) \big(\underset{t \in \Gamma}{\Pi} (I_{H^2(\D^n)} - M_{\psi_t}M_{\psi_t}^*)\big).
\]
It is a straightforward observation that $P_{\clq_{\Phi_{\Lambda}}} P_{\clq_{\Psi_{\Gamma}}}$ is a projection if and only if $P_{\cls_{\Phi_{\Lambda}}} P_{\cls_{\Psi_{\Gamma}}}$ is a projection, which is further equivalent to the condition that $[P_{\cls_{\Phi_{\Lambda}}}, P_{\cls_{\Psi_{\Gamma}}}] = 0$. From Theorem \ref{comm_projn}, it follows that for each $j \in \Lambda \cap \Gamma$
\[
\text{either } \vp_j \text{ divides } \psi_j \text{ or } \psi_j \text{ divides } \vp_j,
\]
in other words, we get 
\[
(I_{H^2(\D^n)} - M_{\vp_j}M_{\vp_j}^*)  (I_{H^2(\D^n)} - M_{\psi_j}M_{\psi_j}^*) = (I_{H^2(\D^n)} - M_{\vp_j}M_{\vp_j}^*)
\]
or,
\[
(I_{H^2(\D^n)} - M_{\vp_j}M_{\vp_j}^*)  (I_{H^2(\D^n)} - M_{\psi_j}M_{\psi_j}^*) = (I_{H^2(\D^n)} - M_{\psi_j}M_{\psi_j}^*),
\]
respectively.  Let $A:= \{j \in \Lambda \cap \Gamma: \vp_j \text{ divides } \psi_j \}$, and $B:= \{j \in  \Lambda \cap \Gamma: \psi_j \text{ divides } \vp_j \}$. From the above discussion it follows that if $P_{\clq_{\Phi_{\Lambda}}} P_{\clq_{\Psi_{\Gamma}}}$ is a projection, then
\begin{multline}\label{prod}
P_{\clq_{\Phi_{\Lambda}}} P_{\clq_{\Psi_{\Gamma}}} = \underset{i \in A}{{\Pi}} (I_{H^2(\D^n)} - M_{\vp_i}M_{\vp_i}^*) \underset{j \in B}{{\Pi}} (I_{H^2(\D^n)} - M_{\psi_j}M_{\psi_j}^*)\\ \underset{\lambda \in \Lambda \setminus \Lambda \cap \Gamma}{{\Pi}} (I_{H^2(\D^n)} - M_{\vp_\lambda}M_{\vp_\lambda}^*) 
 \underset{t \in \Gamma \setminus \Lambda \cap \Gamma}{{\Pi}} (I_{H^2(\D^n)} - M_{\psi_t}M_{\psi_t}^*).
\end{multline}
Furthermore, $P_{\clq_{\cls_{\Phi_{\Lambda}}}} P_{\clq_{\cls_{\Psi_{\Gamma}}}}$ is finite-rank if and only if the individual components of the product in (\ref{prod}) are finite-rank. It is clear that for any $\lambda \in \Lambda$, the projections $(I_{H^2(\D^n)} - M_{\vp_\lambda}M_{\vp_\lambda}^*)$, or $(I_{H^2(\D^n)} - M_{\psi_\lambda}M_{\psi_\lambda}^*)$ is finite-rank if and only $\vp_\lambda$ is a finite  Blaschke product, or $\psi_\lambda$ is a finite  Blaschke product, respectively. This completes the proof.
\end{proof}

In the case of both $\Lambda = \Gamma = \{1,\ldots,n\}$, and 
\[
\cls_{\Phi} = \sum_{i=1}^n \vp_i(z_i) H^2(\D^n); \quad \cls_{\Psi} = \sum_{j=1}^n \psi_j(z_j) H^2(\D^n)
\]
we have the following result.
\begin{cor}
Let $\cls_{\Phi}$ and $\cls_{\Psi}$ be submodules of $H^2(\D^n)$, and $\clq_{\Phi}, \clq_{\psi}$ be the corresponding quotient modules. Then the following are equivalent
\begin{enumerate}
\item[(i)] $P=P_{\clq_{\Phi}}P_{\clq_{\psi}}$ is a finite rank projection,
\item[(ii)] for all $j \in \{1,\ldots,n\}$ any one of the following conditions hold
\begin{enumerate}
\item[(a)] $\vp_j \text{ divides } \psi_j$ and $\vp_j$ is a finite Blaschke product.
\item[(a)] $\psi_j \text{ divides } \vp_j$ and $\psi_j$ is a finite Blaschke product.
\end{enumerate}
\end{enumerate}
\end{cor}

\smallskip

\noindent\textsf{Acknowledgement:} The authors would like to thank E.K. Narayanan for his suggestions. The first-named author is supported in part by the NBHM Postdoctoral fellowship 0204/10(21)/2023\\ /R\&D-II/2795. The second-named author is supported by the Department of Science and Technology via the INSPIRE faculty fellowship IFA19-MA141.

\end{document}